\documentclass[reqno,10pt]{amsart}
\usepackage{ esint }
\usepackage{overpic}
\usepackage{amssymb}
\usepackage{amsmath}
\usepackage{amsthm}
\usepackage{mathrsfs}
 \usepackage{tikz}
\usetikzlibrary{shapes,snakes,spy,calc,shapes.geometric}

\usepackage{pgf}
\usepackage{color}
\usepackage{graphicx}   
\usepackage{hyperref}

\usepackage{amsmath}
  \usepackage{paralist}
\usepackage{graphicx}  
\usepackage{psfrag}
\usepackage{comment}
\usepackage{esvect}

\newtheorem{theorem}{Theorem}[section]
\newtheorem{corollary}[theorem]{Corollary}

\newtheorem{lemma}[theorem]{Lemma}
\newtheorem{proposition}[theorem]{Proposition}

\theoremstyle{definition}

\newtheorem{remark}[theorem]{Remark}

\newcommand{\R}{\mathbb{R}}
\newcommand{\N}{\mathbb{N}}

\newcommand{\eps}{\varepsilon}

\def\diam{{\rm diam}}

\def\dist{{\rm dist}}

\newcommand{\BBB}{\color{black}}

\newcommand{\EEE}{\color{black}}

\usepackage{latexsym,amsfonts}
\usepackage{mathrsfs}
\usepackage{centernot}
\usepackage{paralist}

\usepackage{eso-pic}  
\usepackage{pifont}
\usepackage{fixmath}

\numberwithin{equation}{section}

\begin{document}

\title[Functionals defined on piecewise rigid functions]{Functionals defined on piecewise rigid functions: \\ Integral representation and $\Gamma$-convergence}



\author{Manuel Friedrich}
\address[Manuel Friedrich]{Applied Mathematics,  
Universit\"{a}t M\"{u}nster, Einsteinstr. 62, D-48149 M\"{u}nster, Germany}
\email{manuel.friedrich@uni-muenster.de}
\urladdr{https://www.uni-muenster.de/AMM/Friedrich/index.shtml}

\author{Francesco Solombrino}
\address[Francesco Solombrino]{Dip. Mat. Appl. ``Renato Caccioppoli'', Univ. Napoli ``Federico II'', Via Cintia, Monte S. Angelo
80126 Napoli, Italy}
\email{francesco.solombrino@unina.it}
\urladdr{http://www.docenti.unina.it/francesco.solombrino}

\begin{abstract}
 We analyse integral representation and $\Gamma$-convergence properties of functionals defined on \emph{piecewise rigid functions}, i.e., functions which are piecewise affine on a Caccioppoli partition \BBB where the \EEE derivative in each component is constant and  lies in a  set  without rank-one connections. Such functionals \BBB account for interfacial energies in the variational modeling of materials which locally show a rigid behavior. \EEE Our results are  based on  localization techniques for $\Gamma$-convergence and a careful adaption of the global method for  relaxation    \cite{BFLM, BDM}  to this new setting, under rather general assumptions. They constitute a first step  towards the investigation of   lower semicontinuity, relaxation, and homogenization  for free-discontinuity problems in spaces of (generalized) functions of bounded deformation.
\end{abstract}

\maketitle

\section{Introduction}

Many  problems in  materials science, physics, computer science, and other fields  involve the minimization of surface energies for configurations which represent partitions of the domain into regions of finite perimeter.  Among the vast body of literature, we only mention examples in the direction of liquid crystals \cite{Ericksen}, phase transition problems in immiscible fluids \cite{baldo, Modica, Morgan}, fracture mechanics \cite{AmbrosioBraides3}, image segmentation \cite{MumSha},  spin-like lattice systems \cite{ABC, ACR, BraidesCicalese}, or  polycrystalline structures \cite{Caraballo1, deluca}, and refer to the references cited therein. 

In the framework of the calculus of variations, these phenomena can be formulated by means of integral functionals defined on \emph{Caccioppoli partitions} or \emph{piecewise constant functions} on such partitions, see \cite[Section 4.4]{Ambrosio-Fusco-Pallara:2000}  or Section \ref{sec: cacc} below for their definition.  Problems of this kind have first been  studied in the seminal work by {\sc Almgren} \cite{Almgren}. Later, {\sc Ambrosio and Braides} \cite{AmbrosioBraides, AmbrosioBraides2} carried out a comprehensive analysis by developing a theory of integral representation, compactness, $\Gamma$-convergence, and relaxation. They also addressed the problem of lower semicontinuity which has been further developed over the last years, see, e.g.,   \cite[Section 5.3]{Ambrosio-Fusco-Pallara:2000} or \cite{Caraballo1, Caraballo2, Caraballo3}. Recent advances dealing with density and continuity results \cite{BCG-part, matthias} witness that the study of this class of functionals is of ongoing interest.

Understanding the properties of Caccioppoli partitions is also a mainstay in the analysis of free-discontinuity problems \cite{Ambrosio-Fusco-Pallara:2000, DeGiorgi-Ambrosio:1988} defined on  \emph{special functions of bounded variation} ($SBV$) (see \cite[Section 4]{Ambrosio-Fusco-Pallara:2000}). Indeed, in this context, the study of lower semicontinuity conditions \cite{Amb, Ambrosio:90-2}, the derivation of integral representation formulas \cite{BFLM, BDM, BraidesPiat:96, Braides-Defranceschi}, or compactness properties \cite{Manuel} can often be  reduced to corresponding problems on partitions. In a similar fashion, homogenization and $\Gamma$-convergence  for free-discontinuity problems \cite{Braides-Defranceschi, Caterina, Caterina2, Giacomini-Ponsiglione:2006},  their approximation \cite{Annika, BellettiniCoscia, Braides-approx, matthias2}, as well as results on the existence of quasistatic evolutions \cite{Francfort-Larsen:2003, Giacomini-Ponsiglione:2006}   rely fundamentally on the decoupling of bulk and surface effects, for which a profound understanding of energies defined on piecewise constant functions is necessary.

In the present paper we are interested in analogous problems for functionals defined on \emph{piecewise rigid functions}, i.e., functions which are piecewise affine on a Caccioppoli partition \BBB where the \EEE derivative in each component is constant and  lies in a set $L$ without rank-one connections \cite{ball.james}. Our standard examples are the rotations $L= SO(d)$ and the space $L= \R^{d\times d}_{\rm skew}$ of skew symmetric matrices. In the application  of materials science, particularly in fracture mechanics, piecewise rigid functions  for $L= SO(d)$  can be interpreted as the configurations which   may exhibit  cracks along surfaces but do not store nonlinear elastic energy. In fact, in \cite{Chambolle-Giacomini-Ponsiglione:2007}, a remarkable \emph{piecewise rigidity} result has been proved showing that the set of these  functions coincides 
with the (seemingly larger) set of functions $u \in SBV$ with approximate gradient $\nabla u \in SO(d)$ almost everywhere. An analogous result holds in the geometrically linear setting $L= \R^{d\times d}_{\rm skew}$ for functions in the space $(G)SBD$ of \emph{(generalized) special functions of bounded deformation}, introduced in \cite{ACD, DM}. \BBB In the context of fracture mechanics, these results imply that a deformation of a cracked hyperelastic (respectively, linearly elastic) body does not store elastic energy if and only if it is piecewise rigid. Thus, interfacial energies of materials which show locally rigid behavior in different regions of the body can be naturally modeled by functionals defined on piewise rigid functions. \EEE

\BBB However, our primary purpose comes from the study of  free-discontinuity problems defined on the space $GSBD^p$, see  \cite{DM},  which has obtained steadily  increasing attention over the last years, cf., e.g., \cite{Chambolle-Conti-Francfort:2014, Iu3, Crismale2, Crismale,  Conti-Iurlano:15, Conti-Focardi-Iurlano:15, Friedrich:15-2, Friedrich:15-3, Friedrich:15-4, FriedrichSolombrino}. We have indeed already mentioned before how the analysis of partition problems has proved to be a relevant tool in the study of free-discontinuity problems on $SBV$. When coming to similar problems on $GSBD^p$, where only a control on the symmetrized gradient of the competitors is available, a larger space than piecewise constant functions must be taken into account in order, e.g., to provide lower semicontinuity conditions for surface integrands, or representation formulas for $\Gamma$-limits, and, in general, to deal with the issues that we mentioned above in the $SBV$ context. In our opinion, it is quite natural to expect that the understanding of energies defined on piecewise rigid functions for $L = \R^{d\times d}_{\rm skew}$ represents a significant first step (or maybe even the building block) of such a research program. \EEE

In this first paper on this topic, we investigate integral representation and $\Gamma$-convergence for functionals defined on piecewise rigid functions. Lower semicontinuity, homogenization, and relaxation will be carried out in a forthcoming paper. We now proceed by describing our setting in more detail. 

Let $L\subset \mathbb{R}^{d\times d}$ be a closed set of \emph{rigid matrices}  not  satisfying the \emph{Hadamard compatibility condition}   (equivalently,  having no rank-one connections between each other, see \cite{ball.james}),  for which  a locally  Bilipschitz parametrization exists, see \eqref{eq: repr0.0} below for details. \BBB The condition of no rank-one connections is needed to ensure that functions exhibit discontinuities along the interface of two components with different constant derivative in $L$. This  rules out the formation of laminates. The local Bilipschitz parametrization allows us to treat the matrices $L$ as a subset of a linear space instead of a manifold, cf.\ the case $L= SO(d)$. \EEE For $\Omega \subset \R^d$ open and bounded, we denote by $PR_L(\Omega)$ the set of piecewise rigid functions $u$, i.e., 
\begin{align}\label{eq: basic functions} 
u(x) = \sum\nolimits_{j\in \N} (Q_j\, x + b_j) \chi_{P_j}(x),
\end{align}
 where $(P_j)_{j\in\N}$ is a Caccioppoli partition of $\Omega$,  $Q_j \in L$, and $b_j \in \R^d$  for all $j \in \N$.  For open subsets $A \subset \Omega$, we consider functionals $\mathcal{F}(\cdot,A): PR_L(\Omega) \to [0,\infty)$ of the form
\begin{align}\label{eq: basic problem}
\mathcal{F}(u,A) = \int_{J_u\cap  A} f(x,[u](x),\nu_u(x))\, d\mathcal{H}^{d-1}(x),
\end{align}
 where by $[u]$ and $\nu_u$ we denote the jump height and a normal to the jump (i.e., a normal to the interface), respectively, and $f$ represents an interfacial energy   density  which may additionally depend on the material point $x$. 
 
 We are interested in the problem if, for a sequence of functionals $(\mathcal{F}_n)_n$ with densities $(f_n)_n$, an effective limiting problem exists in the sense of variational convergence ($\Gamma$-convergence). Then, it is a natural question if also the $\Gamma$-limit is  of the form  \eqref{eq: basic problem}.  In this context, a standard procedure relies on localization techniques for $\Gamma$-convergence  (see \cite{DalMaso:93}), i.e., passing to a $\Gamma$-limit $\mathcal{F}(\cdot,A)$  of the sequence $\mathcal{F}_n(\cdot,A)$ for every open $A \subset \Omega$. Afterwards, one shows that under certain conditions for  $\mathcal{F}(\cdot,A)$, including  suitable semicontinuity, locality, and measure theoretic properties, there exists an integral representation   for  $\mathcal{F}(\cdot,A)$ in the sense of \eqref{eq: basic problem}.

An approach in this spirit has been performed in \cite{AmbrosioBraides} for finitely valued piecewise constant functions, i.e., for functions of the form \eqref{eq: basic functions} with $Q_j = 0$ and $b_j \in K$ for a finite set $K \subset \R^d$. $\Gamma$-convergence and integral representation are guaranteed under the natural \emph{growth conditions} $0<\alpha \le  f_n(x,\xi,\nu) \le \beta$ and a \emph{uniform continuity condition}  $x \mapsto f_n(x,\xi,\nu)$   along the sequence of densities $(f_n)_n$, which are maintained in the $\Gamma$-limit.  Later, for the problem of integral representation (but not for $\Gamma$-convergence), the continuity assumption  in $x$ has been dropped in \cite[Theorem 3]{BFLM}, and,  under a continuity condition  $\xi \mapsto f(x,\xi,\nu)$,   the result has been generalized to  $K = \R^d$ in \cite[Theorem 3.2]{BraidesPiat:96}. In the present paper,  under similar growth and continuity conditions,  we derive analogous results for $PR_L(\Omega)$ in place of piecewise constant functions. As a byproduct, choosing $L = \lbrace 0 \rbrace$, we also generalize the above mentioned $\Gamma$-convergence results to the case $K = \R^d$.

We now give a more thorough outline of our proof strategy and  provide a comparison  with \cite{AmbrosioBraides}.  First, concerning integral representation, we follow the global method for relaxation developed in \cite{BFLM, BDM}, which essentially consists in comparing  asymptotic Dirichlet problems on small  balls  with different boundary data depending on the local properties of $u$. For $\Gamma$-convergence, we apply the localization techniques described above, see e.g.\ \cite{Braides-Defranceschi:98, DalMaso:93}. 

For both methods, the key ingredient  is a construction for  \emph{joining} two functions $u,v \in PR_L(\Omega)$, which is usually called the \emph{fundamental estimate}. Typically, this is achieved by means of a cut-off construction of the form $w := u \varphi + (1 - \varphi) \BBB v \EEE $ for some smooth $\varphi$ with $0 \le \varphi \le 1$. In the present context, however, a crucial problem has to be faced since in general $w$ is not in $PR_L(\Omega)$. In the case of piecewise constant functions, this issue was solved in \cite{AmbrosioBraides} by using the coarea formula in $BV$, see \cite[Lemma 4.4]{AmbrosioBraides},  which allows to approximate $w$ by a piecewise constant function $\tilde{w}$. Geometrically, the joining of $u$ and $v$ to $\tilde{w}$ consists in modifying the partitions and adding additional interface whose length is controlled by $d(u,v)$, where $d(u,v)$ is a suitable metric on the space.  The same strategy cannot be pursued in the present context: e.g., when $L=\R^{d\times d}_{\rm skew}$, we have $PR_L(\Omega) \subset SBD(\Omega)$, where no analog of the coarea  formula  is known to hold. (We refer to \cite{Friedrich:15-4} for more details in that direction.) 

Our main trick is the following: we apply the coarea formula twice, once for the functions themselves and once for their derivatives. Roughly speaking, this allows to join two functions $u,v \in PR_L(\Omega)$ by adding additional interface whose length is controlled in terms of $d(u,v)$ and $d_\nabla(\nabla u,\nabla v)$ for suitable metrics $d$ and $d_\nabla$. Unfortunately, the metric $d_\nabla$ is too strong and not compatible with the available compactness results. Therefore, we apply this construction only on components $P_j$ in \eqref{eq: basic functions}  whose volume is `not too small'  since on such sets the derivative of  an affine mapping can be controlled in terms of the mapping itself  by elementary arguments (cf.\ Lemma \ref{lemma: rigid motion}). This in turn allows to control $d_\nabla(\nabla u,\nabla v)$ in terms of $d(u,v)$ on such components. On the other components (i.e.,  those having small volume), we introduce additional interface by a direct geometrical construction, see Lemma \ref{lemma: fundamental estimate}. This strategy leads to a fundamental estimate in $PR_L(\Omega)$, see Lemma \ref{lemma: fundamental estimate2}. Under an additional technical condition, see \eqref{eq: extra condition}, we are able to  provide also a refined version of this result in Lemma \ref{lemma: fundamental estimate2-new} where boundary values are preserved. This is instrumental for the application to the global method of relaxation.

Apart from the fundamental estimate, we encounter another technical difficulty with respect to other integral representation results \cite{BFLM, BDM, BraidesPiat:96,  Conti-Focardi-Iurlano:15}.  There, at least as an intermediate step, one may consider growth conditions of the form  
$$\alpha\mathcal{H}^{d-1}(J_u \cap A)  + \alpha' \int_{J_u \cap A} |[u]| \,d\mathcal{H}^{d-1} \le  \mathcal{F}(u,A)  \le \beta\mathcal{H}^{d-1}(J_u \cap A) + \beta' \int_{J_u \cap A} |[u]| \,d\mathcal{H}^{d-1}$$
for $0 < \alpha \le \beta$ and $0 <\alpha' \le \beta'$. The lower bound allows to apply compactness results in $SBV$. In the present context, however, we  are forced to work with $\alpha' = \beta' = 0$   since in the construction of the fundamental estimate we control only the length of the added interface but not the modification of the jump heights.  Thus,  more delicate arguments are necessary to obtain suitable compactness results and, as a consequence, convergence of  minima for asymptotic Dirichlet problems, see  Lemma \ref{lemma: gamma-min-upper} and Lemma \ref{lemma: gamma-min-lower}. The latter is not only of general interest, but in particular instrumental to show that the uniform continuity condition   $\xi \mapsto f_n(x,\xi,\nu)$   along the sequence of densities $(f_n)_n$ is maintained in the $\Gamma$-limit, see \eqref{eq: h5/6}. These arguments are based on novel truncation techniques, see Section \ref{sec: truncation}, which are inspired by the recent work \cite{Manuel} where compactness results for free-discontinuity  problems  on $(G)SBV^p$ have been derived in a very general sense.

Finally, let us briefly compare our result for $L = \mathbb{R}^{2 \times 2}_{\rm skew}$ with the  integral representation in $SBD^p$, $p>1$, in dimension two, proved in \cite{Conti-Focardi-Iurlano:15}. Although in this specific case our functionals are defined on a subspace of $(G)SBD^p$, our result is not merely a simple consequence of  \cite{Conti-Focardi-Iurlano:15} since there is in general no obvious way to extend a functional from $PR_L$ to  $(G)SBD^p$. Indeed, as explained above, the issue of joining two functions is more delicate in the present context and calls for novel versions of a fundamental estimate.

The paper is organized as follows. In Section \ref{sec: main} we introduce our setting and present our main results about integral representation and $\Gamma$-convergence. Section \ref{sec: prelim} is devoted to preliminaries about Caccioppoli partitions and (piecewise) rigid functions. In Section \ref{sec: fund} we formulate and prove the fundamental estimate. Here, we also present a refinement preserving boundary values and a scaled version on small balls. Section \ref{sec: representation} and Section \ref{sec: gamma} are devoted to the proofs of the integral representation and the $\Gamma$-convergence result, respectively. Finally, Section \ref{sec: examples} discusses the examples $L=SO(d)$, $L = \R^{d\times d}_{\rm skew}$ and introduces a truncation method which is instrumental for the convergence of minima for asymptotic Dirichlet problems.

\section{The setting and main results}\label{sec: main}

\textbf{Notation:} Throughout the paper  $\Omega \subset \R^d$ is open, bounded with Lipschitz boundary. Let $\mathcal{A}(\Omega)$ be the  family of open subsets of $\Omega$, and  $\mathcal{A}_0(\Omega) \subset \mathcal{A}(\Omega)$ be the subset of sets with regular boundary. By $\mathcal{B}(\Omega)$ we denote the family of Borel sets contained in $\Omega$. By $\omega_m$ we denote  the $m$-dimensional measure of the unit ball \BBB in $\R^m$. \EEE  The symbol $B_R(x)$ will denote a ball of radius $R$ centered at $x$ in an Euclidian space.    The notations $\mathcal{L}^d$ and $\mathcal{H}^{d-1}$ are used for  the Lebesgue measure, and the $(d-1)$-dimensional Hausdorff measure in $\mathbb{R}^d$, respectively. For a $\mathcal{L}^d$-measurable set $E\subset\mathbb{R}^d$,  the symbol $\chi_E$ denotes its \BBB characteristic \EEE  function. For $A,B \in \mathcal{A}(\Omega)$ with $\overline{B} \subset A$, we write $B \subset \subset A$.

\textbf{Jump set:} If  $u: \Omega \to \R^d$ is a $\mathcal{L}^d$-measurable function, $u$ is said to have an \emph{approximate limit} $a \in \R^d$ at a point $x \in \Omega$ if and only if
$$
\lim_{\varrho \to 0^+}\frac{\mathcal L^d\left(\{|u-a|\ge \varepsilon\}\cap B_\varrho(x)\right)}{\varrho^d}=0\hbox{ for every }\varepsilon >0\,.
$$
In this case, one writes $\mathrm{ap }\lim_{y\to x}u(y)=a$. The \emph{approximate jump set} $J_u$ is defined as the set of points $x \in \Omega$ such that there exist $a\neq b \in \R^d$ and $\nu \in S^{d-1}:=\{\xi \in \R^{d}: |\xi|=1\}$ with
$$
\mathrm{ap}\lim_{\substack{y\to x\\ \langle y-x,  \nu \rangle >0}}u(y)=a\ \ \ \ \ \ \ \ \ \ \quad \mathrm{ap}\lim_{\substack{y\to x\\ \langle y-x,  \nu \rangle <0}}u(y)=b\,.
$$
The triplet $(a, b, \nu)$ is uniquely determined up to a permutation of $(a, b)$ and a change of sign of $\nu$, and is denoted by $(u^+ (x), u ^-(x), \nu_u (x))$. The jump of  $u$ is the function $[u]:J_u \to \R^d$ defined by $[u](x) := u^+ (x)-u^- (x)$ for every $x \in J_u$.

\textbf{Set of rigid matrices:}  We consider a closed subset  $L\subset \mathbb{R}^{d\times d}$ with the following two properties: First, each pair of matrices in  $L$ does not satisfy 
the \emph{Hadamard compatibility condition} (see \cite{ball.james}), i.e.,  there holds
\begin{align}\label{eq: matrix-assumption}
{\rm rank} (Q_1 - Q_2) \ge 2 \ \ \ \ \ \ \text{ for all $Q_1, Q_2 \in L$, \ \ \  $Q_1 \neq Q_2$.}
\end{align}
Moreover, we suppose that, roughly speaking, there exists a locally  Bilipschitz parametrization of $L$. 
More precisely, we suppose that there exist constants $d_L \in \N$, $0 < c_L< 1$, $ C_L >0$, $r_L \in (0,+\infty]$,  and a surjective Lipschitz mapping  $\Psi_L:(-r_L,r_L)^{d_L} \to L$ with Lipschitz constant $C_L$  such that,  for each $Q \in L$, there exists a {\it right inverse} mapping $\Xi_L: B_{c_Lr_L}(Q) \cap L \to (-r_L,r_L)^{d_L}$ of $\Psi_L$ satisfying  
\begin{align}\label{eq: repr0.0}
|\Xi_L(Q_1) - \Xi_L(Q_2)| \le C_L |Q_1 - Q_2| \ \ \ \text{for all} \ \ Q_1,Q_2 \in B_{c_Lr_L}(Q) \cap L.
\end{align}
In particular, $r_L = \infty$ is admissible. In this case, we use the convention  $c_L r_L=\infty$, which means that $\Psi_L$ has a globally Lipschitz right inverse $\Xi_L$ defined on all of $L$.  If instead $r_L<+\infty$ (that is, $L$ is  compact), it suffices that a Lipschitz right inverse is defined on small balls around each point having uniform radius, and that its Lipschitz constant is uniformly bounded  by $C_L$.

It is well-known that property \eqref{eq: matrix-assumption} is satisfied for  $L=\R^{d\times d}_{\rm skew}$ as well as for $L=SO(d)$. 
Property  \eqref{eq: repr0.0}  is immediate in the case $L=\R^{d\times d}_{\rm skew}$ since it suffices to define $\Psi_L$ as the canonical isomorphism between $\R^{\frac{d(d-1)}2}$ and $\R^{d\times d}_{\rm skew}$,  which is bijective and Bilipschitz. Actually, property \eqref{eq: repr0.0} is also satisfied for $L=SO(d)$, when $d=2$ or $d=3$.

\begin{proposition}\label{prop: rotations}
Let $d=2$, or $d=3$. Then, the set $L=SO(d)$ complies with property \eqref{eq: repr0.0}.
\end{proposition}

This fact, although based on standard  representation properties of rotation matrices,  seems to be nontrivial to us.  For the reader's convenience,  we will  thus give a proof below in Appendix \ref{proof1}.

\textbf{Piecewise rigid functions:} We introduce the space of \emph{piecewise rigid functions} by
\begin{align}\label{eq: PR-def}
PR_L(\Omega) := & \big\{ u: \Omega \to   \R^d \text{ $\mathcal{L}^d$-measurable: } \   u(x) = \sum\nolimits_{j\in \N} (Q_j\, x + b_j) \chi_{P_j}(x), \notag \\
 & \text{where }  Q_j \in L, \, b_j \in \R^d, \text{ and }  (P_j)_j  \text{ is a Caccioppoli partition of } \Omega \big\}. 
\end{align}
 Here and henceforth, we will call an affine mapping of the form $q_{Q,b}(x):=Qx+b$  with $Q\in L$ and $b \in \R^d$ a  {\it rigid motion}. It follows from the properties of Caccioppoli partitions, see  Section \ref{sec: cacc}, that for each $u \in PR_L(\Omega)$ we have that $\mathcal{H}^{d-1}(J_u \setminus \bigcup_j \partial^* P_j) = 0$ and thus $\mathcal{H}^{d-1}(J_u)  < + \infty$. We equip $PR_L(\Omega)$ with the topology induced by measure convergence on $\Omega$. 
 
 When $L = \R^{d\times d}_{\rm skew}$,  one can equivalently characterize $PR_L(\Omega)$ as the subspace of $GSBD$ functions  (see \cite{DM}) whose symmetrized approximate gradient $e(u)$ equals zero $\mathcal{L}^d$-almost everywhere.   For a proof  we refer to \cite[Theorem A.1]{Chambolle-Giacomini-Ponsiglione:2007} and  \cite[Remark 2.2(i)]{Friedrich:15-4}.  In a similar fashion, in the case $L=SO(d)$, $PR_L(\Omega)$ coincides with the $GSBV$ functions whose approximate gradient satisfies $\nabla u(x) \in SO(d)$ for $\mathcal{L}^d$-a.e.\ $x \in \Omega$, see \cite{Chambolle-Giacomini-Ponsiglione:2007}.

\textbf{Functionals:} We consider functionals $\mathcal{F}: PR_L(\Omega) \times \mathcal{B}(\Omega) \to [0,\infty)$ with the following general assumptions: 
\begin{itemize}
\item[{\rm (${\rm H_1}$)}]  $\mathcal{F}(u,\cdot)$ is a Borel measure for any $u \in PR_L(\Omega)$,
\item[(${\rm H_2}$)]  $\mathcal{F}(\cdot,A)$ is lower semicontinuous with respect to convergence in measure on $\Omega$ for any $A \in \mathcal{A}(\Omega)$,
\item[(${\rm H_3}$)]   $\mathcal{F}(\cdot, A)$ is local for any $A \in \mathcal{A}(\Omega)$, in the sense that, if $u,v \in PR_L(\Omega)$ satisfy $u=v$ a.e.\ in $A$, then $\mathcal{F}(u,A) = \mathcal{F}(v,A)$,
\item[(${\rm H_4}$)]  there exist $0 < \alpha < 1$ and $\beta \ge 1$ such that for any $u \in PR_L(\Omega)$ and $B \in \mathcal{B}(\Omega)$, 
$$\alpha\mathcal{H}^{d-1}(J_u \cap B) \le \mathcal{F}(u,B) \le \beta \mathcal{H}^{d-1}(J_u \cap B),$$ 
\item[{\rm (${\rm H_5}$)}]  there exists an increasing modulus of continuity $\sigma:[0,+\infty) \to [0,\beta]$ with $\sigma(0)=0$ such that for any $u,v \in PR_L(\Omega)$ and $S \in B(\Omega)$ with $S \subset  J_u \cap J_v$ we have
$$|\mathcal{F}(u,S) - \mathcal{F}(v,S)| \le \int_{S} \sigma(|[u]- [v]|) \, d\mathcal{H}^{d-1},$$ 
where we choose the orientation $\nu_u = \nu_v$ on $J_u \cap J_v$.
\end{itemize}

We remark that {\rm (${\rm H_1}$)}--{\rm (${\rm H_3}$)} are standard assumptions, see \cite{AmbrosioBraides,BFLM, BraidesPiat:96, ButtazzoDalMaso,  Conti-Focardi-Iurlano:15}. In these results, the growth condition in {\rm (${\rm H_4}$)} is replaced by one of the form $\int_{J_u} (1+|[u]|)\, d\mathcal{H}^{d-1}$ from below and above. Our growth assumption from below is more relevant for fracture models and the growth assumption from above  is instrumental for our fundamental estimate proved in Section \ref{sec: fund}. However, it comes at the expense of more elaborated compactness arguments and  the fact that we need to consider functionals defined on measurable, but possibly not integrable functions. A continuity condition of the form {\rm (${\rm H_5}$)} was also used, e.g., in \cite{BraidesPiat:96, Braides-Defranceschi}.

\textbf{Main results:} We now formulate the first main result of this article addressing integral representation of functionals $\mathcal{F}$ satisfying {\rm (${\rm H_1}$)}--{\rm (${\rm H_5}$)}. To this end, we introduce some further notation: for every $u \in PR_L(\Omega)$ and $A \in \mathcal{A}(\Omega)$ we define 
\begin{align}\label{eq: general minimization} 
\mathbf{m}_{\mathcal{F}}(u,A) = \inf_{v \in PR_L(\Omega)} \  \lbrace \mathcal{F}(v,A): \ v = u \ \text{ in a neighborhood of } \partial A \rbrace,
\end{align}
and  for $x_0 \in \Omega$, $\xi \in \R^d$, and $\nu \in S^{d-1}$ we introduce the functions
\begin{align}\label{eq: jump competitor}
u_{x_0,\xi,\nu}(x) = \begin{cases}  0 & \text{if } \langle x-x_0,\nu\rangle > 0,\\ \xi & \text{if } \langle x-x_0,\nu\rangle < 0. \end{cases} 
\end{align}

 \begin{theorem}[Integral representation]\label{theorem: PR-representation}
Let $\Omega \subset \R^d$ be open, bounded with Lipschitz boundary and   $\mathcal{F}:  PR_L(\Omega) \times \mathcal{B}(\Omega) \to [0,\infty)$ be such that {\rm (${\rm H_1}$)}--{\rm (${\rm H_5}$)} hold. Then 
\begin{align}\label{eq:gdef-new}
\mathcal{F}(u,B) = \int_{J_u\cap  B} f(x,[u](x),\nu_u(x))\, d\mathcal{H}^{d-1}(x)
\end{align}
for all $u \in PR_L(\Omega)$, $B \in \mathcal{B}(\Omega)$, where $f$ is given by
\begin{align}\label{eq:gdef}
f(x_0,\xi,\nu) = \limsup_{\eps \to 0} \frac{\mathbf{m}_{\mathcal{F}}(u_{x_0,\xi,\nu},B_\eps(x))}{\omega_{d-1}\,\eps^{d-1}}
\end{align}
for all $x_0 \in \Omega$, $\xi \in \R^d$, and $\nu \in S^{d-1}$. 
\end{theorem}

The second main theorem addresses $\Gamma$-convergence of functionals $\mathcal{F}$ satisfying {\rm (${\rm H_1}$)} and {\rm (${\rm H_3}$)}--{\rm (${\rm H_5}$)}. For an exhaustive treatment of $\Gamma$-convergence we refer to \cite{Braides-Defranceschi:98, DalMaso:93}.

\begin{theorem}[$\Gamma$-convergence]\label{th: gamma}
 Let $\Omega \subset \R^d$ open, bounded with Lipschitz boundary.  Let $\mathcal{F}_n: PR_L(\Omega) \times \mathcal{B}(\Omega) \to [0,\infty)$ be a sequence of functionals satisfying {\rm (${\rm H_1}$)}, {\rm (${\rm H_3}$)}--{\rm (${\rm H_5}$)} for the same $0 <\alpha < \beta$ and  $\sigma: [0,+\infty) \to [0,\beta]$.   Then there exists $\mathcal{F}:  PR_L(\Omega) \times \mathcal{B}(\Omega) \to [0,\infty)$ satisfying {\rm (${\rm H_1}$)}--{\rm (${\rm H_4}$)} and a subsequence (not relabeled) such that
$$\mathcal{F}(\cdot,A) =\Gamma\text{-}\lim_{n \to \infty} \mathcal{F}_n(\cdot,A) \ \ \ \ \text{with respect to convergence in measure on $A$} $$
for all $A \in  \mathcal{A}_0(\Omega) $. Moreover, if there holds 
\begin{align}\label{eq: condition-new-new}
\limsup_{n \to \infty} \mathbf{m}_{\mathcal{F}_n}(u, B_\eps(x_0)) \le  \mathbf{m}_{{\mathcal{F}}}(u, B_\eps(x_0)) \le \sup\nolimits_{0<\eps' < \eps} \,  \liminf_{n \to \infty} \mathbf{m}_{\mathcal{F}_n}(u, B_{\eps'}(x_0)) 
\end{align}
for all  $u \in PR_L(\Omega)$ and  each ball $B_\eps(x_0) \subset \Omega$,  then   $\mathcal{F}$ satisfies also {\rm (${\rm H_5}$)} and  admits the representation \eqref{eq:gdef-new}--\eqref{eq:gdef}. 
\end{theorem}

We note that condition \eqref{eq: condition-new-new} can be verified  for $L = \mathbb{R}^{d\times d}_{\rm skew}$ and $L = SO(d)$, $d=2,3$, see Section \ref{sec: examples}. Theorem \ref{theorem: PR-representation} and Theorem \ref{th: gamma} will be proved in Section \ref{sec: representation} and Section \ref{sec: gamma}, respectively. The key ingredient for both results, namely a fundamental estimate in $PR_L(\Omega)$, is addressed in Section \ref{sec: fund}. From now on we drop the index $L$ and write $PR(\Omega)$ instead of $PR_L(\Omega)$ if now confusion arises.

\section{Preliminaries}\label{sec: prelim}

\subsection{Caccioppoli partitions}\label{sec: cacc}

We say that a partition $\mathcal{P} = (P_j)_j$ of an open set $\Omega\subset \R^d$, $d \ge 2$, is a \textit{Caccioppoli partition} of $\Omega$ if $\sum\nolimits_j \mathcal{H}^{d-1}(\partial^* P_j) < + \infty$, where $\partial^* P_j$ denotes  the \emph{essential boundary} of $P_j$ (see \cite[Definition 3.60]{Ambrosio-Fusco-Pallara:2000}). Moreover, by $(P_j)^1$ we denote the points where $P_j$ has density one (see again \cite[Definition 3.60]{Ambrosio-Fusco-Pallara:2000}). By definition, the sets $(P_j)^1$ and $\partial^* P_j$ are Borel measurable. The  local structure of Caccioppoli partitions can be characterized as follows (see \cite[Theorem 4.17]{Ambrosio-Fusco-Pallara:2000}).
 
\begin{theorem}[Local structure]\label{th: local structure}
Let $(P_j)_j$ be a Caccioppoli partition of $\Omega$. Then 
$$\bigcup\nolimits_j (P_j)^1 \cup \bigcup\nolimits_{i \neq j} (\partial^* P_i \cap \partial^* P_j)$$
contains $\mathcal{H}^{d-1}$-almost all of $\Omega$.
\end{theorem}

 Essentially, the  theorem states that $\mathcal{H}^{d-1}$-a.e.\ point of $\Omega$ either belongs to exactly one element of the partition or to the intersection of exactly two sets $\partial^* P_i$, $\partial^* P_j$.  We say that a  partition is \textit{ordered} if $\mathcal L^d(P_i) \ge \mathcal L^d(P_j)$ for $i \le j$. Moreover, we say that a set of finite perimeter $P_j$ is \emph{indecomposable} if it cannot be written as $P^1 \cup P^2$ with $P^1 \cap P^2 = \emptyset$, $\mathcal{L}^d(P^1), \mathcal{L}^d(P^2) >0$ and $\mathcal{H}^{d-1}(\partial^* P_j) = \mathcal{H}^{d-1}(\partial^* P^1) + \mathcal{H}^{d-1}(\partial^* P^2)$. We  state a compactness result for ordered Caccioppoli partitions. (See \cite[Theorem 4.19, Remark 4.20]{Ambrosio-Fusco-Pallara:2000} or \cite[Theorem 2.8]{FriedrichSolombrino} for the slightly adapted version presented here.)  

\begin{theorem}[Compactness]\label{th: comp cacciop}
Let $\Omega \subset \R^d$ be a bounded Lipschitz domain.  Let $\mathcal{P}_i = (P_{j,i})_j$, $i \in \N$, be a sequence of ordered Caccioppoli partitions of $\Omega$ with $$\sup\nolimits_{i \ge 1} \sum\nolimits_{j}\mathcal{H}^{d-1}(\partial^* P_{j,i}) < + \infty.$$
Then there exists a Caccioppoli partition $(P_j)_j$  of $\Omega$   and a  subsequence (not relabeled) such that  
$\sum_{j} \mathcal L^d\left(P_{j,i} \triangle  P_j\right) \to 0$  as $i \to \infty$, where   $P_{j,i} \triangle  P_j = (P_{j,i}\setminus P_{j}) \cup (P_{j}\setminus P_{j,i})$.  
\end{theorem}

\subsection{Properties of rigid and piecewise rigid functions}\label{sec: rig/piec-rig}

Recall the function space $PR(\Omega)$ introduced in \eqref{eq: PR-def}, and the fact that each $u \in PR(\Omega)$ can be written as 
 $u = \sum_j q_j \chi_{P_j}$, where  $(P_j)_j$ is a Caccioppoli partition of $\Omega$ and $(q_j)_j$ are  rigid motions, i.e., $q_j(x) = Q_j\, x + b_j$ with $Q_j \in L$ and $b_j \in \R^d$. We point out that the representation of $u$ is not unique. In the following, we will use two specific representations of $u$: (a) We say that the representation is \emph{pairwise distinct} if all affine mappings $(q_j)_{j }$ are pairwise different. In this case, we observe by \eqref{eq: matrix-assumption} that
\begin{align}\label{eq: bdy and jump}
\mathcal{H}^{d-1}\Big(J_u \triangle \Big(\bigcup\nolimits_{j \in \N} \partial^* P_j \setminus \partial \Omega \Big)\Big) = 0.
\end{align}
(b) We say that the representation is \emph{indecomposable}  if each $P_j$ is a indecomposable set of finite perimeter and we have 
$$\mathcal{H}^{d-1}(\partial^* P_i \cap \partial^* P_j) >0 \ \ \text {for} \ \ i\neq j \ \ \ \ \Rightarrow \ \ \ \ q_i \neq q_j.$$
Note that for such representations there also holds by \eqref{eq: matrix-assumption}
\begin{align}\label{eq: bdy and jump-new}
\mathcal{H}^{d-1}\Big(J_u \triangle \Big(\bigcup\nolimits_{j \in \N} \partial^* P_j \setminus \partial \Omega \Big)\Big) = 0.
\end{align}
An indecomposable representation can be deduced from a piecewise distinct representation by splitting each $P_j$ uniquely into its connected components, i.e., into a countable family of pairwise disjoint, indecomposable sets, see \cite[Theorem 1]{Ambrosio-Morel}. 
We start by a compactness result in $PR(\Omega)$.

\begin{lemma}[Compactness and lower semicontinuity]\label{lemma: PR compactness}
Let $\Omega \subset \R^d$ be open, bounded with Lipschitz boundary. \\
\noindent {\rm (i)} Let $(u_n)_n \subset PR(\Omega)$ be a sequence with $\sup_n \int_{\Omega} \psi(|u_n|) + \mathcal{H}^{d-1}(J_{u_n}) <+\infty$, where $\psi:[0,\infty) \to [0,\infty)$ is continuous, strictly increasing,  and satisfies $\lim_{t \to \infty}\psi(t) = + \infty$. Then there exist $u \in PR(\Omega)$ and a subsequence (not relabeled)  such that $u_n \to u$ in measure.\\
\noindent {\rm (ii)} \BBB Given $(u_n)_n \subset PR(\Omega)$ with $u_n \to u$ in measure, there holds $\mathcal{H}^{d-1}(J_u) \le \liminf_{n \to \infty} \mathcal{H}^{d-1}(J_{u_n})$\EEE.  
\end{lemma}

\EEE

The proof of the above compactness result relies on \eqref{eq: matrix-assumption}, as well as on the following auxiliary result, which will be used several times in the sequel.

\begin{lemma}\label{lemma: rigid motion}
Let  $G\in \mathbb{R}^{d \times d}$, $b \in \R^d$.  Let $\delta >0$, $R >0$, and let $\psi \colon \R_+\to \R_+$ be a continuous, strictly increasing function with $\psi(0) = 0$. Consider a measurable, bounded set $E \subset \R^d$  with $E \subset B_R(0)$ and $\mathcal{L}^d(E)\ge \delta$. Then there exists a continuous, strictly increasing function $\tau_\psi: \psi(\R_+) \to \R_+$ with $\tau_\psi(0) = 0$ only depending on $\delta$, $R$, and $\psi$ such that 
\begin{align}\label{eq: estimate1}
|G| + |b| \le \tau_\psi\Big( \fint_{E} \psi(|G\,x + b|)\, \mathrm{d}x \Big).
\end{align}
If $\psi(t) = t^p$, $p \in [1,\infty)$, then $\tau_\psi$ can be chosen as $\tau_\psi(t) = ct^{1/p}$ for  $c=c(p,\delta,R)>0$. Moreover, there exists $c_0>0$ only depending on $\delta$ and $R$ such that 
\begin{align}\label{eq: estimate2}
\Vert G\, x + b \Vert_{L^\infty(B_R(0))} \le (\omega_d R^{d})^{-1} c_0 \Vert  G\,x + b\Vert_{L^1(E)}  \le c_0 \Vert  G\,x + b\Vert_{L^\infty(E)}.
\end{align}
\end{lemma}

\begin{proof}
We start by proving an estimate  under weaker assumptions than in the statement above. We claim that for each measurable, bounded set $E$  with $\diam(E) \le 2R$ (not necessarily contained in $B_R(0)$)  and $\mathcal{L}^d(E)\ge \delta$ there holds
\begin{align}\label{eq: estimate1-newly}
|G|  \le \hat{\tau}_\psi\Big( \fint_{E} \psi(|G\,x + b|)\, \mathrm{d}x \Big)
\end{align}
for a continuous, strictly increasing function $\hat{\tau}_\psi$ with $\hat{\tau}_\psi(0)=0$. It is not restrictive to consider only the case $b=0$. Indeed, every $b \in \mathbb{R}^d$ can be decomposed orthogonally  as $b=\hat b- Gy$, with $\hat{b} \in ({\rm im}(G))^\perp$ and $y \in \R^d$. Then clearly $|Gx+b|\ge |G(x-y)|$ for all $x \in \mathbb{R}^d$. Since $\psi$ and  $\hat{\tau}_\psi$ are increasing, and  $\mathcal{L}^d(E)$  and ${\rm diam}(E)$ are left unchanged by a translation of  $E$, we can then assume $y=\hat b  = b = 0$.

 Since matrix norms are equivalent, we endow $\mathbb{R}^{d \times d}$ with the spectral norm throughout the proof. We fix   an  eigenvector $v$  with unit norm    corresponding to the maximal eigenvalue of the symmetric positive semidefinite matrix $G^TG$, \BBB i.e.\ $G^TG v = |G|^2 v$ by the definition of spectral norm. \EEE   Let $v^\perp$ be the $(d-1)$-dimensional hyperplane orthogonal to $v$. Since $|v|=1$, for each $y \in v^\perp$ there holds  
 \begin{equation}\label{eq: matrix-estimate}
|G (y  + s \, v  ) | = \big( |Gy|^2 + 2s |G|^2 \langle y, v \rangle + s^2|G|^2|v|^2 \big)^{1/2} \ge |s| \,|G|
\end{equation}
for all $s \in \R$. For $r>0$ we define   
 \begin{equation}\label{eq: matrix-estimate2}
E_r:=\{x = y + s\,v \in E: \, y \in v^\perp, \,  |s| \EEE \ge   r\}.
\end{equation}
Using the isodiametric inequality  and the fact that $\diam(E) \le 2R$, we have $\mathcal{L}^d(E \setminus E_{r})\le 2r\omega_{d-1}R^{d-1}$.   Let $m := \sup_{t\in\R_+} \psi(t) \in \R_+ \cup \lbrace + \infty \rbrace$. Hence, setting for each $t \in [0,m)$ $m_t = \min \lbrace 4t, m \rbrace$, \BBB   $r(t)>0$ can be chosen, \EEE only depending on $\delta$ and $R$, such that  
\begin{align}\label{eq: volume away}
\mathcal{L}^d\big(E \setminus E_{r(t)}\big)\le \delta \frac{\sqrt{m_t}-\sqrt{t}}{\sqrt{m_t}+\sqrt{t}}.
\end{align}
Above the right-hand side is extended by continuity with the value $\frac13\delta$ for $t=0$. Note that $r(t)$ is continuous in $t$. We define the function  
\begin{align}\label{eq: hatpsi}
\hat{\tau}_\psi(t) =   r(t)^{-1} \psi^{-1}\big(\sqrt{t}(\sqrt{m_t}+\sqrt{t})/2\big), \ \ \ \ \ t \in [0,m),
\end{align}
 which is clearly well defined for $t \in [0,m)$, and satisfies $\hat{\tau}_\psi(0) = 0$ since $\psi^{-1}(0) = 0$.

By \eqref{eq: volume away} and $\mathcal{L}^d(E) \ge \delta$ we get   $\mathcal{L}^d(E_{r(t)}) \ge 2\sqrt{t} (\sqrt{m_t}+\sqrt{t})^{-1}\mathcal{L}^d(E)$. Let $t = \fint_E \psi(|G\,x|)$ for brevity.  This along with \eqref{eq: matrix-estimate}-\eqref{eq: matrix-estimate2} and the fact that  $\psi \ge 0$ is monotone increasing yields
\begin{align*}
\psi(r(t) |G|) &\le  \frac{1}{\mathcal{L}^d(E_{r(t)})} \int_{E_{r(t)}} \psi(|G\,x|)\, \mathrm{d}x \le \frac{1}{\mathcal{L}^d(E_{r(t)})} \int_{E} \psi(|G\, x|)\, \mathrm{d}x \le \frac{\sqrt{m_t}+\sqrt{t}}{2\sqrt{t}}\fint_{E} \psi(|G\, x|)\, \mathrm{d}x\\
& = \sqrt{t}(\sqrt{m_t}+\sqrt{t})/2.
\end{align*}
This implies $|G| \le r(t)^{-1} \psi^{-1}(\sqrt{t}(\sqrt{m_t}+\sqrt{t})/2) = \hat{\tau}_\psi(t)$ since   $\psi^{-1}$ is strictly increasing, too. This concludes the proof of \eqref{eq: estimate1-newly}.

We now show \eqref{eq: estimate1} for $\tau_\psi := (2R+1)\hat{\tau}_\psi + 2\psi^{-1}$. Whenever $|b|\le 2R|G|$, the statement follows directly from \eqref{eq: estimate1-newly}. If instead $|b|>2|G|R$, since $|Gx|\le R|G|$ for all $x\in B_R(0)$, we have $|Gx+b|>\frac12 |b|$ for all $x \in E \subset B_R(0)$. This implies   $\psi(|b|/2) \le \fint_{E} \psi(|G\,x + b|)\, \mathrm{d}x $ and thus 
\[
|b|\le 2\psi^{-1} \Big(  \fint_{E} \psi(|G\,x + b|)\, \mathrm{d}x \Big)\,.
\]
This along with \eqref{eq: estimate1-newly} and the definition $\tau_\psi = (2R+1)\hat{\tau}_\psi + 2\psi^{-1}$ shows  \eqref{eq: estimate1}. 

We consider the special situation $\psi(t) = t^p$, $p \in [1,\infty)$.  Since $m=\infty$ in this case, in view of \eqref{eq: hatpsi}, it is not hard to check that $\hat{\tau}_\psi(t) \le ct^{1/p}$ and thus $\tau_\psi(t) \le ct^{1/p}$ for some $c$ sufficiently large depending only on $\delta$, $R$, and $p$. Thus, $\tau_\psi$ can be replaced by the function $t \mapsto ct^{1/p}$.

We finally show \eqref{eq: estimate2}. We apply \eqref{eq: estimate1} with $\psi(t)=t$. By using that $\tau_\psi(t) \le ct$ we get $|G| + |b| \le c \fint_E |  G\,x + b| \, \mathrm{d}x$. We conclude the proof by recalling that $\mathcal{L}^d(E) \ge \delta$ and noting that   $|G \, x + b| \le |G|R + |b|$ for all $x \in  B_R(0)$.  
\end{proof}

For similar estimates of this kind, we also refer to \cite{Chambolle-Conti-Francfort:2014, Friedrich:15-4, FriedrichSolombrino}. We can now prove Lemma \ref{lemma: PR compactness}.

\begin{proof}[Proof of Lemma \ref{lemma: PR compactness}]
\BBB We start with (i). \EEE We consider the pairwise distinct representation $u_n = \sum_j q_{j,n} \chi_{P_{j,n}}$ of each $u_n$ and the associated ordered Caccioppoli partitions  $\mathcal{P}_n = (P_{j,n})_j$, $n \in \N$. Observe that the assumption $\sup_{n \ge 1} \mathcal{H}^{d-1}(J_{u_n})<+\infty$ and \eqref{eq: bdy and jump} imply that
$$
\sup\nolimits_{n \ge 1} \sum\nolimits_{j}\mathcal{H}^{d-1}(\partial^* P_{j,n}) < + \infty.
$$
Thus, up  to a subsequence (not relabeled), there exists a  limiting  Caccioppoli partition $(P_j)_j$ in the sense of Theorem \ref{th: comp cacciop}. It is clearly not restrictive to assume that  $\mathcal{L}^d(P_j)>0$ for all $j$, since, after neglecting all null sets, we still have a Caccioppoli partition of $\Omega$. By lower semicontinuity of the perimeter, by using  Theorem \ref{th: local structure}, and by \eqref{eq: bdy and jump}  we also have
\begin{equation}\label{eq: lowsemic}
\frac12\sum\nolimits_{j}\mathcal{H}^{d-1}(\partial^* P_{j}\setminus \partial \Omega) \le \liminf_{n \to \infty} \frac12\sum\nolimits_{j}\mathcal{H}^{d-1}(\partial^* P_{j,n}\setminus \partial \Omega)=\liminf_{n \to \infty} \mathcal{H}^{d-1}(J_{u_n})\,.
\end{equation}

For a fixed $j\in \mathbb{N}$, Theorem \ref{th: comp cacciop} implies that there exists $\delta_j$, independently of $n$, with $\mathcal{L}^d(P_{j,n})\ge \delta_j$ for all $n$. Now,  by assumption there  holds $\fint_{P_{j,n}}\psi(|q_{j,n}(x)|)\,\mathrm{d}x \le \frac{M}{\delta_j}$, where $M := \sup_n \int_\Omega \psi(|u_n|)$.  Hence, we deduce by Lemma \ref{lemma: rigid motion} and the coerciveness of $\psi$  that there exists a constant $c_{ \Omega,  M, j}$ such that
\[
\sup\nolimits_{n \ge 1}\|q_{j,n}\|_{W^{1,\infty}(\Omega)} \le c_{ \Omega,  M, j}.
\]
By the Ascoli-Arzel\`a Theorem, a diagonal argument,  and by the fact that  $L$ is closed, we deduce that there exist rigid motions $(q_j)_j$ so that, for each $j$, there holds
\begin{equation}\label{eq: rigid comp}
\lim_{n\to \infty}\|q_{j,n}-q_j\|_{L^\infty(\Omega)}\to  0
\end{equation}
along a subsequence independent of $j$, which we do not relabel.   We  set $u=\sum_j q_j \chi_{P_j}$, and clearly we get $u \in PR(\Omega)$, while \eqref{eq: rigid comp} and Theorem \ref{th: comp cacciop} give $u_n \to u$ in measure. \BBB To see (ii), we note that \EEE  by construction $J_u \subset \bigcup_j (\partial^* P_j\setminus \partial \Omega)$ up to an  $\mathcal{H}^{d-1}$-negligible  set. Thus,  we deduce the inequality  $\mathcal{H}^{d-1}(J_u) \le \liminf_{n \to \infty} \mathcal{H}^{d-1}(J_{u_n})$ directly  from Theorem \ref{th: local structure} and \eqref{eq: lowsemic}.  
\end{proof}

We now collect some crucial properties of piecewise rigid functions in the blow-up at jump points. In particular, we  construct suitable modifications with \BBB the property that they converge to the function defined in \eqref{eq: jump competitor} in $L^p$, $1\le p < +\infty$, see \eqref{eq: blow up-new}(vi). This convergence property will be instrumental for the proof of the integral representation formula in Section \ref{sec: representation}.  We \EEE denote the half spaces $\lbrace \langle x-x_0,\nu\rangle > 0\rbrace$ and $\lbrace \langle x-x_0,\nu\rangle < 0\rbrace$ by $H^+(x_0,\nu)$ and $H^-(x_0,\nu)$, respectively.

\begin{lemma}[Blow-up at jump points]\label{lemma: blow up}
Let $u = \sum_{j\in \N}q_j\chi_{P_j} \in PR(\Omega)$. Let $\theta \in (0,1)$. For $\mathcal{H}^{d-1}$-a.e.\ $x_0 \in J_u$ we find $i,j\in \N$   such that $x_0 \in \partial^* P_i \cap \partial^* P_j$, and a sequence $u_\eps \in PR(B_\eps(x_0))$ satisfying
\begin{align}\label{eq: blow up-new} 
{\rm (i)} & \ \ \lim_{\eps\to 0} \  \frac{1}{\eps^d}  \mathcal{L}^d\big( (B_\eps(x_0) \cap H^+(x_0,\nu_u)) \setminus P_i \big) +  \lim_{\eps\to 0} \ \frac{1}{\eps^d}  \mathcal{L}^d\big( (B_\eps(x_0) \cap H^-(x_0,\nu_u)) \setminus P_j \big)  = 0, \notag\\
{\rm (ii)} & \ \ \lim_{\eps\to 0} \ \frac{1}{\omega_{d-1}\,\eps^{d-1}} \mathcal{H}^{d-1}\Big(J_u \cap \big(B_\eps(x_0) \setminus B_{t\eps}(x_0)\big)\Big) = (1-  t^{d-1}  ) \ \ \text{ for all } t \in (0,1), \notag\\
{\rm (iii)} & \ \ u_\eps= q_i \chi_{P_i} +q_j\chi_{P_j} \   \text{on $B_{(1-\theta)\eps}(x_0)$,}\notag\\
{\rm (iv)} & \ \ \text{$u_\eps = u$ in a neighborhood of $\partial B_\eps(x_0)$,} \notag \\ 
{\rm (v)} & \ \  \lim_{\eps \to 0} \ \frac{1}{\eps^{d-1}}\Big(\mathcal{H}^{d-1}(J_{u_\eps} \setminus J_u) +  \mathcal{H}^{d-1}\big(\lbrace x\in J_{u_\eps} \cap J_u: \ [u_\eps](x) \neq [u](x)\rbrace \big)\Big) = 0,\notag \\
{\rm (vi)} & \ \ \lim_{\eps \to 0} \   \frac{1}{\eps^d} \int_{B_{(1-\theta)\eps}(x_0)} \big|u_\eps(x) - (u^-(x_0)+ u_{x_0,[u](x_0),\nu_u(x_0)})\big|^p \, \mathrm{d}x = 0 \ \ \ \forall \, 1 \le p < \infty.
\end{align}
\end{lemma}

\begin{proof}
For  $\mathcal{H}^{d-1}$-a.e.\ $x_0 \in J_u$ there exist two components $P_i$ and $P_j$ such that $x_0 \in \partial^* P_i \cap \partial^* P_j$ and
\begin{align}\label{eq: blow up-1}
{\rm (i)} &\ \, \lim_{\eps \to 0} \frac{\mathcal{L}^d\big( (B_\eps(x_0) \cap H^+(x_0,\nu_u)) \setminus P_i \big) + \mathcal{L}^d\big( (B_\eps(x_0) \cap H^-(x_0,\nu_u)) \setminus P_j\big)}{\eps^d} = 0, \notag\\
 {\rm (ii)}& \ \, \lim_{\eps \to 0} \frac{\mathcal{H}^{d-1}\big(B_{\eps}(x_0) \cap J_u \big)}{\omega_{d-1}\,\eps^{d-1}} =  \lim_{\eps \to 0} \frac{\mathcal{H}^{d-1}\big(B_{\eps}(x_0) \cap J_u \cap \partial^* P_i\cap \partial^*P_j\big)}{\omega_{d-1}\,\eps^{d-1}} = 1.
\end{align}
This follows from Theorem \ref{th: local structure} and \cite[Theorem 3.59]{Ambrosio-Fusco-Pallara:2000}. Note that \eqref{eq: blow up-1}  implies \eqref{eq: blow up-new}(i),(ii).  Using the coarea formula and \eqref{eq: blow up-1}(i) we can choose $\gamma_\eps \in ((1-\theta)\eps,\eps)$ such that
\begin{align}\label{eq: blow up-2}
\lim_{\eps \to 0} \frac{\mathcal{H}^{d-1}( \partial B_{\gamma_\eps}(x_0) \setminus (P_i\cup P_j))}{\eps^{d-1}} = 0.
\end{align}
We define $u_\eps \in PR(B_\eps(x_0))$ by
\begin{align*}
u_\eps(x) = \begin{cases}
 u(x)\chi_{P_i \cup P_j}(x) & \text{if} \ x \in B_{\gamma_\eps}(x_0),\\
 u(x) & \text{if} \ x \in B_{\eps}(x_0) \setminus \overline{B_{\gamma_\eps}(x_0)}.
\end{cases}
\end{align*}
The definition directly implies \eqref{eq: blow up-new}(iii),(iv).  By \eqref{eq: blow up-1}(ii)  we observe 
$$\frac{1}{\eps^{d-1}} \mathcal{H}^{d-1}\big((B_{\gamma_\eps}(x_0) \cap J_u ) \setminus (\partial^* P_i \cap \partial^* P_j ) \big) \to 0.$$
 This along with \eqref{eq: blow up-2} shows \eqref{eq: blow up-new}(v). Finally, \eqref{eq: blow up-new}(vi) follows from \eqref{eq: blow up-new}(i),(iii) and the fact that $q_i(x)$ and $q_j(x)$ converge uniformly to $u^+(x_0)$ and $u^-(x_0)$, respectively, as $x \to x_0$.
\end{proof}

\section{Fundamental estimate for $PR(\Omega)$}\label{sec: fund}
 
This section is devoted to a fundamental estimate for functionals  defined on piecewise rigid functions.  It will be the key tool to prove our integral representation and $\Gamma$-convergence results. \BBB  The results in this section will be proven using a  weaker assumption than  (${\rm H_5}$).  We will namely assume 
\begin{itemize}
\item[{\rm (${\rm H_5}^\prime$)}]  there exists an increasing modulus of continuity $\sigma:[0,+\infty) \to [0,\beta]$ with $\sigma(0)=0$ such that for any $u,v \in PR_L(\Omega)$ and $S \in B(\Omega)$ with $S \subset  J_u \cap J_v$ we have
$$|\mathcal{F}(u,S) - \mathcal{F}(v,S)| \le \int_{S} \sigma(|u^+- v^+|+|u^--v^-|) \, d\mathcal{H}^{d-1},$$ 
where we choose the orientation $\nu_u = \nu_v$ on $J_u \cap J_v$.
\end{itemize}
\EEE
 
\subsection{Fundamental estimate} 
 
In this section we formulate different versions of the fundamental estimate. We first give the main statement and afterwards we provide a generalization which also takes boundary data into account. We use the following convention in the whole section: given $A,U\in\mathcal{A}_0(\Omega)$, $A \subset U$, we may regard every $u \in PR(A)$ as a function on $U$, extended by $u = 0$ on $U \setminus A$.

%

 \begin{lemma}[Fundamental estimate]\label{lemma: fundamental estimate2}
Let $\eta >0$ and  $A', A, B \in \mathcal{A}_0(\Omega)$ with $A' \subset \subset  A$, and let $\psi: \R_+ \to \R_+$ be continuous and strictly increasing with $\psi(0) = 0$. Then  there exist a constant $M>0$ and  a  lower semicontinuous function  $\Lambda: PR(A) \times PR(B) \to \R_+\cup\{+\infty\}$ satisfying 
\begin{align}\label{eq: Lambda0}
\Lambda(z_1,z_2) \to 0 \ \text{ whenever } \ \int_{(A\setminus A')\cap B} \psi(|z_1 - z_2|) \to 0
\end{align}  
such that for every functional $\mathcal{F}$ satisfying {\rm (${\rm H_1}$)}, \BBB {\rm (${\rm H_3}$)}, {\rm (${\rm H_4}$)}, and {\rm (${\rm H_5}^\prime$) }\EEE  and for all $u \in PR(A)$, $v \in PR(B)$ there exists a function $w \in PR(A' \cup B)$ such that
\begin{align}\label{eq: assertionfund}
{\rm (i)}& \ \ \mathcal{F}(w, A' \cup B) \le  \mathcal{F}(u,A \cap J_w)  + \mathcal{F}(v, B \cap J_w) \notag \\
 & \  \  \ \ \ \ \ \ \ \ \ \ \ \ \ \ \ \ \ \  \  + \big(\mathcal{H}^{d-1}(\partial A \cup \partial A' \cup \partial B) + \mathcal{F}(u,A)  + \mathcal{F}(v, B) \big) \big(\eta + M\sigma (\Lambda(u,v)) \big), \notag \\ 
{\rm (ii)}& \ \  \BBB \Vert \min\lbrace |w - u|, |w-v| \rbrace \Vert_{L^\infty(A' \cup B)} \EEE \le \Lambda(u,v),
\end{align}
 where $\sigma$ is given in \BBB {\rm (${\rm H_5}^\prime$)}\EEE.  If $\psi(t) = t^p$, $1 \le p < \infty$, then  $\Lambda(z_1,z_2) = M\Vert z_1-z_2 \Vert_{L^p((A \setminus A')\cap B)}$.
\end{lemma}

 \begin{remark}[Topology]\label{rem: topo}
 {\normalfont
We recall that $PR(\Omega)$ is equipped with the topology induced by measure convergence, i.e., a natural choice in Lemma \ref{lemma: fundamental estimate2} is $\psi(t) = \frac{t}{1+t}$. For the applications, however, we are also interested in other topologies, e.g.\ $\psi(t) = t^p$, and therefore we account for different choices in the statement. Note that $\int_{(A\setminus A')\cap B} \psi(|z_1 - z_2|) $ might be infinite. In this case, also $\Lambda$ satisfies $\Lambda(z_1,z_2) = +\infty$, and $\sigma (\Lambda(u,v))$ has to be understood as $\lim_{t \to \infty} \sigma(t)$.  }
 \end{remark}

 \begin{remark}[$L^\infty$-estimate]\label{rem: l-infty}
 {\normalfont
In the case of piecewise constant functions studied by {\sc Ambrosio and Braides} \cite{AmbrosioBraides}, it is possible to construct $w$ in such a way that $w(x) \in \lbrace u(x),v(x)\rbrace$ for a.e.\ $x \in A' \cup B$. In our setting, we slightly have to modify  rigid  motions  by the coarea formula, with modifications controlled in terms of $\Lambda(u,v)$. This allows us to establish an $L^\infty$-control on  \BBB $\min\lbrace |w - u|, |w-v| \rbrace$ \EEE  in  \eqref{eq: assertionfund}(ii). (Note that each function $u,v,w$ itself might not lie in $L^\infty$.)
 }
 \end{remark}

 \begin{remark}[Non-attainment of boundary data]\label{rem: extra cond}
 {\normalfont
{\rm (i)} We emphasize that the function $w$ provided above does not necessarily satisfy $w = v \text{ on } B\setminus A$, as it will be often required in the applications in Section \ref{sec: representation} and Section \ref{sec: gamma}. Indeed, consider the following example (for simplicity, in the planar setting $d=2$ \BBB for scalar-valued functions. The extension to the vector case is straightforward): \EEE 

Let $\rho>0$ and define the set $A' = B_{1-2\rho}(0)$,   $A = B_{1-\rho}(0)$, and $B = B_1(0) \setminus B_{1-3\rho}(0)$. For $\eps>0$, we consider the piecewise constant functions $u \in PR(A)$ and $v_\eps \in PR(B)$  defined by
\begin{align}\label{eq: eps-val}
u = 0 \text{ on } A, \ \ \ \ v_\eps = 0 \text{ on } B \cap \lbrace x_2 > 0 \rbrace, \ \ \ v_\eps = \eps \text{ on } B \cap \lbrace x_2 < 0 \rbrace.
\end{align}
For each $w \in PR(B_1(0))$ with $w = v_\eps \text{ on } B\setminus A$, one observes that each line parallel to $e_2$ intersects $J_w$. \BBB To see this, we choose a piecewise constant function $\bar{w} \in SBV(B_1(0);[0,1])$ with $\bar{w} = w$ on $B \setminus A$ and $\mathcal{H}^1(J_w \triangle J_{\bar{w}}) = 0$, and apply the slicing property  \cite[Theorem 3.108]{Ambrosio-Fusco-Pallara:2000} of $BV$ functions. \EEE  This implies $\mathcal{H}^1(J_w) \ge 2$ and thus $\mathcal{F}(w, B_1(0)) \ge 2\alpha$ by (${\rm H_4}$).  On the other hand, we have 
 \begin{align*}
  & \mathcal{F}(u,A)  +  \mathcal{F}(v_\eps, B) + \big(\mathcal{H}^1(\partial A' \cup \partial A \cup \partial B) + \mathcal{F}(u,A)  + \mathcal{F}(v_\eps, B) \big) \big(\eta +  M\sigma(\Lambda(u,v_\eps)) \big)\\
  & \le 6\rho\beta +  (6\pi + 6\rho\beta)\big(\eta +  M\sigma(\Lambda(u,v_\eps)) \big).
  \end{align*}
  Observe that $\int_{(A \setminus A') \cap B} \psi(|u-v_\eps|) \le    \pi\rho(1-\frac{3}{2}\rho)\psi(\eps) \to 0$ as $\eps\to 0$ for fixed $\rho$ and thus $\Lambda(u,v_\eps) \to 0$ by \eqref{eq: Lambda0}. In view of  $\mathcal{F}(w, B_1(0)) \ge 2\alpha$,   this contradicts \eqref{eq: assertionfund}{\rm (i)} when we choose $\eta$ small enough, and let  first   $\eps \to 0$ and then $\rho \to 0$. 

{\rm (ii)} The example in {\rm (i)} shows that the issue of non-attainment of boundary data  occurs already on the level of piecewise constant functions. The only reason why this problem did not appear in the fundamental estimate for piecewise constant functions by {\sc Ambrosio and Braides}, see \cite[Lemma 4.4]{AmbrosioBraides},  is due to the fact that the functions considered there attain only a \emph{finite} number of different values. In fact, the delicate point here is the case where the functions $u$ and $v_\eps$ attain very similar values, see \eqref{eq: eps-val}. 
 }
 \end{remark}

 For the formulation of a version of the  fundamental estimate with boundary data, we need to introduce the following technical definition:  for sets $A', U \in \mathcal{A}_0(\Omega)$ with \BBB $A' \subset  U$, \EEE a piecewise rigid function  $v = \sum_{j \in \N} q_j \chi_{P_j} \in PR(U \setminus \overline{A'})$ in its pairwise distinct representation  (see Section \ref{sec: rig/piec-rig}), and a constant $\delta>0$ we define
\begin{align}\label{eq: Psi}
\Phi(A',U;v,\delta) := \min\nolimits_{j_1,j_2 \in J, \, j_1 \neq j_2} \Vert q_{j_1} - q_{j_2} \Vert_{L^\infty(U)},
 \end{align}
where $J \subset \N$ denotes the index set of \emph{large components} defined by
\begin{align}\label{eq: touching cond}
J:= \big\{ j \in \N:   \ \BBB \mathcal{L}^d\big(P_j \cap (U \setminus \overline{A'})\big) \EEE  \ge \delta \big\}.
\end{align}
As $J$ contains a finite number of indices, it is clear that $\Phi(A',U;v,\delta)>0$. If $\# J \le 1$, then $\Phi(A',U;v,\delta) = +\infty$. \BBB We remark that the difference of the affine mappings in \eqref{eq: Psi} is compared on $U$ and not on $U \setminus \overline{A'}$ (where $v$ is defined) as in the proof we need to   modify functions in the whole domain $U$ and not only inside $U \setminus \overline{A'}$. On the contrary, we emphasize that in \eqref{eq: touching cond} the volume of the components inside $U \setminus \overline{A'}$ is measured.    \EEE

 \begin{lemma}[Fundamental estimate, boundary data]\label{lemma: fundamental estimate2-new}
Let $\eta >0$ and  $A', A, B \in \mathcal{A}_0(\Omega)$ with $A' \subset \subset  A$. Let $\psi: \R_+ \to \R_+$ be continuous and strictly increasing with $\psi(0) = 0$. Let $\Lambda$ be the function of Lemma \ref{lemma: fundamental estimate2}. Then  there exist constants  $\delta>0$ and $M_1 \ge 1$ such that for   every functional $\mathcal{F}$ satisfying {\rm (${\rm H_1}$)}, \BBB {\rm (${\rm H_3}$)}, {\rm (${\rm H_4}$)}, and {\rm (${\rm H_5}^\prime$) }\EEE  and for all $u \in PR(A)$, $v \in PR(B)$ satisfying the condition 
\begin{align}\label{eq: extra condition}
M_1\Lambda(u,v) \le \Phi(A',A'\cup B; v|_{B \setminus \overline{A'}},\delta),
\end{align}  
 there exists a function $w \in PR(A' \cup B)$ such that
\begin{align}\label{eq: assertionfund-newnew}
{\rm (i)}& \ \ \mathcal{F}(w, A' \cup B) \le  \mathcal{F}(u,A)  + \mathcal{F}(v, B) \notag \\
 & \  \ \ \ \ \ \ \ \ \ \ \ \ \ \ \ \ \  \ \ \ \  + \big(\mathcal{H}^{d-1}(\partial A' \cup \partial A \cup \partial B) + \mathcal{F}(u,A)  + \mathcal{F}(v, B) \big) \big(2\eta + M_2\sigma (\Theta(u,v))\big), \notag \\ 
{\rm (ii)}& \ \   \BBB \Vert \min\lbrace |w - u|, |w-v| \rbrace \Vert_{L^\infty(A' \cup B)} \EEE \le \Theta(u,v),\notag\\
{\rm (iii)} & \ \ w = v \text{ on } B\setminus A,
\end{align}
 where $\sigma$ is given in \BBB{\rm (${\rm H_5}^\prime$)}\EEE, and $M_2>0$ as well as $\Theta: PR(A) \times PR(B) \to\R_+\cup\{+\infty\}$  are independent of $u$, $v$, and $\mathcal{F}$. Here, $\Theta$   is a  lower  semicontinuous  function satisfying 
\begin{align}\label{eq: Theta convergence}
\Theta(z_1,z_2) \to 0 \ \text{ whenever } \ \int_{(A\setminus A')\cap B} \psi(|z_1 - z_2|) \to 0\,.
\end{align}
If $\psi(t) = t^p$, $1 \le p < \infty$, then  $\Theta(u,v) = M_2\Vert u - v \Vert_{L^p((A \setminus A') \cap B)}$.
\end{lemma}

The object $\Phi$ measures how `similar' a function is on different (large) components. Roughly speaking, the technical condition \eqref{eq: extra condition} ensures that, for the functions $u$ and $v$, the phenomenon described in Remark \ref{rem: extra cond} cannot occur. \BBB In this context, we remark that, for given $\delta>0$, the constant $M_1$ will be chosen sufficiently large in the proof, depending on the constant $c_0$ in Lemma~\ref{lemma: rigid motion}. \EEE

In the applications, we will need to use the fundamental estimate on balls of different sizes. To this end, we formulate a scaled version of Lemma \ref{lemma: fundamental estimate2-new}.

 \begin{corollary}[Scaled version of the fundamental estimate]\label{rem: scaling}
 
Let $\eta>0$ and $\rho>0$. Suppose that   $A', A, B \in \mathcal{A}_0(\Omega)$ with $A' \subset \subset  A$ are given such that $\rho A',  \rho A, \rho B \subset \Omega$. Let $u_\rho \in PR(\rho A)$ and $v_\rho \in PR(\rho B)$. Under the assumption that
\begin{align}\label{eq: extra condition-new}
\rho^{-d}MM_1\Vert u_\rho - v_\rho \Vert_{L^1((\rho A\setminus \rho A')\cap \rho B)}\le  \Phi(\rho A', \, \rho A'\cup \rho B; \, v_\rho|_{\rho B \setminus \rho \overline{A'}}, \, \rho^d \delta)
\end{align}  
one finds a function  $w_\rho \in PR(\rho A' \cup \rho B)$ satisfying
\begin{align}\label{eq: extra condition-new-new}
{\rm (i)}& \ \ \mathcal{F}(w_\rho, \rho A' \cup \rho B) \le  \mathcal{F}(u_\rho,\rho A)  +  \mathcal{F}(v_\rho, \rho B) \notag \\
 & \  + \big(\rho^{d-1}C_{A',A,B} + \mathcal{F}(u_\rho,\rho A)  + \mathcal{F}(v_\rho, \rho B) \big) \Big(2\eta + M_2\sigma\Big(   M_2\rho^{-d}\Vert u_\rho - v_\rho\Vert_{L^1(\rho (A\setminus A')\cap \rho B)}\Big) \Big), \notag \\ 
{\rm (ii)}& \ \   \BBB \Vert \min\lbrace |w_\rho - u_\rho|, |w_\rho-v_\rho| \rbrace \Vert_{L^\infty(\rho A' \cup \rho B)} \EEE \le M_2 \rho^{-d} \Vert u_\rho - v_\rho \Vert_{L^1( \rho  (  A\setminus   A')\cap \rho B)},\notag\\
{\rm (iii)}& \ \ w_\rho = v_\rho \text{ on } \rho B\setminus \rho A, 
\end{align}
where $M$ is the constant of Lemma \ref{lemma: fundamental estimate2}, $M_1$, $M_2$, $\delta$ are the constants of Lemma \ref{lemma: fundamental estimate2-new}  (applied for $\psi(t)=t$), and $C_{A',A,B}:= \mathcal{H}^{d-1}(\partial A' \cup \partial A \cup \partial B)$ for brevity.

 \end{corollary}

  The proof of Lemma \ref{lemma: fundamental estimate2} will be addressed in Section \ref{sec: fund-proof1}. The proofs of Lemma \ref{lemma: fundamental estimate2-new} and Corollary \ref{rem: scaling}  will be given in Section \ref{sec: fund-proof2}.   The reader may also skip the following subsections and go directly to the proofs of our main results in Section \ref{sec: representation} and Section \ref{sec: gamma}.

  \subsection{Proof of Lemma \ref{lemma: fundamental estimate2}}\label{sec: fund-proof1}
  This section is devoted to the proof of Lemma \ref{lemma: fundamental estimate2}. As a preparation, we formulate and prove a lemma about the choice of subsets.

 \begin{lemma}[Choice of subsets]\label{lemma: fundamental estimate}
Let $\lambda >0$. Let  $A', A \in \mathcal{A}_0(\Omega)$ with $A' \subset \subset  A$. For $0<t < d_{A',A} := \dist(\partial A', \partial A)$ we define
\begin{align}\label{eq:Et}
E_t := \lbrace x\in \R^d: \ \dist(x,A') < t \rbrace.
\end{align}
 Then for each set of finite perimeter $D\subset \Omega$ there exist $\frac{1}{4} d_{A',A} < T_1 < T_2 < \frac{3}{4} d_{A',A}$ and a function $\varphi \in C^\infty(A)$ with $0 \le \varphi \le 1$, $\varphi = 1$ in a neighborhood of $\overline{E_{T_1}}$, and ${\rm supp}(\varphi) \subset \subset E_{T_2}$ such that the set of finite perimeter $F:= D \cap (E_{T_2} \setminus \overline{E_{T_1}})$  and the function $\varphi$ satisfy
\begin{align}\label{eq: Mlambda}
\mathcal{H}^{d-1}(\partial^* F ) \le \lambda  \mathcal{H}^{d-1}(\partial^* D) + M_\lambda\mathcal{L}^d(D) \ \ \ \ \text{and}  \ \ \ \ \Vert \nabla \varphi \Vert_{\infty} \le M_\lambda, 
\end{align}
where $M_\lambda$ only depends on $\lambda$, $A'$, and $A$. 
 \end{lemma}

 \begin{proof}
 Choose $k \in \N$ such that $k \ge \lambda^{-1}$. Let $t_i = (\frac{1}{4} + \frac{i}{2k}) \, d_{A,A'}$ for $i=0,\ldots,k$, and define  $A_i = E_{t_i} \setminus \overline{E_{t_{i-1}}}$ for $i=1,\ldots,k$. We also define the smaller sets $B_i = E_{t_i^-} \setminus \overline{E_{t_{i-1}^+}}$, where $t_{i}^\pm = t_i \pm \frac{1}{8k} d_{A,A'}$.
For $i=1,\ldots,k$, let $\varphi_i\in C^\infty(A)$  with  $0 \le \varphi_i \le 1$, $\varphi_i = 1$ in a neighborhood of $\overline{E_{t_{i-1}^+}}$, and ${\rm supp}(\varphi) \subset  \subset E_{t_i^-}$, i.e., $\lbrace 0<\varphi_i < 1 \rbrace \subset \subset B_i$. Define
$$M_\lambda = \max\big\{ 16k \, d_{A,A'}^{-1}, \,  \max\nolimits_{i=1,\ldots,k} \Vert \nabla \varphi_i \Vert_\infty \big\}.$$ 
By recalling $k \ge \lambda^{-1}$ we find $i_0 \in \lbrace 1, \ldots, k \rbrace$ such that
\begin{align}\label{eq: Et-newe-old2}
\mathcal{H}^{d-1}(\partial^* D \cap A_{i_0})   \le \frac{1}{k}\sum\nolimits_{i=1}^k  \mathcal{H}^{d-1}(\partial^* D \cap A_{i})   \le \lambda   \mathcal{H}^{d-1}(\partial^* D).
\end{align}
 We  now claim that one can  find $t_{i_0-1}<  T_1 < t_{i_0-1}^+$ and $t_{i_0}^- <T_2 <  t_{i_0}$ such that 
\begin{align}\label{eq: Et-newe-old}
\mathcal{H}^{d-1}(D \cap \partial E_{T_1})  \le \frac{8k}{d_{A',A}} \, \mathcal{L}^d(D \cap A_{i_0}), \ \   \mathcal{H}^{d-1}(D \cap \partial E_{T_2})  \le \frac{8k}{d_{A',A}}  \, \mathcal{L}^d(D \cap A_{i_0}).
\end{align}
 We only prove the first inequality above since the other one is similar. To this aim, we observe that
\[
\{x\in \mathbb{R}^d: t_{i_0-1} <  \mathrm{dist}(x, A') < t_{i_0-1}^+\} \subset A_{i_0}\,.
\]Hence, applying the coarea formula to the Lipschitz function $g(x):= \mathrm{dist}(x, A')$, whose gradient has norm $1$ a.e., (see for instance \cite[Theorem 3.14]{EvansGariepy92})   we get
\[
 \mathcal{L}^d(D \cap A_{i_0})\ge \int_{t_{i_0-1}}^ {t_{i_0-1}^+}\mathcal{H}^{d-1}(D \cap \{g=t\})\,\mathrm{d}t\,. 
\]
Thus, since $t_{i_0-1}^+  - t_{i_0-1}=\frac{d_{A',A}}{8k}$, we can choose $t_{i_0-1}<  T_1 < t_{i_0-1}^+$ such that \eqref{eq: Et-newe-old} holds. 

We define $F:= D \cap (E_{T_2} \setminus \overline{E_{T_1}})$. In view of $\lbrace 0<\varphi_{i_0} < 1 \rbrace \subset B_{i_0}$, the definition of $T_1$ and $T_2$ implies that $\varphi_{i_0}$ satisfies $\varphi_{i_0} = 1$ in a neighborhood of $\overline{E_{T_1}}$, and ${\rm supp}(\varphi_{i_0}) \subset\subset E_{T_2}$. Moreover, by \eqref{eq: Et-newe-old2}--\eqref{eq: Et-newe-old} we get
\begin{align*}
\mathcal{H}^{d-1}(\partial^* F) & \le \mathcal{H}^{d-1}(\partial^* D \cap A_{i_0}) + \mathcal{H}^{d-1}(D \cap \partial E_{T_1}) + \mathcal{H}^{d-1}(D \cap \partial E_{T_2})\\
& \le \lambda   \mathcal{H}^{d-1}(\partial^* D)+ M_\lambda \mathcal{L}^d(D),
\end{align*}
where we used $M_\lambda \ge 16k d_{A,A'}^{-1}$. This yields the first part of \eqref{eq: Mlambda}. The second part of \eqref{eq: Mlambda} follows directly from the definition of $M_\lambda$.   
 \end{proof}

For the proof of Lemma \ref{lemma: fundamental estimate2}, we will need another two ingredients. First, we state an elementary property about the covering of points by balls.

\begin{lemma}[Covering with balls]\label{lemma: balls} 
Let $N \in \mathbb{N}$ and $r_0>0$. Then  each set of points $\lbrace x_1, \ldots, x_N \rbrace \subset \mathbb{R}^m$ can be covered by finitely many pairwise disjoint balls $\lbrace B_{r_k}(y_k) \rbrace_{k=1}^M$, $M \le N$, $(y_k)_{k=1}^M \subset \mathbb{R}^m$, satisfying
\begin{align*}
r_k \in [8^{-N} r_0, r_0], \ \ \ \ \ \  \ \ \  \dist\big( B_{r_i}(y_i),  B_{r_j}(y_j) \big)    > 2   \max_{k=1,\ldots,M} r_k \ \ \ \ \ \  \text{for} \  1 \le i < j \le M.
\end{align*}
\end{lemma}

\begin{proof}
From \cite[Lemma 3.7]{Manuel} applied for $\gamma=4$ and    $R_0 = r_0 8^{-N}$ we get pairwise disjoint balls $\lbrace B_{r_k}(y_k) \rbrace_{k=1}^M$ with $r_k \in [8^{-N} r_0, r_0]$ and $|y_i-y_j|    > 4   \max_{k=1,\ldots,M} r_k$ for $1 \le i < j \le M$. The statement follows with the triangle inequality. 
\end{proof}

We will also need the following result on the approximation of   $GSBV$ functions  with {\it piecewise constant} functions, which can be seen as a piecewise Poincar\'e inequality and essentially relies on the $BV$ coarea formula. For basic properties of $GSBV$ functions we refer to \cite[Section 4]{Ambrosio-Fusco-Pallara:2000}. 

\begin{theorem}[Piecewise Poincar\'e inequality]\label{thm: poincare}
Let $m\ge 1$ and $z \in (GSBV(\Omega))^m$ \BBB with $\Vert \nabla z \Vert_{L^1(\Omega)} + \mathcal{H}^{d-1}(J_z)<\infty$. \EEE Consider a Borel subset $D \subset \Omega$ with finite perimeter.    Fix  $\theta >0$.  Then there exists a partition  $(P_k)_{k=1}^\infty$ of $D$, made of sets of finite perimeter, and a piecewise constant function $z_{\rm pc}:=\sum_{k=1}^\infty b_k \chi_{P_k}$ such that
\[
\begin{split}
{\rm (i)} & \ \   \sum\nolimits_{k=1}^\infty \mathcal{H}^{d-1}\big( (\partial^* {P}_k \cap \BBB D^1) \EEE \setminus J_{z} \big) \le   \theta, \\
{\rm (ii)} &\ \ \Vert z-z_{\rm pc}  \Vert_{L^{\infty}( D )} \le c\theta^{-1}  \Vert  \nabla z\Vert_{L^1(D)},
\end{split}
\]
for a dimensional constant $c=c(m) >0$, \BBB where $D^1$ denotes the set of points with density one. \EEE If additionally  for some $i=1, \dots, m$ the component $z^i$ satisfies $\|z^i\|_{L^\infty(D)}\le M$, we also have
 $\|z_{\rm pc}^i\|_{L^\infty(D)}\le M$.
\end{theorem}
  
  For a proof we refer to \cite[Theorem 2.3]{Friedrich:15-4}, although the argument can be retrieved in previous literature (see for instance  \cite{Amb, Braides-Defranceschi}). The additional property that $L^\infty$-caps are preserved by the approximation, which was not  stated explicitly there, is a direct consequence of the proof. (The values of the piecewise constant approximation are sampled  from  intersections of  nonempty superlevel sets  of the $GSBV$ function.) \BBB Moreover, we remark that in \cite{Friedrich:15-4} only sets $D$ with Lipschitz boundary have been considered. The statement is still true in the present situation, provided that the $\mathcal{H}^{d-1}$-measure of $\partial^* D$ does not contribute in estimate (i). To this end, it is essential to intersect with $D^1$. \EEE
  
As a final preparation for the proof of Lemma \ref{lemma: fundamental estimate2},   we  recall the definition of the  Lipschitz  mapping $\Psi_L$  before \eqref{eq: repr0.0}, and  we discuss how piecewise rigid functions can be parametrized by means of the mapping $\Psi_L$. Given a Caccioppoli partition $(P_j)_{j=1}^\infty$ of $\Omega$, $(\gamma_j)_{j=1}^\infty \subset (-r_L,r_L)^{d_L}$, and $(b_j)_{j=1}^\infty \subset \R^{d}$, we can define a piecewise rigid function $z \in  PR(\Omega)$ by 
\begin{align}\label{eq: basic parametrization}
z(x) = \sum\nolimits_{j=1}^\infty (\Psi_L(\gamma_j)\, x + b_j)\chi_{P_j}(x) \ \ \ \ \ \ \text{ for } x \in \Omega.
\end{align}
   We call $z_{\rm par} = \sum\nolimits_{j=1}^\infty (\gamma_j,b_j)\chi_{P_j} \in GSBV(\Omega;\R^{d_L} \times \R^d)$ a \emph{parametrization} of $z$ and observe that $z_{\rm par}$ is a piecewise constant function in the sense of \cite[Definition 4.21]{Ambrosio-Fusco-Pallara:2000}. Given another piecewise rigid function $\tilde{z} \in PR(\Omega)$ and a corresponding parametrization  $\tilde{z}_{\rm par} = \sum\nolimits_{j=1}^\infty (\tilde{\gamma}_j,\tilde{b}_j)\chi_{\tilde{P}_j}$, we observe that for all $i,j \in \N$ 
\begin{align*}
\Vert z - \tilde{z} \Vert_{L^\infty(P_i \cap \tilde{P}_j)}& \le \sup\nolimits_{x \in \Omega} |x| \, |   \Psi_L(\gamma_i)  -  \Psi_L(\tilde{\gamma}_j)|+ |b_i  -  \tilde{b}_j| \le \sup\nolimits_{x \in \Omega} |x| \, C_L|\gamma_i  -  \tilde{\gamma}_j|+ |b_i  -  \tilde{b}_j|, 
\end{align*}
where $C_L$ is larger or equal to the Lipschitz constant of $\Psi_L$. This implies 
\begin{align}\label{eq: parametrization property}
\Vert z - \tilde{z} \Vert_{L^\infty(\Omega)} \le \big(C_L \sup\nolimits_{x \in \Omega} |x| + 1 \big)    \Vert z_{\rm par} - \tilde{z}_{\rm par} \Vert_{L^\infty(\Omega)}. 
\end{align}
We are now ready to prove Lemma \ref{lemma: fundamental estimate2}.

\begin{proof}[Proof of Lemma \ref{lemma: fundamental estimate2}]
Let $A',A,B \in \mathcal{A}_0(\Omega)$ with $A' \subset \subset A$  and $\eta>0$.  Let $\lambda= \eta\alpha/(8\beta)$, where $\alpha,\beta$ are the constants from {\rm (${\rm H_4}$)},  and let $M_\lambda$ be the constant from Lemma \ref{lemma: fundamental estimate} applied for $A'$, $A$, and $\lambda$. We define $d_{A',A} = \dist(\partial A', \partial A)$. Let $\delta =(\alpha\eta/(8\beta c_{\pi,d}M_\lambda))^d$, where $c_{\pi,d}$ denotes the isoperimetric constant in dimension $d$.  All constants may depend on $L$ without further notice.   Throughout the proof we will assume without loss of generality that
\begin{align}\label{eq: closeuv}
\frac1{\delta} \int_{(A\setminus A^\prime)\cap B}\psi(|u-v|) <\mathrm{sup}_{t\in \mathbb{R}^+}\psi(t)\,.
\end{align}
Indeed, if this  does not hold we simply set $\Lambda(u, v)=+\infty$ and $w=u\chi_A + v\chi_{B\setminus A}$. Then, in view of  {\rm (${\rm H_1}$)} and {\rm (${\rm H_3}$)}--{\rm (${\rm H_4}$)}, \eqref{eq: assertionfund} is satisfied for $M= \beta (\lim_{t \to +\infty}\sigma(t))^{-1}$,  see also Remark \ref{rem: topo}. 

Let $u \in PR(A)$, $v \in PR(B)$ be given, and let  $u = \sum_j q^u_j \chi_{P_j^u}$ and  $v = \sum_j q^v_j \chi_{P_j^v}$ be their pairwise distinct representations (see Section \ref{sec: rig/piec-rig}).   We first define parametrizations $u_{\rm par}$ and $v_{\rm par}$ of $u$ and $v$ in the sense of \eqref{eq: basic parametrization} (Step 1). Then we decompose $A' \cup B$ into a \emph{good set} and a \emph{bad set} (Step 2). Roughly speaking, the bad set consists of the sets $(P_i^u \cap P_j^v)_{i,j \in \N}$ of measure smaller than $\delta$. On the good set, we join the parametrizations $u_{\rm par}$ and $v_{\rm par}$  by means of a  cut-off construction to a function $z_{\rm par}$  (Step 3). Afterwards, we use Theorem \ref{thm: poincare} to approximate $z_{\rm par}$ by a piecewise constant function $w_{\rm par}$. In the good set, the desired function $w$ is then obtained from $w_{\rm par}$ via  \eqref{eq: basic parametrization} and in the bad set we define $w=u$ (Step 4). Finally, we prove  \eqref{eq: Lambda0}--\eqref{eq: assertionfund}  for $w$ (Step 5).

\noindent \emph{Step 1 (Parametrization of $u$ and $v$):} We   introduce the index sets  $\mathcal{P}^u_{\rm large} = \lbrace i \in \N:  \, \mathcal{L}^d(P^u_i) \ge \delta \rbrace$ and  $\mathcal{P}^v_{\rm large} = \lbrace j \in \N: \, \mathcal{L}^d(P^v_j) \ge \delta \rbrace$. Let $Q_i^u$ and $Q_j^v$ be the corresponding matrices in $L$, and  denote by $b^u_i$ and $b^v_j$ the translations. We will show that for all $i \in \mathcal{P}^u_{\rm large}$ and  $j \in \mathcal{P}^v_{\rm large}$, respectively, there exist $\gamma_i^u \in \Psi^{-1}_L(Q_i^u)$ and   $\gamma_j^v \in \Psi^{-1}_L(Q_j^v)$  such that
\begin{align}\label{eq: repr1}
|\gamma^u_i - \gamma^v_j| \le C_\delta|Q^u_i - Q^v_j| \ \ \ \ \  \text{for all} \ \ \  i \in \mathcal{P}^u_{\rm large}, \ j \in \mathcal{P}^v_{\rm large},
\end{align}
for a constant   $C_\delta >0$   depending only on $\delta$, $A$, $B$, \BBB and $L$. \EEE Once this is proved, we define the \emph{parametrizations} $u_{\rm par} \in GSBV(A; \R^{d_L} \times \R^d)$ and $v_{\rm par} \in GSBV(B;\R^{d_L} \times \R^d)$ by 
\begin{align}\label{eq: repr2}
u_{\rm par} = \sum\nolimits_{i=1}^\infty (\gamma^u_i, b^u_i) \chi_{P^u_i} \ \ \ \ \  \text{and} \ \ \ \ \  v_{\rm par} = \sum\nolimits_{j=1}^\infty (\gamma^v_j, b^v_j) \chi_{P^v_j},
\end{align}
where for $i \notin \mathcal{P}^u_{\rm large}$ and  $j \notin \mathcal{P}^v_{\rm large}$ we can choose \emph{arbitrary} $\gamma_i^u \in \Psi^{-1}_L(Q_i^u)$ and   $\gamma_j^v \in \Psi^{-1}_L(Q_j^v)$, respectively.

We now proceed to show \eqref{eq: repr1}. First, if $r_L = + \infty$, then  $\Psi_L$  has a globally Lipschitz right inverse $\Xi_L$ defined on all of $L$,  and the property follows directly from \eqref{eq: repr0.0} when we choose $C_\delta \ge C_L$. Otherwise, we proceed as follows:  as a preliminary observation, we note  that 
\begin{align}\label{eq: repr3}
N:= \# \mathcal{P}^u_{\rm large} +  \# \mathcal{P}^v_{\rm large} \le \delta^{-1} \, \big(\mathcal{L}^d(A) +  \mathcal{L}^d(B)\big).
\end{align}
Indeed, since  $\delta  \le \mathcal{L}^d(P^u_i)$ for $i \in \mathcal{P}^u_{\rm large}$, we have
\begin{align*}
\# \mathcal{P}^u_{\rm large} & \le    \sum\nolimits_{i \in \mathcal{P}^u_{\rm large}}  \delta^{-1}\mathcal{L}^{d}(P^u_i) \le \delta^{-1}\mathcal{L}^{d}(A).
\end{align*}
A similar estimate holds for $\# \mathcal{P}^v_{\rm large}$ with $B$ in place of $A$. This yields \eqref{eq: repr3}.

Let $\mathcal{R} = \lbrace Q_i^u: \,  {i \in \mathcal{P}^u_{\rm large}} \rbrace \cup \lbrace Q_j^v: \, {j \in \mathcal{P}^v_{\rm large}}\rbrace$. For  convenience, we write $\mathcal{R} = (Q_k)_k$. Using Lemma \ref{lemma: balls} for $r_0 = c_L r_L$  we find  a finite number of pairwise disjoint balls $B_1,\ldots,B_n \subset \R^{d \times d}$, $n \le N$, with radius smaller than $c_Lr_L$ such that the balls $(B_i)_{i=1}^n$ cover $(Q_k)_k$, and one has
\begin{align}\label{eq: different balls}
Q_{k_1} \in B_{i_1} \ \ \text{ and } \ \ Q_{k_2} \in B_{i_2}  \ \ \text{ for } \ \ k_1 \neq k_2, \, i_1 \neq i_2 \ \ \ \Rightarrow \ \ \ |Q_{k_1} - Q_{k_2}| \ge   8^{-N}c_Lr_L.  
\end{align}
In view of \eqref{eq: repr0.0}, on each $B_i$, $i=1,\ldots,n$,   a  Lipschitz right-inverse  mapping $\Xi_L$ of $\Psi_L$ is well defined. We set   $\gamma_k = \Xi_L(Q_k)$ for all $Q_k \in B_i$.   We recall that each $Q^u_i$, $i \in \mathcal{P}^u_{\rm large}$, coincides with some $Q_k \in \mathcal{R}$, and we let $\gamma^u_i = \gamma_k$. For each $Q^v_j$ we proceed in a similar fashion. In view of this definition, we derive that \eqref{eq: repr1} holds for the constant $C_\delta = \max\lbrace C_L, 2\sqrt{d_L} 8^N /c_L\rbrace$. Indeed, if $Q_{k_1}, Q_{k_2}$ are contained in the same ball $B_i$, the property follows from \eqref{eq: repr0.0}. Otherwise, \eqref{eq: different balls} and the fact that $\gamma_{k_1},  \gamma_{k_2} \in (-r_L,r_L)^{d_L}$ imply
\begin{align*}
|\gamma_{k_1} - \gamma_{k_2}| \le 2\sqrt{d_L}r_L  \le 2\sqrt{d_L}  8^N c_L^{-1}   |Q_{k_1} - Q_{k_2}|.
\end{align*}
\BBB We note that    $C_\delta >0$   depends only on $\delta$, $A$, $B$, and $L$, see \eqref{eq: repr3}. \EEE

\noindent \emph{Step 2 (Identification of good and bad sets):} Let $(P^{u, v}_k)_k$ be the partition of $(A \setminus A') \cap B$ consisting of the nonempty sets $P_i^u \cap P_j^v \cap ((A \setminus A') \cap B)$, $i,j \in \N$. Clearly,  by Theorem \ref{th: local structure} and {\rm (${\rm H_4}$)} we have
\begin{align}\label{eq: Puv}
\sum\nolimits_{k=1}^\infty\mathcal{H}^{d-1}( \partial^* P^{u, v}_k )& \le \mathcal{H}^{d-1}(\partial A \cup \partial A' \cup \partial B) + 2\mathcal{H}^{d-1}(J_u) + 2\mathcal{H}^{d-1}(J_v)\notag \\
& \le 2\alpha^{-1} \big(\mathcal{F}(u,A) + \mathcal{F}(v,B) + \mathcal{H}^{d-1}(\partial A \cup \partial A' \cup \partial B)\big).
\end{align}
Let $\mathcal{P}^{u, v}_{\rm large} = \lbrace k: \, \mathcal{L}^d(P^{u, v}_k) \ge \delta \rbrace$ and $\mathcal{P}^{u, v}_{\rm small} = \N \setminus \mathcal{P}^{u, v}_{\rm large}$. We also define  
\begin{align}\label{eq: large/small-def}
D_{\rm large} = \bigcup\nolimits_{k \in \mathcal{P}^{u, v}_{\rm large}} P^{u, v}_k, \ \ \ \ \ \ \ \ \ \  D_{\rm small} = \big((A \setminus A') \cap B\big) \setminus D_{\rm large}.
\end{align}
 We observe   by \eqref{eq: Puv}  and the isoperimetric inequality that
\begin{align}\label{eq: Dsmall}
\mathcal{L}^d(D_{\rm small}) &= \sum_{k \in\mathcal{P}^{u,v}_{\rm small}} \mathcal{L}^d(P_k^{u,v}) \le  \delta^{1/d}  \sum_{k \in\mathcal{P}^{u,v}_{\rm small}} (\mathcal{L}^d(P^{u,v}_k))^{(d-1)/d}  \le  c_{\pi,d}\delta^{1/d} \sum_{k \in\mathcal{P}^{u,v}_{\rm small}} \mathcal{H}^{d-1}(\partial^* P_k^{u,v})\notag \\
& \le 2c_{\pi,d}\delta^{1/d} \alpha^{-1} \big(\mathcal{F}(u,A) + \mathcal{F}(v,B) + \mathcal{H}^{d-1}(\partial A \cup \partial A' \cup \partial B)\big). 
\end{align}
  We apply Lemma \ref{lemma: fundamental estimate} on $D_{\rm small}$ for $\lambda=\eta\alpha/(8\beta)$  to obtain $\frac{1}{4} d_{A',A} < T_1 < T_2 < \frac{3}{4} d_{A',A}$ and a function  $\varphi \in C^\infty(A)$ with $\varphi = 1$ in a neighborhood of $\overline{E_{T_1}}$ and ${\rm supp}(\varphi) \subset \subset E_{T_2}$  satisfying \eqref{eq: Mlambda}.  We define the sets   
\begin{align}\label{eq: goodDdef}
D_{\rm bad} = \big(D_{\rm small} \cap (E_{T_2} \setminus \overline{E_{T_1}})\big)^1, \ \  \ \ \ \ D_{\rm good} = \big((A' \cup B) \setminus D_{\rm bad}\big)^1,
\end{align}
where $(\cdot)^1$ denotes the set of points with density one. For an illustration of the sets we refer to Figure \ref{figure 1}. Lemma \ref{lemma: fundamental estimate} and  \eqref{eq: Puv}--\eqref{eq: Dsmall}  imply 
\begin{align}\label{eq: Et-newe-NNN}
\mathcal{H}^{d-1}\big(\partial^* D_{\rm bad}\big)& \le  \lambda  \mathcal{H}^{d-1}(\partial^* D_{\rm small}) + M_\lambda\mathcal{L}^d(D_{\rm small}) \notag\\
& \le \frac{\eta}{2\beta} (\mathcal{F}(u,A) + \mathcal{F}(v,B) + \mathcal{H}^{d-1}(\partial A \cup \partial A' \cup \partial B)), 
\end{align}
where we used $\lambda = \eta\alpha/(8\beta)$ and $\delta = (\alpha\eta/(8\beta c_{\pi,d}M_\lambda))^d$.

\begin{figure}[h]
\centering
\begin{tikzpicture}[scale=0.65]
%
%

 \hspace{-5.1cm}
 \draw[fill=red!10, opacity=0.2](-3, -3)rectangle	 (3, 3) ;
\draw[dashed](-1.2, -1.2)rectangle	 (1.2, 1.2) ; 
\draw[dashed](-2.5, -2.5)rectangle	 (2.5, 2.5) ; 
\draw[fill=red!10, opacity=1.0](-0.7, -0.7)rectangle	 (0.7, 0.7) ;  
\draw[fill=cyan!10, opacity=0.4](0, -2)rectangle	 (4, 2) ;  
%

 \node at (3.6, 1.7){$B$};
 \node at (0.3,0 ){$A'$};
 \node at (-2,2 ){$A$};

 \hspace{5.12cm}
 \draw[fill=red!10, opacity=0.2](-3, -3)rectangle	 (3, 3) ;
\draw[dashed](-1.2, -1.2)rectangle	 (1.2, 1.2) ; 
\draw[dashed](-2.5, -2.5)rectangle	 (2.5, 2.5) ; 
\draw[fill=cyan, opacity=0.2](0, -2)rectangle	 (3, 2) ;
\draw[fill=red!10, opacity=1.0](-0.7, -0.7)rectangle	 (0.7, 0.7) ;  
\draw[fill=cyan!10, opacity=0.2](0, -2)rectangle	 (4, 2) ;  
\draw[fill=gray, opacity=0.9](0.8, 0.7)rectangle	 (2, 0.8) ;  
\draw[fill=gray, opacity=0.9](0.8, 0.7)rectangle	 (2, 0.8) ; 
\draw[fill=gray, opacity=0.9](0.85, 0.3)rectangle	 (2.7, 0.38) ; 
\draw[fill=gray, opacity=0.9](1.5, -1.5)rectangle	 (1.9, -1.1) ; 
\draw[fill=gray, opacity=0.9](2.65, -1.3)rectangle	 (2.85, -1.1) ; 
\draw[fill=gray, opacity=0.9](1.5, -0.5)rectangle	 (2.1, -0.35) ;

 \hspace{5.12cm}
 \draw[fill=red!10, opacity=0.2](-3, -3)rectangle	 (3, 3) ;

\draw[fill=cyan!25, opacity=1.0](0, -2)rectangle	 (3, 2) ;
\draw[fill=cyan!25, opacity=1.0](-0.7, -0.7)rectangle	 (0.7, 0.7) ; 
\draw[fill=cyan!25, opacity=1.0](0, -2)rectangle	 (4, 2) ;  
 \draw[fill=gray, opacity=0.9](1.2, 0.7)rectangle	 (2, 0.8) ;  
\draw[fill=gray, opacity=0.9](1.2, 0.3)rectangle	 (2.5, 0.38) ; 
\draw[fill=gray, opacity=0.9](1.5, -1.5)rectangle	 (1.9, -1.1) ; 
\draw[fill=gray, opacity=0.9](1.5, -0.5)rectangle	 (2.1, -0.35) ; 
\draw[dashed](-1.2, -1.2)rectangle	 (1.2, 1.2) ; 
\draw[dashed](-2.5, -2.5)rectangle	 (2.5, 2.5) ; 

\node[shape=rectangle, minimum height=0.87cm, minimum width=0.1cm, fill=cyan!25, opacity=1.0] at (0, 0){};
 
\end{tikzpicture}
\caption{Left: The sets $A'$, $A$, and $B$, and $\partial E_{T_1}$, $\partial E_{T_2}$ (dashed).  (For illustration purposes, we replaced $\dist(x,A')$ in \eqref{eq:Et} by $\dist_\infty(x,A')$ in the picture.) \EEE  Middle:  $D_{\rm large}$ (blue) and $D_{\rm small}$ (gray).     Right: $D_{\rm good}$ (blue) and $D_{\rm bad}$ (gray).}
\label{figure 1}
\end{figure}
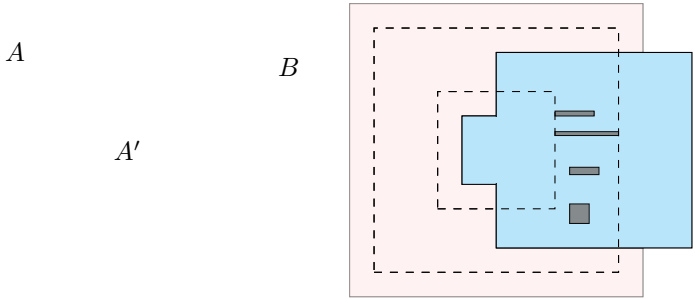

\noindent \emph{Step 3 (Joining $u_{\rm par}$ and $v_{\rm par}$ on $D_{\rm good}$):} Choose $R$ sufficiently large depending on $A$ and $B$ such that $A' \cup B \subset B_R(0)$. Recall the function $\psi$ given in the statement of the lemma.  Consider  $P^{u, v}_k$, $k \in \mathcal{P}^{u, v}_{\rm large}$, and choose $i \in \mathcal{P}^u_{\rm large}$, $j \in \mathcal{P}^v_{\rm large}$ such that $P_k^{u, v} = P_i^u \cap P_j^v \cap ((A \setminus A') \cap B)$.  By Lemma \ref{lemma: rigid motion} there exists a continuous, strictly increasing function $\tau_\psi: \psi(\R_+) \to \R_+$ with $\tau_\psi(0) = 0$ only depending on $\delta$, $R$, and $\psi$ such that by \eqref{eq: repr1}
 $$|\gamma^u_i - \gamma^v_j| + |b^u_i - b^v_j| \le C_\delta |Q_i^u - Q_j^v| + |b^u_i - b^v_j| \le  C_\delta  \, \tau_\psi\Big(\fint_{P_k^{u,v}} \psi\big(|(Q_i^u - Q_j^v)\,x + (b^u_i - b^v_j)|\big)\, \mathrm{d}x \Big).$$
Recall that $\mathcal{L}^d(P_k^{u,v}) \ge \delta$.   For each $z_1 \in PR(A)$, $z_2 \in PR(B)$,  we let 
\begin{align}\label{eq: Lambda1}
\Lambda_*(z_1,z_2) :=  C_\delta \,   \tau_\psi\Big( \frac{1}{\delta}  \int_{(A \setminus A') \cap B} \psi(|z_1 -z_2|)\, \mathrm{d}x \Big) 
\end{align}
if $\delta^{-1} \int_{(A \setminus A') \cap B} \psi(|z_1 -z_2|)\, \mathrm{d}x < \sup\nolimits_{t \in \R_+} \psi(t)$, and $\Lambda_*(z_1,z_2) = +\infty$ else. (Note that this is consistent with the definition below \eqref{eq: closeuv}.)  Recalling \eqref{eq: repr2} we thus find that
\begin{align}\label{eq: u,v-dist}
\Vert u_{\rm par} - v_{\rm par} \Vert_{L^\infty(D_{\rm large})} \le \Lambda_*(u,v) <+\infty,
\end{align}
where $D_{\rm large}$ is defined in \eqref{eq: large/small-def}. 

Let $\varphi \in C^\infty(A \cup B)$ be the function provided by Lemma \ref{lemma: fundamental estimate} which satisfies $0 \le \varphi \le 1$,    $\varphi = 1$ in a neighborhood of $\overline{E_{T_1}}$, and ${\rm supp} \, \varphi \subset \subset E_{T_2}$. We define 
$$z_{\rm par} = \varphi  u_{\rm par} + (1-\varphi)  v_{\rm par} \in GSBV(A' \cup B;\R^{d_L} \times \R^d).$$
 As $u_{\rm par}$ and $v_{\rm par}$ are piecewise constant, we get $\nabla z_{\rm par} = 0$ on $(A' \cup B) \setminus \lbrace 0 <\varphi< 1 \rbrace$    and 
$ \nabla z_{\rm par} = \nabla \varphi \otimes (u_{\rm par} -   v_{\rm par})$  on $\lbrace 0 <\varphi< 1 \rbrace$. By recalling the definition of $D_{\rm large}$ and  $D_{\rm good}$  in \eqref{eq: large/small-def} and \eqref{eq: goodDdef}, respectively, we observe $D_{\rm good} \cap \lbrace 0 <\varphi< 1 \rbrace \subset D_{\rm good} \cap (E_{T_2} \setminus \overline{E_{T_1}})  \subset D_{\rm large}$ \BBB up to an $\mathcal{L}^d$-negligible set, \EEE  see Figure \ref{figure 1}.  Therefore,   we obtain  by \eqref{eq: u,v-dist}
\begin{align}\label{eq: goodset1}
{\rm (i)} & \  \ \BBB \Vert \max\lbrace |z_{\rm par} - u_{\rm par}|,  |z_{\rm par} - v_{\rm par}|  \rbrace \Vert_{L^\infty(D_{\rm large})}  \EEE \le  \Vert  u_{\rm par} - v_{\rm par} \Vert_{L^\infty(D_{\rm large})} \le \Lambda_*(u,v),   \notag\\
{\rm (ii)} & \ \ \Vert \nabla z_{\rm par} \Vert_{L^1(D_{\rm good})} \le \Vert \nabla \varphi \Vert_\infty \Vert \Vert u_{\rm par} - v_{\rm par} \Vert_{L^1(D_{\rm large})} \le  \mathcal{L}^d(A' \cup B) \,  \Vert \nabla \varphi \Vert_\infty \, \Lambda_*(u,v). 
\end{align}
Moreover, since $J_{u},J_{u_{\rm par}}$ and  $J_{v},J_{v_{\rm par}}$ coincide up to $\mathcal{H}^{d-1}$-negligible sets,  we have 
\begin{align}\label{eq: Jzper}
J_{z_{\rm par}}\cap  D_{\rm good} \subset  \big( (J_u \cap E_{T_2}) \cup (J_v \setminus \overline{E_{T_1}}) \big) \cap  D_{\rm good}
\end{align}
up to an $\mathcal{H}^{d-1}$-negligible set. 

\noindent \emph{Step 4 (Definition of the piecewise rigid function $w$ using $z_{\rm par}$):}   We apply Theorem \ref{thm: poincare} for $z=z_{{\rm par}}$, $D={D_{\rm good}}$,  and for $\theta =  \eta (2\beta)^{-1} \, \mathcal{H}^{d-1}(\partial A \cup \partial A' \cup \partial B)$  to find a partition $({P}_k)_k$ of $D_{\rm good}$ and  corresponding constants $(\gamma_k,b_k)_{k=1}^\infty \subset (-r_L,r_L)^{d_L} \times  \R^{d}$  such that 
\begin{align}\label{eq: kornpoinsharp0-appl}
\begin{split}
{\rm (i)} & \ \   \sum\nolimits_{k=1}^\infty \mathcal{H}^{d-1}\big( (\partial^* {P}_k \cap D_{\rm good}) \setminus J_{z_{\rm par}} \big) \le  \eta (2\beta)^{-1} \, \mathcal{H}^{d-1}(\partial A \cup \partial A' \cup \partial B),\\
{\rm (ii)} &\ \ \Vert z_{\rm par} - (\gamma_k,b_k)  \Vert_{L^{\infty}({P}_k)} \le C_\eta  \Vert  \nabla z_{\rm par}\Vert_{L^1(D_{\rm good})} \ \ \ \ \text{for all} \ \ k \in \N,
\end{split}
\end{align}
where $C_\eta > 0$ depends on $\eta$, $\beta$, $A$, $A'$, and $B$.
We define $w_{\rm par} = \sum_{k=1}^\infty (\gamma_k,b_k)\chi_{{P}_k}$ on $D_{\rm good}$. By \eqref{eq: Mlambda}, \eqref{eq: goodset1}(ii), and \eqref{eq: kornpoinsharp0-appl}(ii) we obtain 
$$\Vert w_{\rm par} - z_{\rm par} \Vert_{L^\infty(D_{\rm good})} \le C_\eta \, {  \mathcal{L}^d(A' \cup B) \, } \Vert \nabla \varphi  \Vert_\infty \, \Lambda_*(u,v) \le C_\eta M_\lambda  \, \Lambda_*(u,v), $$
 where in the last step we passed to a larger constant $C_\eta$.  We observe that $D_{\rm good}$ coincides with $( (A' \cup B) \cap E_{T_1}) \cup (B\setminus  \overline{E_{T_2}}) \cup D_{\rm large}$ up to set of negligible $\mathcal{L}^d$-measure, see \eqref{eq: goodDdef}  and Figure \ref{figure 1}.  By  \eqref{eq: goodset1}(i) and the fact that $z_{\rm par}  = u_{\rm par}$ on $ (A' \cup B) \cap E_{T_1}$,  $z_{\rm par}  = v_{\rm par}$ on $B \setminus \overline{E_{T_2}}$, we get
\begin{align*}
{\rm (i)} & \ \ \Vert w_{\rm par} - u_{\rm par} \Vert_{L^\infty( (A' \cup B) \cap E_{T_1})} \le C_\eta M_\lambda \,  \Lambda_*(u,v), \ \  \Vert w_{\rm par} - v_{\rm par} \Vert_{L^\infty(B \setminus  {E_{T_2}})} \le C_\eta M_\lambda \, \Lambda_*(u,v),\\
{\rm (i)} & \ \  \BBB \Vert \max\lbrace |w_{\rm par} - u_{\rm par}|,  |w_{\rm par} - v_{\rm par}|  \rbrace \Vert_{L^\infty(D_{\rm large})} \EEE \le (1+ C_\eta M_\lambda)\Lambda_*(u,v). 
\end{align*}
Define  $w^{\rm good} \in PR(D_{\rm good})$  by $w^{\rm good}(x) = \sum_{k=1}^\infty (\Psi_L(\gamma_k) \, x + b_k)\chi_{{P}_k}(x)$.   Recalling $A' \cup B \subset B_R(0)$ by the choice of $R$, we get by \eqref{eq: parametrization property}
\begin{align}\label{eq: large-w-est}
{\rm (i)} & \ \ \Vert   w^{\rm good} - u\Vert_{L^\infty( (A' \cup B) \cap E_{T_1})} \le  \tfrac{1}{2}\Lambda(u,v) , \ \ \ \ \  \Vert   w^{\rm good} - v\Vert_{L^\infty(B \setminus {E_{T_2}})} \le  \tfrac{1}{2} \Lambda(u,v),\notag\\
{\rm (ii)} & \ \  \BBB \Vert \max\lbrace |w^{\rm good} - u|, |w^{\rm good} - v| \rbrace \Vert_{L^\infty(D_{\rm large})} \EEE  \le  \tfrac{1}{2}\Lambda(u,v), 
\end{align}
where $\Lambda$ is defined by 
\begin{align}\label{eq: Lambda2}
\Lambda(z_1,z_2): = 2   (1+ C_LR) (1+ C_\eta M_\lambda) \Lambda_*(z_1,z_2) \ \ \ \ \ \text{for} \ \ \ z_1 \in PR(A), \ z_2 \in PR(B).
\end{align}
 We now  define the piecewise rigid function $w \in PR(A' \cup B)$ by
\begin{align}\label{eq: ww-def}
w = \begin{cases}
w^{\rm good}   & \text{on } D_{\rm good}, \\
u & \text{on } D_{\rm bad}.
\end{cases}
\end{align}
In particular, this definition implies
 \begin{align}\label{eq: two diff prop}
 \Vert w - u \Vert_{L^\infty( D_{\rm good} \cap E_{T_2} )}  \le \tfrac{1}{2}\Lambda(u,v), \ \ \ \ \ \ \  \Vert w - v \Vert_{L^\infty( D_{\rm good} \setminus  {E_{T_1}} )}  \le \tfrac{1}{2}\Lambda(u,v).
\end{align} 
In fact, this follows from \eqref{eq: large-w-est} and the fact that $(E_{T_2}\setminus E_{T_1}) \cap D_{\rm good} \subset D_{\rm large}$, see  \eqref{eq: large/small-def} and \eqref{eq: goodDdef}.  We close this step of the proof by noticing that 
\begin{align}\label{eq:last in step 2}
\mathcal{H}^{d-1}\big( (J_{w} \cap D_{\rm good})\setminus J_{z_{\rm  par }}  \big) \le \frac{\eta}{2\beta} \mathcal{H}^{d-1}(\partial A \cup \partial A' \cup \partial B) 
\end{align}
 which follows from the definition of $ w^{\rm good}$ and  \eqref{eq: kornpoinsharp0-appl}(i).

\noindent \emph{Step 5 (Proof of \eqref{eq: Lambda0}--\eqref{eq: assertionfund}):} Having defined $w$, it remains to confirm \eqref{eq: Lambda0}--\eqref{eq: assertionfund}.  Recall the definition of $\Lambda$ in \eqref{eq: Lambda2}. In view of \eqref{eq: Lambda1}, \eqref{eq: Lambda2},  and the fact that $\tau_\psi(0) = 0$,  property \eqref{eq: Lambda0} holds.  By Fatou's lemma it is elementary to check that $\Lambda$ is lower semicontinuous.  In particular,  if $\psi(t) = t^p$, $1 \le p < \infty$, then  $\Lambda(z_1,z_2) = M\Vert z_1-z_2 \Vert_{L^p((A \setminus A')\cap B)}$ for some $M>0$ sufficiently large  since in this case $\tau_\psi(t) = ct^{1/p}$, see Lemma \ref{lemma: rigid motion}. 

Let us now show  \eqref{eq: assertionfund}. We first observe that \eqref{eq: assertionfund}(ii) follows from \eqref{eq: ww-def}--\eqref{eq: two diff prop}. Thus, it remains to prove \eqref{eq: assertionfund}(i). Recall the definition of   $D_{\rm good}$ and $D_{\rm bad}$ in \eqref{eq: goodDdef}. By {\rm (${\rm H_1}$)}, {\rm (${\rm H_3}$)},   {\rm (${\rm H_4}$)}, and the definition of $w$ we obtain 
\begin{align}\label{eq: conclusi1}
\mathcal{F}(w, A' \cup B) & =  \mathcal{F}(w, D_{\rm good} ) +  \mathcal{F}(w, \partial^* D_{\rm bad}) +  \mathcal{F}(w, D_{\rm bad} )\notag \\
& = \mathcal{F}(w,  D_{\rm good}  ) +  \mathcal{F}(w, \partial^* D_{\rm bad}) +  \mathcal{F}(u, J_w \cap D_{\rm bad}).
\end{align}
It now suffices to show that there holds
\begin{align}\label{eq: conclusi2}
{\rm (i)}  & \ \  \mathcal{F}(w, D_{\rm good})  \le \mathcal{F}(u, D_{\rm good} \cap A \cap J_w  ) + \mathcal{F}(v, D_{\rm good} \cap B \cap J_w) + \Delta,\notag\\
{\rm (ii)} & \ \   \mathcal{F}(w, \partial^* D_{\rm bad}) \le   \Delta, 
\end{align}
where for brevity we set  
$$\Delta =  \big(\mathcal{H}^{d-1}(\partial A \cup \partial A' \cup \partial B\big) + \mathcal{F}(u,A) + \mathcal{F}(v,B)) \, \big) \big(\eta/2 + \alpha^{-1}\sigma(\Lambda(u,v) \big) .$$ 
In fact, once this is shown, \eqref{eq: assertionfund}(i) follows from \eqref{eq: conclusi1} for $M \ge 2 \alpha^{-1}$.

\emph{Proof of \eqref{eq: conclusi2}(i).}  In view of {\rm (${\rm H_1}$)}, {\rm (${\rm H_4}$)},  and \eqref{eq: Jzper},  we find 
\begin{align}\label{eq: Gamma0}
\mathcal{F}(w, D_{\rm good}   ) \le \sum\nolimits_{j=1}^3 \mathcal{F}(w, \Gamma_j),
\end{align}
 where we define 
\begin{align*}
&\Gamma_1 := (J_w \cap D_{\rm good}) \cap (J_u \cap E_{T_2}),\\
& \Gamma_2 := (J_w \cap D_{\rm good})  \cap (J_v \setminus \overline{E_{T_1}}), \\
& \Gamma_3 := (J_w \cap D_{\rm good})  \setminus J_{z_{\rm   par }}. 
\end{align*}
We  estimate each $\mathcal{F}(w, \Gamma_j)$ separately. 

(1) For $\mathcal{H}^{d-1}$-a.e.\ $x \in \Gamma_1 = J_w \cap J_u \cap E_{T_2} \cap D_{\rm good}$, the one-sided approximate limits $w^+(x),w^-(x)$ of $w$ satisfy $|w^+(x) -u^+(x)|,|w^-(x)-u^-(x)| \le \frac{1}{2}\Lambda(u,v)$ by \eqref{eq: two diff prop}, where we choose $\nu_w = \nu_u$ on $J_w \cap J_u$. This implies \BBB $|w^+(x) -u^+(x)|+|w^-(x)-u^-(x)|   \le \Lambda(u,v) $ \EEE for $\mathcal{H}^{d-1}$-a.e.\ $x \in \Gamma_1$. Thus, by  {\rm (${\rm H_4}$)} this yields 
\BBB $$\int_{\Gamma_1} \sigma(|w^+ -u^+|+|w^--u^-|) \, d\mathcal{H}^{d-1} \le \mathcal{H}^{d-1}(J_u) \, \sigma(\Lambda(u,v)) \le \alpha^{-1}\mathcal{F}(u,A) \, \sigma(\Lambda(u,v)),$$
where $\sigma$ is the modulus of continuity from {\rm (${\rm H_5}^\prime$)}. This implies  by {\rm (${\rm H_5}^\prime$)} 
\begin{align}\label{eq: Gamma1}
\mathcal{F}(w,\Gamma_1) &\le \mathcal{F}(u,\Gamma_1) + \int_{\Gamma_1} \sigma(|w^+ -u^+|+|w^--u^-|) \, d\mathcal{H}^{d-1}\notag\\
& \le  \mathcal{F}(u,\Gamma_1) +  \alpha^{-1}\mathcal{F}(u,A) \, \sigma(\Lambda(u,v)). 
\end{align}

\EEE (2) In a similar fashion, for $\mathcal{H}^{d-1}$-a.e.\ $x \in \Gamma_2$, we have \BBB  $|w^+(x) -u^+(x)|+|w^-(x)-u^-(x)| \le 2 \frac{1}{2}\Lambda(u,v) =  \Lambda(u,v) $ by \eqref{eq: two diff prop}. Therefore, we have by {\rm (${\rm H_4}$)} and {\rm (${\rm H_5}^\prime$)} 
\begin{align}\label{eq: Gamma2}
\mathcal{F}(w,\Gamma_2) &\le \mathcal{F}(v,\Gamma_2) + \int_{\Gamma_2} \sigma( |w^+ -u^+|+|w^--u^-|) \, d\mathcal{H}^{d-1}\notag \\
 &\le\mathcal{F}(v,\Gamma_2) +  \alpha^{-1}\mathcal{F}(v,B) \, \sigma(\Lambda(u,v)).
\end{align}

\EEE (3) Finally, \eqref{eq:last in step 2} and {\rm (${\rm H_4}$)}  imply 
\begin{align}\label{eq: Gamma3}
\mathcal{F}(w,\Gamma_3) \le \beta\mathcal{H}^{d-1}\big( (J_w \cap D_{\rm good})  \setminus J_{z_{\rm par}}\big)   \le   \frac{\eta}{2}  \mathcal{H}^{d-1}(\partial A \cup \partial A' \cup \partial B).  
\end{align}
By combining \eqref{eq: Gamma0}--\eqref{eq: Gamma3} we obtain \eqref{eq: conclusi2}(i).

 \emph{Proof of \eqref{eq: conclusi2}(ii).}  We use  {\rm (${\rm H_4}$)} and \eqref{eq: Et-newe-NNN} to find 
\begin{align}\label{eq: Gamma4}
\mathcal{F}(w, \partial^* D_{\rm bad}   ) \le \beta\mathcal{H}^{d-1}\big(\partial^* D_{\rm bad}\big) \le   \frac{\eta}{2} \big(\mathcal{F}(u,A) + \mathcal{F}(v,B) + \mathcal{H}^{d-1}(\partial A \cup \partial A' \cup \partial B)\big) \le \Delta. 
\end{align}
This concludes the proof.  
\end{proof}

\begin{remark}\label{rem: for later}
{\normalfont

For later purposes in the proof of Lemma \ref{lemma: fundamental estimate2-new}, we observe that by the estimate on $\Gamma_3$ and $\partial^* D_{\rm bad}$, see \eqref{eq: Gamma3}--\eqref{eq: Gamma4},  we have that
\begin{align*}
\mathcal{H}^{d-1}(J_w) \le \mathcal{H}^{d-1}(J_u \cup J_v)  + \eta \big(\mathcal{F}(u,A) + \mathcal{F}(v,B) + \mathcal{H}^{d-1}(\partial A' \cup \partial A \cup \partial B) \big).
 \end{align*}
 By  {\rm (${\rm H_4}$)} this yields
 \begin{align*}
\mathcal{H}^{d-1}(J_w) \le (1+ \eta) \alpha^{-1} \big(\mathcal{F}(u,A) + \mathcal{F}(v,B) + \mathcal{H}^{d-1}(\partial A' \cup \partial A \cup \partial B) \big).
 \end{align*}
Moreover, \eqref{eq: two diff prop} implies that with $K :=\lbrace x \in \R^d: \, \dist(x,A') \ge \frac{3}{4}d_{A',A} \rbrace$ we get 
$$\Vert w - v \Vert_{L^\infty(B \cap K)}  \le \frac12  \Lambda(u,v)$$
since $K \cap E_{T_2} =\emptyset$ and thus $B \cap K \subset  D_{\rm good} \setminus E_{T_2}$,  see \eqref{eq: goodDdef}.

}
\end{remark}

\subsection{Proofs of  Lemma \ref{lemma: fundamental estimate2-new} and Corollary \ref{rem: scaling}}\label{sec: fund-proof2}

 In this section we prove the fundamental estimate for piecewise rigid functions with boundary data and present a scaled version as corollary.  We start with the proof of Lemma \ref{lemma: fundamental estimate2-new}.

\begin{proof}[Proof of Lemma \ref{lemma: fundamental estimate2-new}]
Let $A',A,B \in \mathcal{A}_0(\Omega)$ with $A' \subset \subset A$ and $\eta>0$ be given. It is not restrictive to suppose that $0 < \eta <1$. Set $U = A' \cup B$ for brevity.  We define  $d_{A',A} = \dist(\partial A', \partial A)$ and $\delta = (d_{A',A}\alpha\eta/(24\beta c_{\pi,d}))^d$, where $c_{\pi,d}$ denotes the isoperimetric constant in dimension $d$, and $\alpha,\beta$ are the constants from {\rm (${\rm H_4}$)}. Choose $R>0$ such that $U \subset B_R(0)$. Let   $c_0 \ge 1$  be the constant in  \eqref{eq: estimate2},  depending on $R$ and $\delta$. Define  $M_1 = 2c_0$.

Let $u \in PR(A)$, $v \in PR(B)$ be given and let  $u = \sum_j q^u_j \chi_{P_j^u}$ and  $v = \sum_j q^v_j \chi_{P_j^v}$ be their pairwise distinct representations. Suppose that \eqref{eq: extra condition} holds, where  $\Lambda(u,v)$ is the function from \eqref{eq: Lambda0}.   It is not restrictive to assume that $\Lambda (u,v) <+\infty$ is satisfied,  so that in particular  \eqref{eq: closeuv} holds. Otherwise, the result follows exactly as discussed below \eqref{eq: closeuv}.   We apply Lemma \ref{lemma: fundamental estimate2} on $u$ and $v$, and denote by $z \in PR(U)$ the piecewise rigid function satisfying  \eqref{eq: assertionfund}. By recalling Remark \ref{rem: for later} and using $0 < \eta < 1$, we also find
\begin{align}\label{eq: given by before}
{\rm (i)} \ \ & \mathcal{H}^{d-1}(J_z) \le 2\alpha^{-1}\big(\mathcal{F}(u,A) + \mathcal{F}(v,B) + \mathcal{H}^{d-1}(\partial A' \cup \partial A \cup \partial B) \big),\notag\\
{\rm (ii)} \ \ & \Vert z - v \Vert_{L^\infty(B \cap K)}   \le   \tfrac12   \Lambda(u,v),
\end{align} 
where $K :=\lbrace x \in \R^d: \, \dist(x,A') \ge \frac{3}{4}d_{A',A} \rbrace$. 

We let $z = \sum_i q^z_i \chi_{P^z_i} \in PR(U)$ be the corresponding pairwise distinct representation.  We first identify the  \emph{small components} which are given by the sets $(P_i^v \cap P_j^z)_{i,j \in \N}$ of measure smaller than $\delta$ (Step 1).
Then we consider the other components and show, by means of condition \eqref{eq: extra condition}, that for each $P_i^z$ there is at most one component $P_j^v$ such that the measure of $P_i^z \cap P_j^v$ exceeds $\delta$. We prove that the difference of the affine mappings $q_i^z$ and $q_j^v$ can be controlled suitably (Step 2). Starting from $z$, we then define  $w$ where the main idea in the definition is to replace $z$ on each $P^z_i$ by $v$ near $B \setminus A$ and by $q_j^v$ otherwise   (Step 3). This allows to show that the  correct boundary values are attained. Moreover,  the control on  the difference of the affine mappings yields that the energy increases only slightly by passing from $z$ to $w$ (Step 4).

\noindent  \emph{Step 1 (Small components):} Let $(P^{v, z}_k)_k$ be the partition of $B$ consisting of the nonempty sets $P_i^z \cap P_j^v$, $i,j \in \N$. Let $J^{v,z}_{\rm small} = \lbrace k \in \N: \mathcal{L}^d(P^{v,z}_k \cap K)  <  \delta\rbrace$ and $J^{v,z}_{\rm large} = \N \setminus J^{v,z}_{\rm small}$. We define $F_{\rm small} = \bigcup_{k \in J^{v,z}_{\rm small}} P^{v,z}_k$ and observe by \eqref{eq: bdy and jump} that
  \begin{align}\label{eq: Dsmall-bdy} 
\mathcal{H}^{d-1}\big( (\partial^* F_{\rm small} \cap B) \setminus (J_v \cup J_z)\big) = 0. 
\end{align}
 By using the isoperimetric inequality we get  
  \begin{align*}
\mathcal{L}^{d}(F_{\rm small} \cap K) & = \sum\nolimits_{k \in J^{v,z}_{\rm small}}\mathcal{L}^d({P}^{v,z}_k \cap K)  \le \delta^{1/d}\sum\nolimits_{k \in J^{v,z}_{\rm small}}(\mathcal{L}^d({P}^{v,z}_k))^{(d-1)/d}\\
&  \le \delta^{1/d}c_{\pi,d} \sum\nolimits_{k \in J^{v,z}_{\rm small}}\mathcal{H}^{d-1}(\partial^* {P}^{v,z}_k)  
 \\&\le  2\delta^{1/d} c_{\pi,d} \,  \big( \mathcal{H}^{d-1}(J_v) + \mathcal{H}^{d-1}(J_z) + \mathcal{H}^{d-1}(\partial B) \big),
\end{align*} 
where the last step follows from \eqref{eq: bdy and jump} and  Theorem \ref{th: local structure}. Similar to the proof of Lemma \ref{lemma: fundamental estimate}, we cut small components.  For $t>0$ define 
\begin{align*}
E_t := \lbrace x\in \R^d: \ \dist(x,A') < t \rbrace
\end{align*}
and observe that $E_t \cap (U \setminus A) = \emptyset$ for all $t \in (0,d_{A',A})$.  By  repeating the argument leading to \eqref{eq: Et-newe-old},     we find $T \in (\frac{3}{4} d_{A',A} ,d_{A',A})$ such that 
\begin{align}\label{eq: FsmallE}
\mathcal{H}^{d-1}(F_{\rm small} \cap \partial E_{T}) &= \mathcal{H}^{d-1}((F_{\rm small} \cap K) \cap \partial E_{T}) \le 4d_{A',A}^{-1}\mathcal{L}^d(F_{\rm small} \cap K)\notag\\
& \le 8\delta^{1/d} d_{A',A}^{-1} c_{\pi,d} \,  \big( \mathcal{H}^{d-1}(J_v) + \mathcal{H}^{d-1}(J_z) + \mathcal{H}^{d-1}(\partial B) \big).
\end{align}

 \noindent \emph{Step 2 (Large components):} For each $i \in \N$, we define  
\begin{align}\label{eq: Ji-def}
J_i = \Big\{ j\in \N:\, \exists \,  k \in J_{\rm large}^{v,z} \text{ such that } P^{v,z}_k = P_i^z \cap P_j^v \Big\}, 
\end{align}
and observe that for each $i \in \N$
\begin{align}\label{eq: onlyone2}
 \bigcup\nolimits_{j \in J_i} (P_i^z \cap P_j^v) = P_i^z \cap \bigcup\nolimits_{k \in  J_{\rm large}^{v,z}} P^{v,z}_k = P_i^z \cap (B \setminus F_{\rm small}),
\end{align}
where in the last step we used the definition of $F_{\rm small}$ before \eqref{eq: Dsmall-bdy}. We point out that $J_i = \emptyset$ is possible.  In this case, \eqref{eq: onlyone2} still holds because both sides of the equality are empty.  

We now provide some properties of the sets $J_i$. For each $i \in \N$ and each $j \in J_i$, we choose $k \in   J_{\rm large}^{v,z}$ such that $P^{v,z}_k = P_i^z \cap P_j^v$.   Since $U \subset B_R(0)$ and $\mathcal{L}^d(P^{v,z}_k \cap K) \ge \delta$, by \eqref{eq: estimate2} we have  $\Vert q^z_i -q^v_j \Vert_{L^\infty(U)} \le c_0 \Vert q^z_i - q^v_j \Vert_{L^\infty(P^{u,v}_k \cap K)}$.  By using the fact that $v = q^v_j$ and $z = q^z_i$ on $P^{v,z}_k$, by recalling  $M_1 = 2c_0$, and applying \eqref{eq: given by before}(ii) we derive 
\begin{align}\label{eq: fundest7}
\Vert q^z_i -q^v_j \Vert_{L^\infty(U)}\le c_0 \Vert q^z_i - q^v_j \Vert_{L^\infty(P^{u,v}_k \cap K)} \le  c_0 \Vert v - z \Vert_{L^\infty(B \cap K)}  \le  \frac{1}{4}M_1 \Lambda(u,v).
\end{align}
We now show that 
\begin{align}\label{eq: onlyone}
\# J_i \le 1 \ \ \ \ \ \text{for all } i \in \N.
\end{align}
In fact, assume by contradiction that for some $i$ there exist  two different $j,j' \in J_i$.  Then \eqref{eq: fundest7} together with the triangle inequality yields 
$$\Vert q^v_j - q^v_{j'}  \Vert_{L^\infty(U)} \le \frac{1}{2} M_1\Lambda(u,v).$$
Moreover, by \eqref{eq: Ji-def} and the definition of $J_{\rm large}^{v,z}$ we have $\mathcal{L}^d(P^v_j) \ge \delta$ and $\mathcal{L}^d(P^v_{j'}) \ge \delta$.  In view of \eqref{eq: Psi} and the fact that $j \neq j'$, this yields $0< \Phi(A',  U  ; v|_{B \setminus \overline{A'}},\delta) \le \frac{1}{2} M_1\Lambda(u,v)$. This, however, contradicts \eqref{eq: extra condition}.   In the following, the unique index in $J_i$, if existent, will be denoted by $j_i$.

\noindent  \emph{Step 3 (Definition of $w$):}  We now introduce the piecewise rigid function $w$. We define $w:U \to \R^d$ on each $P_i$ separately by distinguishing the two cases $\# J_i = 1$ and $\# J_i = 0$, see \eqref{eq: onlyone}. Recall $E_T$ defined before \eqref{eq: FsmallE} and the fact that $\R^d \setminus E_T \subset K$.    We let
\begin{align}\label{eq: division2}
  w = q^v_{j_i}   \text{ on } P^z_i \cap   E_T, \ \ \ \ w = v  \text{ on } P^z_i \setminus  E_T &  \ \ \ \ \ \text{if } \# J_i=1\notag\\
 w = z  \text{ on } P^z_i \cap   E_T, \ \ \ \ w = v  \text{ on } P^z_i \setminus E_T & \ \ \ \ \  \text{if } \# J_i=0,
\end{align}
where $j_i\in J_i$ is the index corresponding to $i \in \N$. Clearly,  $w \in PR(U)$ is well defined and piecewise rigid  since $v \in PR(B)$ and $U \setminus E_T \subset K \cap B$. For later purposes, we observe that up to sets of negligible $\mathcal{H}^{d-1}$-measure there holds
 \begin{align}\label{eq: K,B} 
{\rm (i)} & \ \  J_w \cap (J_v\setminus J_z) \subset K \cap B, \notag\\
{\rm (ii)} & \ \  J_w \setminus (J_z \cup J_v)  \subset F_{\rm small} \cap \partial E_T,
  \end{align}
where $F_{\rm small} \subset B$ was defined before \eqref{eq: Dsmall-bdy}. Indeed, property (i) follows from \eqref{eq: division2} and the fact that $U \setminus E_T \subset K \cap B$. To see (ii), we first observe that Theorem \ref{th: local structure}, \eqref{eq: bdy and jump}, and  \eqref{eq: division2} imply (up to sets of negligible $\mathcal{H}^{d-1}$-measure)
$$J_w \setminus (J_z \cup J_v)   \subset J_w \cap \partial E_T \cap \bigcup\nolimits_{i \in \N} (P^z_i)^1 \subset (\partial E_T \cap F_{\rm small}) \cup \bigcup\nolimits_{i \in \N} \big(J_w \cap (P^z_i)^1   \cap (\partial E_T \setminus F_{\rm small})\big).$$
 By using \eqref{eq: onlyone2} we have $P^z_i \cap (B \setminus F_{\rm small}) = \emptyset$ if $\# J_i = 0$ and   $P^z_i \cap (B \setminus F_{\rm small}) = P^z_i \cap P^v_{j_i}$ for $\# J_i = 1$.  In view of \eqref{eq: division2}, we also observe that $w$ does not jump on $P^z_i \cap  P^v_{j_i} \cap  \partial E_T$ for $\# J_i = 1$.  In both cases, we thus have $\mathcal{H}^{d-1}(J_w \cap (P^z_i)^1  \cap (\partial E_T \setminus F_{\rm small}) ) = 0$. This yields \eqref{eq: K,B}(ii).

\noindent  \emph{Step 4 (Proof of \eqref{eq: assertionfund-newnew}):} We define
  \begin{align}\label{eq: newTheta-def}
  \Theta(z_1,z_2) =  \big(\frac12 M_1+1\big)\Lambda(z_1,z_2)   \ \ \ \text{for } \ z_1 \in PR(A), \  z_2 \in PR(B),
  \end{align}
  where $\Lambda$ is given in \eqref{eq: Lambda0}. Then, if $\psi(t) = t^p$, $1 \le p < \infty$, $\Theta$ has the form  $\Theta(u,v)= M_2\Vert u - v \Vert_{L^p((A \setminus A') \cap B)}$ for some $M_2$ sufficiently large.

    We now establish \eqref{eq: assertionfund-newnew}.    First,    \eqref{eq: assertionfund-newnew}{\rm (iii)} follows directly from \eqref{eq: division2} and the fact that $B \setminus A = U \setminus A \subset U \setminus E_T$. As a preparation for \eqref{eq: assertionfund-newnew}{\rm (ii)}, we observe that
  \begin{align}\label{eq: assertionfund-lemma-iii}
\Vert    w- z    \Vert_{L^\infty(U)}   \le \frac{1}{4} M_1 \Lambda(u,v).
\end{align}
In fact,   on $U \setminus E_T \subset B \cap K$  we have $w=v$ by \eqref{eq: division2}, hence the inequality holds by  \eqref{eq: given by before}(ii) and the fact that $M_1 \ge 2$.  On the other hand, on each $P^z_i \cap E_T$,  we either have $w=z$, if  $\# J_i = 0$, or we can apply \eqref{eq: fundest7} for $j=j_i$, if $\# J_i = 1$.  In both cases, \eqref{eq: assertionfund-lemma-iii} follows. \EEE This along with \eqref{eq: assertionfund}(ii) (applied for $z$ in place of $w$) and \eqref{eq: newTheta-def} yields   \eqref{eq: assertionfund-newnew}{\rm (ii)}.

Finally, we prove \eqref{eq: assertionfund-newnew}{\rm (i)}. In view of {\rm (${\rm H_1}$)} and  {\rm (${\rm H_4}$)}, we have
\begin{align}\label{eq: Gamma0-new}
\mathcal{F}(w,U) \le \sum\nolimits_{j=1}^3 \mathcal{F}(w, \Gamma_j),
\end{align}
 where we define 
\begin{align*}
\Gamma_1 := J_w  \cap J_z , \ \ \ \ \ \ \ \Gamma_2 := J_w \cap (J_v \setminus J_z),  \ \  \ \ \ \ \  \Gamma_3 := J_w \setminus (J_z \cup J_v). 
\end{align*}
We  estimate each $\mathcal{F}(w, \Gamma_j)$ separately.

(1) For $\mathcal{H}^{d-1}$-a.e.\ $x \in \Gamma_1$, the one-sided approximate limits $w^+(x),w^-(x)$ of $w$ satisfy $|w^+(x) -z^+(x)|,|w^-(x)-z^-(x)| \le \frac{1}{4}  M_1\Lambda(u,v)$ by \eqref{eq: assertionfund-lemma-iii}, where we choose $\nu_w = \nu_z$ on $J_w \cap  J_z$. \BBB This implies $|w^+(x) -z^+(x)|+|w^-(x)-z^-(x)| \le  \frac12  M_1\Lambda(u,v) \le \Theta(u,v) $ for $\mathcal{H}^{d-1}$-a.e.\ $x \in \Gamma_1$, where we used \eqref{eq: newTheta-def}. Thus, by   {\rm (${\rm H_5}^\prime$)} this yields 
\begin{align*}
\begin{split}
\mathcal{F}(w,\Gamma_1)& \le \mathcal{F}(z,\Gamma_1) + \int_{\Gamma_1} \sigma(|w^+ -z^+| + |w^--z^-|) \, d\mathcal{H}^{d-1}\le  \mathcal{F}(z,U) + \mathcal{H}^{d-1}(J_z) \, \sigma(\Theta(u,v)),  
\end{split}
\end{align*}
where $\sigma$ is the modulus of continuity from {\rm (${\rm H_5}^\prime$)}\EEE. Then by   \eqref{eq: assertionfund}(i) (applied for $\mathcal{F}(z,U)$) and \eqref{eq: given by before}(i) we get
\begin{align}\label{eq: Gamma1-new}
\mathcal{F}(w,\Gamma_1)& \le \mathcal{F}(u,A \cap J_z)  + \mathcal{F}(v, B \cap J_z)    \\
&  \ \ \ + \big(\mathcal{H}^{d-1} (\partial A \cup \partial A' \cup \partial B) + \mathcal{F}(u,A)  + \mathcal{F}(v, B) \big) \big(\eta + M\sigma (\Lambda(u,v)) + 2\alpha^{-1}\, \sigma(\Theta(u,v))\big). \notag
\end{align}

(2) In a similar fashion, for $\mathcal{H}^{d-1}$-a.e.\ $x \in \Gamma_2$, we have $|w^+(x) -v^+(x)|,|w^-(x)-v^-(x)| \le \frac{1}{4}(M_1 + 2)\Lambda(u,v)$ by \eqref{eq: given by before}, \eqref{eq: assertionfund-lemma-iii}, and the fact that $\Gamma_2 \subset K \cap B$, see \eqref{eq: K,B}(i). \BBB Thus, we get $|w^+(x) -z^+(x)|+|w^-(x)-z^-(x)| \le \left (\frac12 M_1 + 1\right)\Lambda(u,v) = \Theta(u,v)$ by \eqref{eq: newTheta-def}.    Therefore, we obtain by {\rm (${\rm H_4}$)} and {\rm (${\rm H_5}^\prime$)}
\begin{align}\label{eq: Gamma2-new}
\mathcal{F}(w,\Gamma_2)& \le \mathcal{F}(v,\Gamma_2) + \int_{\Gamma_2} \sigma(|w^+ -z^+|+|w^--z^-|) \, d\mathcal{H}^{d-1}\notag\\
& \le \mathcal{F}(v,\Gamma_2) + \mathcal{H}^{d-1}(J_v)  \, \sigma(\Theta(u,v)) \le  \mathcal{F}(v, B \setminus J_z) +  \alpha^{-1}\mathcal{F}(v,B) \, \sigma(\Theta(u,v)).
\end{align}
\EEE

(3) Finally, \eqref{eq: given by before}(i), \eqref{eq: FsmallE}, \eqref{eq: K,B}(ii),  and {\rm (${\rm H_4}$)}  imply 
\begin{align}\label{eq: Gamma3-new}
\mathcal{F}(w,\Gamma_3)& \le \beta\mathcal{H}^{d-1}\big(F_{\rm small} \cap \partial E_T\big)   \le  8\beta\delta^{1/d} d_{A',A}^{-1} c_{\pi,d} \,  \big( \mathcal{H}^{d-1}(J_v) + \mathcal{H}^{d-1}(J_z) + \mathcal{H}^{d-1}(\partial B) \big)\notag \\
& \le  8\beta\delta^{1/d} d_{A',A}^{-1} c_{\pi,d} \,  3\alpha^{-1}\big(\mathcal{F}(u,A) + \mathcal{F}(v,B) + \mathcal{H}^{d-1}(\partial A' \cup \partial A \cup \partial B) \big)\notag\\ 
& \le  \eta \big(\mathcal{F}(u,A) + \mathcal{F}(v,B) + \mathcal{H}^{d-1}(\partial A' \cup \partial A \cup \partial B) \big),  
\end{align}
where in the last step we used the definition $\delta = (d_{A',A}\alpha\eta/(24\beta c_{\pi,d}))^d$. Define $M_2 =M + 3\alpha^{-1}$ and recall $\Theta(u,v) \ge \Lambda(u,v)$ by \eqref{eq: newTheta-def},  as well as that $\sigma$ is increasing. By combining \eqref{eq: Gamma0-new}--\eqref{eq: Gamma3-new} and using {\rm (${\rm H_1}$)} we obtain  \eqref{eq: assertionfund-newnew}{\rm (i)}. This concludes the proof. 
\end{proof}

We now close this section with the proof of Corollary \ref{rem: scaling}.

\begin{proof}[Proof of Corollary \ref{rem: scaling}]
We suppose that   $A', A, B \in \mathcal{A}_0(\Omega)$ with $A'  \subset \subset  A$ are given such that $\rho A',   \rho A, \rho B \subset \Omega$. Let $U=A' \cup B$. Let $M$ be the constant of Lemma \ref{lemma: fundamental estimate2} and  $M_1$, $M_2$, $\delta$ be the constants of Lemma \ref{lemma: fundamental estimate2-new}  (applied for $\psi(t) = t$). For brevity, set $C_{A',A,B}= \mathcal{H}^{d-1}(\partial A' \cup \partial A \cup \partial B)$.

Given $\mathcal{F}: PR(\Omega) \times \mathcal{B}(\Omega) \to [0,\infty)$ satisfying {\rm (${\rm H_1}$)}, \BBB {\rm (${\rm H_3}$)}--{\rm (${\rm H_4}$)} and {\rm (${\rm H_5}^\prime$)}\EEE, we define $\mathcal{F}^\rho: PR(\rho^{-1}\Omega) \times \mathcal{B}(\rho^{-1}\Omega) \to [0,\infty)$ by
\begin{align}\label{eq: scalin1}
\mathcal{F}^\rho(z,B) = \rho^{-(d-1)} \mathcal{F}(z_\rho,\rho B) 
\end{align}
for all $z \in PR(\rho^{-1}\Omega)$ and $B \in \mathcal{B}(\rho^{-1}\Omega)$, where $z_\rho \in PR(\Omega)$ is defined by $z_\rho(x):= z(x/\rho)$. Then it is elementary to check that also $\mathcal{F}^\rho$ satisfies {\rm (${\rm H_1}$)}, \BBB {\rm (${\rm H_3}$)}--{\rm (${\rm H_4}$)} and {\rm (${\rm H_5}^\prime$)}.\EEE

Let $u_\rho \in PR(\rho A)$ and $v_\rho \in PR(\rho B)$. We define $u \in PR(A)$ by $u(x) = u_\rho(\rho x)$ and $v \in PR(B)$ by $v(x) = v_\rho(\rho x)$. Note that a scaling argument yields
\begin{align}\label{eq: scalin2}
\rho^{-d}\Vert u_\rho - v_\rho \Vert_{L^1(\rho(A\setminus  A')\cap \rho B)} = \Vert u - v \Vert_{L^1((A\setminus  A')\cap B)}.
\end{align}
 Assumption \eqref{eq: extra condition-new} and  \eqref{eq: scalin2} imply
\begin{align*}
MM_1\Vert u - v \Vert_{L^1((A\setminus  A')\cap B)} & =  \rho^{-d}MM_1\Vert u_\rho - v_\rho \Vert_{L^1((\rho A\setminus \rho A')\cap \rho B)}\\
& \le   \Phi(\rho A', \, \rho A'\cup \rho B; \, v_\rho|_{\rho B \setminus \rho \overline{A'}}, \, \rho^d \delta) =  \Phi(A', A'\cup B;  v  |_{B \setminus  \overline{A'}}, \delta). 
\end{align*}
We apply Lemma \ref{lemma: fundamental estimate2-new} on $u$ and $v$ for $\psi(t)=t$ and $\mathcal{F}^\rho$, where we note that 
in this case $\Lambda(z_1,z_2) = M\Vert z_1-z_2 \Vert_{L^1((A \setminus A')\cap B)}$, see Lemma \ref{lemma: fundamental estimate2}.  We obtain $w \in PR(A' \cup B)$ such that
\begin{align*}
{\rm (i)}& \ \ \mathcal{F}^\rho(w, A' \cup B) \le  \mathcal{F}^\rho(u,A)  +  \mathcal{F}^\rho(v, B) \notag \\
 & \ \ \ \ \ \ \ \ \  \ \ \ \ \ \ \ \ \ \ \ + \big(C_{A',A,B} + \mathcal{F}^\rho(u,A)  + \mathcal{F}^\rho(v, B) \big) \Big(2\eta + M_2\sigma\big(M_2\Vert u - v\Vert_{L^1((A\setminus A')\cap B)}\big) \Big), \notag \\ 
{\rm (ii)}& \ \   \BBB \Vert \min\lbrace |w - u|, |w-v| \rbrace \Vert_{L^\infty(A' \cup B)} \EEE \le M_2 \Vert u - v \Vert_{L^1( (  A\setminus   A')\cap B)},\notag\\
{\rm (iii)}& \ \ w = v \text{ on } B\setminus A. 
\end{align*}
Define  $w_\rho \in PR(\rho A' \cup \rho B)$ by $w_\rho(x) = w( x/\rho)$. Then \eqref{eq: extra condition-new-new} follows from the estimates on $w$ along with \eqref{eq: scalin1}--\eqref{eq: scalin2}. 
\end{proof}

\section{Integral representation in $PR(\Omega)$}\label{sec: representation}

This section is devoted to the proof of Theorem \ref{theorem: PR-representation}.  In Section \ref{sec: representation1} we show how Theorem \ref{theorem: PR-representation} can be deduced from two auxiliary lemmas whose proofs are given  in Section \ref{sec: representation2}. In Section \ref{sec: representation3}  we also present a generalization which will be instrumental in Section \ref{sec: gamma}.

\subsection{Proof of Theorem \ref{theorem: PR-representation}}\label{sec: representation1}

Let  $\mathcal{F}: PR(\Omega) \times \mathcal{A}(\Omega) \to [0,\infty)$ and $u \in PR(\Omega)$. We first state that $\mathcal{F}$ is equivalent to $\mathbf{m}_{\mathcal{F}}$ (see \eqref{eq: general minimization}) in the sense that the two quantities have the same Radon-Nykodym derivative with respect to $\mathcal{H}^{d-1}\lfloor_{J_u \cap \Omega}$.

\begin{lemma}\label{lemma: G=m}
Suppose that $\mathcal{F}$ satisfies {\rm (${\rm H_1}$)}--{\rm (${\rm H_4}$)}. Let $u \in PR(\Omega)$ and   $\mu = \mathcal{H}^{d-1}\lfloor_{J_u \cap \Omega}$. Then for $\mu$-a.e.\ $x_0 \in \Omega$ we have
 $$\lim_{\eps \to 0}\frac{\mathcal{F}(u,B_\eps(x_0))}{\mu(B_\eps(x_0))} =  \lim_{\eps \to 0}\frac{\mathbf{m}_{\mathcal{F}}(u,B_\eps(x_0))}{\mu(B_\eps(x_0))}.$$
\end{lemma}

We defer the proof of Lemma \ref{lemma: G=m} to Section \ref{sec: representation2}. 
The second ingredient is that, asymptotically as $\eps \to 0$, the minimization problems $\mathbf{m}_{\mathcal{F}}(u,B_\eps(x_0))$ and $\mathbf{m}_{\mathcal{F}}(\bar{u}_{x_0},B_\eps(x_0))$ coincide for $\mathcal{H}^{d-1}$-a.e.\ $x_0 \in J_u$, where we write $\bar{u}_{x_0} := u_{x_0,[u](x_0),\nu_u(x_0)}$ for brevity, see \eqref{eq: jump competitor}.

\begin{lemma}\label{lemma: minsame}
Suppose that $\mathcal{F}$ satisfies {\rm (${\rm H_1}$)} and  {\rm (${\rm H_3}$)}--{\rm (${\rm H_5}$)}. Then for $\mathcal{H}^{d-1}$-a.e.\ $x_0 \in J_u$  we have
\begin{align}\label{eq: PR-proof1}
  \lim_{\eps \to 0}\frac{\mathbf{m}_{\mathcal{F}}(u,B_\eps(x_0))}{\omega_{d-1}\eps^{d-1}} =  \limsup_{\eps \to 0}\frac{\mathbf{m}_{\mathcal{F}}(\bar{u}_{x_0},B_\eps(x_0))}{\omega_{d-1}\eps^{d-1}}. 
  \end{align}
\end{lemma}

We   defer the proof of Lemma \ref{lemma: minsame} also to Section \ref{sec: representation2}  and now proceed to  prove Theorem \ref{theorem: PR-representation}.

\begin{proof}[Proof of Theorem \ref{theorem: PR-representation}]
We need to show that for $\mathcal{H}^{d-1}$-a.e.\ $x_0 \in J_u$ one has 
$$\frac{\mathrm{d}\mathcal{F}(u,\cdot)}{\mathrm{d}\mathcal{H}^{d-1}\lfloor_{J_u}}(x_0) = f(x_0,[u](x_0),\nu_u(x_0)),$$
 where $f$ was defined in \eqref{eq:gdef}.  By Lemma \ref{lemma: G=m} and the fact that  $\lim_{\eps \to 0} (\omega_{d-1}\eps^{d-1})^{-1}\mu(B_\eps(x_0))=1$ for $\mathcal{H}^{d-1}$-a.e.\ $x_0 \in J_u$  we deduce 
 $$\frac{\mathrm{d}\mathcal{F}(u,\cdot)}{\mathrm{d}\mathcal{H}^{d-1}\lfloor_{J_u}}(x_0)  = \lim_{\eps \to 0}\frac{\mathcal{F}(u,B_\eps(x_0))}{\mu(B_\eps(x_0))} =  \lim_{\eps \to 0}\frac{\mathbf{m}_{\mathcal{F}}(u,B_\eps(x_0))}{\mu(B_\eps(x_0))}  = \lim_{\eps \to 0}\frac{\mathbf{m}_{\mathcal{F}}(u,B_\eps(x_0))}{\omega_{d-1}\eps^{d-1}} < \infty$$
 for $\mathcal{H}^{d-1}$-a.e.\ $x_0 \in J_u$. The statement now follows   from \eqref{eq:gdef} and  Lemma \ref{lemma: minsame}.
 \end{proof}

\subsection{Proof of Lemmata \ref{lemma: G=m} and   \ref{lemma: minsame}}\label{sec: representation2}

For the proof of Lemma \ref{lemma: G=m}, we basically follow the lines  of \cite{BFLM, BDM, Conti-Focardi-Iurlano:15}, with the difference that the required compactness results are more delicate due to the weaker growth condition from below (see {\rm (${\rm H_4}$)}) compared to   \cite{BFLM, BDM, Conti-Focardi-Iurlano:15}. We start with some notation. We set $c_d$ as the dimensional constant
\[
c_d:=\frac12 \frac{\omega_{d-1}}{d\omega_d}\,.
\]
For $\delta>0$ and $A \in \mathcal{A}(\Omega)$  we define 
\begin{align}\label{eq: delta-def}
\mathbf{m}^\delta_{\mathcal{F}}(u,A) = \inf\Big\{  &\sum\nolimits_{i=1}^\infty \mathbf{m}_{\mathcal{F}}(u,B_i):  \ B_i \subset A   \text{ pairwise disjoint balls}, \,  \diam(B_i) \le \delta, \notag\\
 & \mathcal{H}^{d-1}(B_i \cap J_u)\ge c_d \mathcal{H}^{d-1}(\partial B_i),\   \mathcal{H}^{d-1}\Big( J_u \cap \Big( A \setminus \bigcup\nolimits_{i=1}^\infty B_i \Big)\Big) = 0\Big\} 
\end{align}
and, as $\mathbf{m}^\delta_{\mathcal{F}}(u,A) $ is decreasing in $\delta$, we can also introduce
\begin{align}\label{eq: m*}
\mathbf{m}^*_{\mathcal{F}}(u,A) =  \lim\nolimits_{\delta \to 0 } \mathbf{m}^\delta_{\mathcal{F}}(u,A).
\end{align}
Notice that the existence of coverings as in \eqref{eq: delta-def} follows from the Morse covering theorem (see, e.g., \cite[Theorem 1.147]{fonseca.leoni}),  provided one  observes  that at $\mathcal{H}^{d-1}$-a.e.\ $x\in J_u$,  there holds by rectifiability
\[
\lim_{\delta\to 0}\frac{\mathcal{H}^{d-1}(J_u\cap B_\delta(x))}{\mathcal{H}^{d-1}(\partial B_\delta(x))}=2c_d\,.
\]

 \begin{lemma}\label{lemma: flaviana}
Suppose that $\mathcal{F}$ satisfies {\rm (${\rm H_1}$)}, {\rm (${\rm H_3}$)}--{\rm (${\rm H_4}$)}. Let $u \in PR(\Omega)$ and   $\mu = \mathcal{H}^{d-1}\lfloor_{J_u \cap \Omega}$. If  $\mathcal{F}(u,A) = \mathbf{m}^*_{\mathcal{F}}(u,A)$ for all $A \in \mathcal{A}(\Omega)$, then for $\mu$-a.e.\ $x_0 \in \Omega$ we have
 $$\lim_{\eps \to 0}\frac{\mathcal{F}(u,B_\eps(x_0))}{\mu(B_\eps(x_0))} =  \lim_{\eps \to 0}\frac{\mathbf{m}_{\mathcal{F}}(u,B_\eps(x_0))}{\mu(B_\eps(x_0))}.$$
\end{lemma}

\begin{proof}
The statement follows by repeating exactly the arguments in   \cite[Proofs of Lemma 4.2 and Lemma 4.3]{Conti-Focardi-Iurlano:15}. \BBB Note that the assumption $\mathcal{F}(u,A) = \mathbf{m}^*_{\mathcal{F}}(u,A)$ enters the proof at the very end of \cite[Proof of Lemma 4.3]{Conti-Focardi-Iurlano:15} and replaces the application of \cite[Lemma 4.1]{Conti-Focardi-Iurlano:15}. \EEE
\end{proof}

In view of Lemma \ref{lemma: flaviana}, in order to see that $\mathcal{F}$ and $\mathbf{m}_{\mathcal{F}}$  have the same Radon-Nykodym derivative with respect to $\mathcal{H}^{d-1}\lfloor_{J_u \cap \Omega}$, it remains to show the following.

\begin{lemma}\label{lemma: G=m-2}
Suppose that $\mathcal{F}$ satisfies {\rm (${\rm H_1}$)}--{\rm (${\rm H_4}$)}. Then for all $u \in PR(\Omega)$ and all $A \in \mathcal{A}(\Omega)$ there holds $\mathcal{F}(u,A) = \mathbf{m}^*_{\mathcal{F}}(u,A)$. 
\end{lemma}

\begin{proof}
We follow the proof of \cite[Lemma 4.1]{Conti-Focardi-Iurlano:15} and only indicate the necessary changes. For each ball $B \subset A$ we have $\mathbf{m}_{\mathcal{F}}(u,B) \le \mathcal{F}(u,B)$ by definition. By {\rm (${\rm H_1}$)}  we  get $\mathbf{m}_{\mathcal{F}}^\delta(u,A) \le \mathcal{F}(u,A)$ for all $\delta>0$. This shows $\mathbf{m}_{\mathcal{F}}^*(u,A) \le \mathcal{F}(u,A)$, see \eqref{eq: m*}. 

We now address the reverse inequality.  We fix $A\in\mathcal{A}(\Omega)$ and  $\delta >0$. Let $(B^\delta_i)_i$ be balls as in the definition of $\mathbf{m}_{\mathcal{F}}^\delta(u,A)$ such that
\begin{align}\label{eq: to show-flaviana1}
\sum\nolimits_{i=1}^\infty \mathbf{m}_{\mathcal{F}}(u,B^\delta_i) \le \mathbf{m}_{\mathcal{F}}^\delta(u,A) + \delta.
\end{align}
By the definition of $\mathbf{m}_{\mathcal{F}}$, we find $v_i^\delta \in PR(B_i^\delta)$ such that $v_i^\delta = u$ in a neighborhood of $\partial B_i^\delta$ and 
\begin{align}\label{eq: new-GM2}
\mathcal{F}(v_i^\delta,B_i^\delta) \le \mathbf{m}_{\mathcal{F}}(u,B_i^\delta) + \delta \mathcal{L}^d(B_i^\delta).
\end{align}
We define 
\begin{align*}
v^\delta = \sum\nolimits_{i=1}^\infty v^\delta_i \chi_{B^\delta_i} + u \chi_{N_0^\delta}, 
\end{align*}
 where $N_0^\delta := \Omega \setminus \bigcup_{i=1}^\infty B^\delta_i$.  By \eqref{eq: to show-flaviana1}--\eqref{eq: new-GM2} and {\rm (${\rm H_4}$)} we get $\mathcal{H}^{d-1}(J_{v^\delta})<+\infty$. Thus,  by construction, we obtain $v^\delta \in PR(\Omega)$. Moreover, by (${\rm H_1}$), (${\rm H_3})$, and \eqref{eq: to show-flaviana1}--\eqref{eq: new-GM2} we have  
\begin{align}\label{eq: to show-flaviana2}
\mathcal{F}(v^\delta,A) &= \sum\nolimits_{i=1}^\infty\mathcal{F}(v_i^\delta,B_i^\delta)  + \mathcal{F}(u,N_0^\delta \cap A) \le  \sum\nolimits_{i=1}^\infty \big(\mathbf{m}_{\mathcal{F}}(u,B_i^\delta) + \delta \mathcal{L}^d(B_i^\delta)\big) \notag\\
&\le \mathbf{m}_{\mathcal{F}}^\delta(u,A) + \delta(1+\mathcal{L}^d(A)),    
\end{align}
where we also used  the fact that $\mathcal{H}^{d-1}(J_u \cap N_0^\delta \cap A)  =  \mathcal{F}(u,N_0^\delta \cap A) = 0$ by the definition of $(B^\delta_i)_i$  and (${\rm H_4})$.  We now  claim that $v^\delta \to u$ in measure. To prove this,  it suffices to show that
\[ 
\sum\nolimits_{i=1}^\infty\mathcal{L}^d(B_i^\delta) \to 0
\]
as $\delta \to 0$. The above limit ensues from the definition of the covering $(B^\delta_i)_i$, the isoperimetric inequality, and \eqref{eq: delta-def}, which yield \EEE 
\[
\begin{split}
\sum\nolimits_{i=1}^\infty\mathcal{L}^d(B_i^\delta) \le \sum\nolimits_{i=1}^\infty\mathcal{L}^d(B_i^\delta)^{\frac{1}d}\mathcal{L}^d(B_i^\delta)^{\frac{d-1}d}\le  c_{\pi,d}  \,  \delta \sum\nolimits_{i=1}^\infty\mathcal{H}^{d-1}(\partial B_i^\delta)\\
\le  \frac{ c_{\pi,d}  }{c_d}\, \delta \sum\nolimits_{i=1}^\infty\mathcal{H}^{d-1}( B_i^\delta  \cap J_u)\le  \frac{ c_{\pi,d}  }{c_d}\, \delta \,\mathcal{H}^{d-1}(J_u) \to 0,
\end{split}
\]
 where $c_{\pi,d}$ denotes the isoperimetric constant.  With this, using (${\rm H_2}$),  \eqref{eq: m*},  and  \eqref{eq: to show-flaviana2} we get the required inequality  $\mathbf{m}_{\mathcal{F}}^*(u,A) \ge \mathcal{F}(u,A)$   in the limit as $\delta \to 0$.
 This concludes the proof.
\end{proof}

\begin{proof}[Proof of Lemma \ref{lemma: G=m}.] 
The combination of Lemma \ref{lemma: flaviana} and Lemma \ref{lemma: G=m-2}  yields the result.
\end{proof}

We now turn our attention to Lemma \ref{lemma: minsame}. Our goal is to show that, asymptotically as $\eps \to 0$, the minimization problems $\mathbf{m}_{\mathcal{F}}(u,B_\eps(x_0))$ and $\mathbf{m}_{\mathcal{F}}(\bar{u}_{x_0},B_\eps(x_0))$ coincide for $\mathcal{H}^{d-1}$-a.e.\ $x_0 \in J_u$. Essentially, the argument relies on Lemma \ref{lemma: fundamental estimate2}, which allows us to join two piecewise rigid functions, and some properties of  piecewise rigid functions, see Lemma \ref{lemma: blow up}.

\begin{proof}[Proof of Lemma \ref{lemma: minsame}]
It suffices  to prove  \eqref{eq: PR-proof1}  for points $x_0 \in J_u$ where the statement of Lemma \ref{lemma: blow up} holds.

\emph{Step 1 (Inequality ``$\le$'' in \eqref{eq: PR-proof1}):}  We  fix $\eta>0$ and $\theta>0$.  Choose $z_\eps \in PR(B_{(1-3 \theta)\eps}(x_0))$ with $z_\eps = \bar{u}_{x_0}$ in a neighborhood of $\partial B_{(1-3\theta)\eps}(x_0)$ and
\begin{align}\label{eq:repr1+}
\mathcal{F}(z_\eps,B_{(1-3\theta)\eps}(x_0)) \le \mathbf{m}_{\mathcal{F}}(\bar{u}_{x_0},B_{(1-3\theta)\eps}(x_0)) + \eps^d.
\end{align}
We extend $z_\eps$ to a function in $PR(B_\eps(x_0))$ by setting  $z_\eps = \bar{u}_{x_0}$ outside $B_{(1-3\theta)\eps}(x_0)$. Let $(u_\eps)_\eps$ be the sequence given by Lemma \ref{lemma: blow up}.   We now want to apply Corollary \ref{rem: scaling} on $z_\eps$ (in place of $u_\rho$) and $u_\eps$ (in place of $v_\rho$) for $\eta$, $\rho= \eps$, $A' = B_{1-2\theta}(x_0)$,  $A = B_{1-\theta}(x_0)$, and $B = B_1(x_0) \setminus  \overline{B_{1-4\theta}(x_0)} $. 

To be in a position for applying  Corollary \ref{rem: scaling},  we must first check that in fact \eqref{eq: extra condition-new} holds for $\eps$ sufficiently small.  Let $\delta$ be the constant provided by Lemma \ref{lemma: fundamental estimate2-new}.
Now, for the given $x_0 \in J_u$, consider the components $P_i$ and $P_j$ provided  by Lemma  \ref{lemma: blow up} satisfying $x_0 \in \partial^* P_i \cap \partial^* P_j$.  Note that $u_\eps= q_i  \chi_{P_i}+ q_j \chi_{P_j}$ on $\eps A$, see \eqref{eq: blow up-new}(iii). Notice that $P_i \cup P_j$ might not form a Caccioppoli partition of $\eps A'\cup \eps B$. However, the remaining  components contained in  $(\eps A'\cup \eps B) \setminus (P_i \cup P_j)$, if nonempty, do not belong to the  index  set $J$ in \eqref{eq: touching cond} (with $\eps^d \delta$ in place of $\delta$, cf.\ \eqref{eq: extra condition-new})  for small values of $\eps$. Indeed, \eqref{eq: blow up-new}(i) implies
\[
\lim_{\eps \to 0} \frac{\mathcal{L}^d\left( (\eps A'\cup \eps B) \setminus (P_i \cup P_j)\right)}{\eps^d}=0\,.
\]
Hence, $J$ contains at most the indices $i$ and $j$.   Now, on the one hand, we find $\Vert q_i- q_j \Vert_{L^\infty(\eps A'\cup \eps B)} \ge |[u(x_0)]|/2$ for $\eps$ sufficiently small. By \eqref{eq: Psi} and \eqref{eq: touching cond} this yields 
$$\Phi\big(\eps A',\eps A'\cup \eps B; u_\eps|_{\eps B \setminus \eps \overline{A'}},\eps^d \delta\big)\ge |[u(x_0)]|/2$$ 
for $\eps$ sufficiently small.  On the other hand, \eqref{eq: blow up-new}(vi) and the fact that $z_\eps = \bar{u}_{x_0}$ on $\eps(A\setminus A')$ imply
\begin{align}\label{eq:repr8}
\lim_{\eps \to 0}\frac{1}{\eps^d}\int_{\eps (A\setminus A')} |z_\eps - u_\eps|\,\mathrm{d}x = 0. 
\end{align}
This shows that  \eqref{eq: extra condition-new} holds  with $z_\eps$ in place of $u_\rho$ and $u_\eps$ in place of $v_\rho$, for $\eps$ sufficiently small.

By  \eqref{eq: extra condition-new-new} there exist  functions  $w_\eps \in PR(B_\eps(x_0))$ such that $w_\eps = u_\eps$ on $B_\eps(x_0) \setminus B_{(1-\theta)\eps}(x_0)$ and 
\begin{align*}
\mathcal{F}&(w_\eps, B_\eps(x_0)) \le   \mathcal{F}(z_\eps,\eps A)  + \mathcal{F}(u_\eps, \eps B) \\
&    + (\mathcal{F}(z_\eps,\eps A)  + \mathcal{F}(u_\eps, \eps B) + 3d\omega_d\eps^{d-1}) \cdot   \big(2\eta +  M_2  \sigma\big( \eps^{-d} M_2  \Vert z_\eps - u_\eps \Vert_{L^1(\eps (A \setminus A'))}  \big)\big),
\end{align*}
where  $M_2$ is the constant of Lemma \ref{lemma: fundamental estimate2-new}. In particular, $w_\eps = u$ in a neighborhood of $\partial B_\eps(x_0)$ by \eqref{eq: blow up-new}(iv).  Using \eqref{eq:repr8} and the fact that $\lim_{t \to 0}\sigma(t)=0$ we find a sequence $(\rho_\eps)_\eps$ with $\rho_\eps \to 0$ such that 
\begin{align}\label{eq: rep2}
\mathcal{F}(w_\eps, B_\eps(x_0)) \le   (1 + 2\eta + \rho_\eps)\big(\mathcal{F}(z_\eps,\eps A)  + \mathcal{F}(u_\eps, \eps B)\big) +  3  d\omega_d\eps^{d-1} (  2  \eta + \rho_\eps).
\end{align}
Using that  $z_\eps = \bar{u}_{x_0}$ on  $B_{\eps}(x_0) \setminus B_{(1-3\theta)\eps}(x_0) \subset \eps B $,  {\rm (${\rm H_1}$)},  {\rm (${\rm H_4}$)},  \eqref{eq: jump competitor}, and \eqref{eq:repr1+} we compute  
\begin{align}\label{eq: rep1}
\limsup_{\eps\to 0}\frac{\mathcal{F}(z_\eps,\eps A)}{\eps^{d-1}} &\le \limsup_{\eps\to 0} \frac{\mathcal{F}(z_\eps,B_{(1- 3 \theta)\eps}(x_0))}{\eps^{d-1}}+  \limsup_{\eps\to 0} \frac{\mathcal{F}(\bar{u}_{x_0},\eps B)}{\eps^{d-1}}\notag\\
&  \le  \limsup_{\eps\to 0} \frac{\mathbf{m}_{\mathcal{F}}(\bar{u}_{x_0},B_{(1-3\theta)\eps}(x_0))}{\eps^{d-1}} + \omega_{d-1}\left[1-(1-4\theta)^{d-1}\right]\beta \notag \\
&\le  (1-3\theta)^{d-1}\limsup_{\eps'\to 0} \frac{\mathbf{m}_{\mathcal{F}}(\bar{u}_{x_0},B_{\eps'}(x_0))}{(\eps')^{d-1}} + \omega_{d-1}\left[1-(1-4\theta)^{d-1}\right]\beta,
\end{align}
where in the  step we substituted $(1-3\theta)\eps$ by $\eps'$. By \eqref{eq: blow up-new}(ii),(v), {\rm (${\rm H_4}$)}, and $B = B_1(x_0) \setminus  \overline{B_{1-4\theta}(x_0)}$ we also find
\begin{align}\label{eq: rep1XXX}
\limsup_{\eps \to 0} \frac{\mathcal{F}(u_\eps, \eps B)}{\eps^{d-1}} \le \limsup_{\eps \to 0} \frac{\mathcal{F}(u, \eps B) }{\eps^{d-1}}\le \omega_{d-1}\left[1-(1-4\theta)^{d-1}\right]\beta. 
\end{align}
Recall that  $w_\eps = u$ in a neighborhood of $\partial B_\eps(x_0)$. This together with  \eqref{eq: rep2}--\eqref{eq: rep1XXX} and  $\rho_\eps \to 0$ yields
\begin{align}\label{eq: cocnlusi}
 \lim\nolimits_{\eps \to 0}\frac{\mathbf{m}_{\mathcal{F}}(u,B_\eps(x_0))}{\omega_{d-1}\eps^{d-1}} &\le  \limsup\nolimits_{\eps \to 0}\frac{\mathcal{F}(w_\eps, B_\eps(x_0))}{\omega_{d-1}\eps^{d-1}} \notag \\
 & \le (1+2\eta) \Big( (1-3\theta)^{d-1} \limsup\nolimits_{\eps \to 0}  \frac{\mathbf{m}_{\mathcal{F}}(\bar{u}_{x_0},B_\eps(x_0))}{\omega_{d-1}\eps^{d-1}}\Big) \notag\\
& \ \  \ + 2(1+2\eta) \left[1-(1-4\theta)^{d-1}\right]\beta+6d\frac{\omega_d}{\omega_{d-1}}\eta.
\end{align}
Passing to $\eta,\theta \to 0$ we obtain  inequality ``$\le$'' in \eqref{eq: PR-proof1}.

\emph{Step 2 (Inequality ``$\ge$'' in \eqref{eq: PR-proof1}):} \BBB Let $(u_\eps)_\eps$ be again the sequence from Lemma \ref{lemma: blow up}.  Since $u_\eps = u$ in a neighborhood of $\partial B_{\eps}(x_0)$ by \eqref{eq: blow up-new}(iv), we get 
\begin{align}\label{eq: auxiliary step}
\mathbf{m}_{\mathcal{F}}(u_\eps,B_\eps(x_0))=   \mathbf{m}_{\mathcal{F}}(u,B_\eps(x_0)) 
\end{align}
for all $\eps >0$. \EEE
With \eqref{eq: auxiliary step} at hand, the proof is now very similar to Step 1, and   we only indicate the main changes. Fix $\eta>0$, $\theta>0$, and choose $z_\eps \in PR(B_{(1- 3 \theta)\eps}(x_0))$ with $z_\eps = u_\eps$ in a neighborhood of $\partial B_{(1-3\theta)\eps}(x_0)$ such that
\begin{align}\label{eq:repr1-new}
\mathcal{F}(z_\eps,B_{(1-3\theta)\eps}(x_0)) \le \mathbf{m}_{\mathcal{F}}(u_\eps,B_{(1-3\theta)\eps}(x_0)) + \eps^d.
\end{align}
We extend $z_\eps$ to a function in $PR(B_\eps(x_0))$ by setting  $z_\eps = u_\eps$ outside $B_{(1-3\theta)\eps}(x_0)$. We apply Corollary \ref{rem: scaling} on $z_\eps$ (in place of $u_\rho$) and $\bar{u}_{x_0}$ (in place of $v_\rho$) for the same sets as in Step 1.  We observe $\Phi(\eps A',\eps A'\cup \eps B; \bar{u}_{x_0}|_{\eps B \setminus \eps \overline{A'}},\eps^d \delta) = |[u(x_0)]|$ and, as $z_\eps = u_\eps$ on $\eps(A\setminus A')$,  \eqref{eq: blow up-new}(vi) yields 
\begin{align}\label{eq:repr8-new}
\lim_{\eps \to 0}\frac{1}{\eps^d}\int_{\eps (A\setminus A')} |z_\eps - \bar{u}_{x_0}|\,\mathrm{d}x = 0. 
\end{align}
Thus,  \eqref{eq: extra condition-new} holds for $\eps$ sufficiently small. By  \eqref{eq: extra condition-new-new} there exist functions  $w_\eps \in PR(B_\eps(x_0))$ such that $w_\eps = \bar{u}_{x_0}$ on $B_\eps(x_0) \setminus B_{(1-\theta)\eps}(x_0)$ and 
\begin{align*}
\mathcal{F}&(w_\eps, B_\eps(x_0)) \le   \mathcal{F}(z_\eps,\eps A)  + \mathcal{F}(\bar{u}_{x_0}, \eps B) \\
&    + (\mathcal{F}(z_\eps,\eps A)  + \mathcal{F}( \bar{u}_{x_0},  \eps B) + 3d\omega_d\eps^{d-1}) \cdot   \big(2\eta +  M_2  \sigma\big( \eps^{-d}  M_2  \Vert z_\eps -  \bar{u}_{x_0}  \Vert_{L^1(\eps (A \setminus A'))}  \big)\big).
\end{align*}
Similar to Step 1, cf.\ \eqref{eq: rep2},  using \eqref{eq:repr8-new}  we find a sequence $(\rho_\eps)_\eps$ with $\rho_\eps \to 0$ such that 
\begin{align}\label{eq: rep2-new}
\mathcal{F}(w_\eps, B_\eps(x_0)) \le   (1 + 2\eta + \rho_\eps)\big(\mathcal{F}(z_\eps,\eps A)  + \mathcal{F}(\bar{u}_{x_0}, \eps B)\big) +  3  d\omega_d\eps^{d-1} (  2  \eta + \rho_\eps).
\end{align}
Repeating the arguments  in  \eqref{eq: rep1}--\eqref{eq: rep1XXX},  in particular using that  $z_\eps = u_\eps$ on  $B_{\eps}(x_0) \setminus B_{(1-3\theta)\eps}(x_0)$,    and  using  \eqref{eq:repr1-new} we derive  
\begin{align}\label{eq: rep1-new}
\limsup_{\eps\to 0}\frac{\mathcal{F}(z_\eps,\eps A)}{\eps^{d-1}}\le (1- 3  \theta)^{d-1}\limsup_{\eps\to 0} \frac{\mathbf{m}_{\mathcal{F}}(u_\eps,B_{\eps}(x_0))}{\eps^{d-1}} + \omega_{d-1}\left[1-(1-4\theta)^{d-1}\right]\beta.
\end{align}
Estimating $\mathcal{F}(\bar{u}_{x_0}, \eps B)$ as in \eqref{eq: rep1}, with \eqref{eq: rep2-new}--\eqref{eq: rep1-new} and     $\rho_\eps \to 0$ we then obtain   
\begin{align*}
 \limsup\nolimits_{\eps \to 0}\frac{\mathcal{F}(w_\eps, B_\eps(x_0))}{\omega_{d-1}\eps^{d-1}} & 
 \le (1+2\eta) \Big( (1-3\theta)^{d-1} \limsup\nolimits_{\eps \to 0}  \frac{\mathbf{m}_{\mathcal{F}}(u_\eps,B_\eps(x_0))}{\omega_{d-1}\eps^{d-1}}\Big) \\
& \ \ \ + 2(1+2\eta) \left[1-(1-4\theta)^{d-1}\right]\beta+6d\frac{\omega_d}{\omega_{d-1}}\eta.
\end{align*}
Passing to $\eta,\theta \to 0$ and recalling that  $w_\eps = \bar{u}_{x_0}$ in a neighborhood of $\partial B_\eps(x_0)$ we derive 
$$\limsup_{\eps \to 0}\frac{\mathbf{m}_{\mathcal{F}}(\bar{u}_{x_0},B_\eps(x_0))}{\omega_{d-1}\eps^{d-1} }\le   \limsup_{\eps \to 0} \frac{\mathbf{m}_{\mathcal{F}}(u_\eps,B_\eps(x_0))}{ \omega_{d-1}\eps^{d-1}  }.$$
 This along with \eqref{eq: auxiliary step}  shows   inequality ``$\ge$'' in \eqref{eq: PR-proof1}.
 \end{proof}

\subsection{A useful generalization}\label{sec: representation3} 
We now formulate a generalization of Theorem \ref{theorem: PR-representation} which will be instrumental in Section \ref{sec: gamma} below. Suppose that we have a sequence of functionals $\mathcal{F}_n: PR(\Omega) \times \mathcal{B}(\Omega) \to [0,\infty)$  satisfying {\rm (${\rm H_1}$)}, {\rm (${\rm H_3}$)}--{\rm (${\rm H_5}$)} \emph{uniformly}, i.e., for the same $0 <\alpha < \beta$ and $\sigma: [0,+\infty) \to [0,\beta]$. 

Let $\mathcal{F}:  PR(\Omega)\times \mathcal{B}(\Omega) \to [0,\infty]$ be a functional satisfying {\rm (${\rm H_1}$)}--{\rm (${\rm H_4}$)}. Later, we will show that these conditions will be guaranteed when $\mathcal{F}$ is the $\Gamma$-limit of the sequence $(\mathcal{F}_n)_n$. If we additionally assume \eqref{eq: condition-new-new}, we have the following generalization of Theorem \ref{theorem: PR-representation}.

 \begin{corollary}\label{corollary: PR-representation-tilde}
Consider a sequence $(\mathcal{F}_n)_n$ satisfying {\rm (${\rm H_1}$)}, {\rm (${\rm H_3}$)}--{\rm (${\rm H_5}$)} uniformly and $\mathcal{F}$ satisfying  {\rm (${\rm H_1}$)}--{\rm (${\rm H_4}$)}. Assume that \eqref{eq: condition-new-new}   holds.  Then $\mathcal{F}$ admits the representation
$$\mathcal{F}(u,B) = \int_{J_u\cap  B} f(x,[u](x),\nu_u(x))\, d\mathcal{H}^{d-1}(x)$$
for all $u \in PR(\Omega)$ and  $B \in \mathcal{B}(\Omega)$, with $f(x,\xi, \nu)$  given by \eqref{eq:gdef}.
\end{corollary} 

We emphasize that we cannot apply  directly Theorem \ref{theorem: PR-representation} on $\mathcal{F}$, since we do not  assume {\rm (${\rm H_5}$)}. The idea is to prove equality in Lemma \ref{lemma: minsame} for $\mathcal{F}$  by  using the corresponding properties for $\mathcal{F}_n$.  
To prove Corollary \ref{corollary: PR-representation-tilde}, we need the following preliminary result, which is itself a corollary of Lemma \ref{lemma: minsame}. In the statement, we write again $\bar{u}_{x_0} := u_{x_0,[u](x_0),\nu_u(x_0)}$ for brevity, see \eqref{eq: jump competitor}.

\begin{corollary}\label{cor: minsame}
Consider a sequence $(\mathcal{F}_n)_n$ satisfying {\rm (${\rm H_1}$)}, {\rm (${\rm H_3}$)}--{\rm (${\rm H_5}$)} uniformly. \BBB  Assume \EEE that   \eqref{eq: condition-new-new}  holds. Let $u \in PR(\Omega) $.  Then  for $\mathcal{H}^{d-1}$-a.e.\ $x_0 \in J_u$  we have
\[
  \lim_{\eps \to 0}\frac{\mathbf{m}_{\mathcal{F}}(u,B_\eps(x_0))}{\omega_{d-1}\eps^{d-1}} =  \limsup_{\eps \to 0}\frac{\mathbf{m}_{\mathcal{F}}(\bar{u}_{x_0},B_\eps(x_0))}{\omega_{d-1}\eps^{d-1}}. 
\]
\end{corollary}

\begin{proof}
The proof is analogous to the one of Lemma \ref{lemma: minsame}. We therefore only highlight the adaptions for one inequality, see Step 1 above.  Fix $\eta, \theta >0$. First, by using \eqref{eq: condition-new-new}, for each $\eps$ we \BBB can choose first $\eps'(\eps) < \eps$  and then $n(\eps)\in  \N$, both depending on $\eps$, \EEE  such that 
\begin{align}\label{eq:repr-cor1}
{\rm (i)} & \ \ \mathbf{m}_{\mathcal{F}}(u,B_\eps(x_0)) \le \mathbf{m}_{{\mathcal{F}_{n(\eps)}}}(u,B_{\eps'  }(x_0)) + \eps^d,\notag\\
{\rm (i)} & \ \ \mathbf{m}_{{\mathcal{F}_{n(\eps)}}}(\bar{u}_{x_0},B_{(1-  3 \theta)  \eps'  }(x_0))  \le  \mathbf{m}_{\mathcal{F}}(\bar{u}_{x_0},B_{(1-  3  \theta)  \eps'  } (x_0))  + \eps^d.
\end{align}
\BBB In fact, first choose $\eps'(\eps) < \eps$ such that $\mathbf{m}_{{\mathcal{F}}}(u, B_\eps(x_0)) \le   \liminf_{n \to \infty} \mathbf{m}_{\mathcal{F}_n}(u, B_{\eps'}(x_0)) +  \eps^d/2$. Then, choose $n(\eps)$ depending on $\eps'$ (and thus on $\eps$) such that \eqref{eq:repr-cor1} holds. \EEE Choose $z_{ \eps' } \in PR(B_{(1-3\theta) \eps'  }(x_0))$ with $z_{ \eps' } = \bar{u}_{x_0}$ in a neighborhood of $\partial B_{(1-3\theta)  \eps' }(x_0)$ and
\begin{align*}
\mathcal{F}_{n(\eps)}(z_{ \eps' },B_{(1-3\theta)  \eps'  }(x_0)) \le \mathbf{m}_{\mathcal{F}_{n(\eps)}}(\bar{u}_{x_0},B_{(1-3\theta)  \eps' }(x_0)) + \eps^d.
\end{align*}
We proceed as in the proof of Lemma \ref{lemma: minsame}: we apply Corollary \ref{rem: scaling} to obtain $w_{ \eps' } \in PR(B_{ \eps' }(x_0))$ with $w_{ \eps' } = u_{ \eps' }$ on $B_{ \eps' }(x_0) \setminus B_{(1-\theta)  \eps'  }(x_0)$ which satisfies  (cf.\ \eqref{eq: rep2}) 
\begin{align}\label{eq:repr-cor3}
\mathcal{F}_{n(\eps)}(w_{ \eps' }, B_{ \eps' }(x_0)) \le   (1 + 2\eta + \rho_\eps)\big(\mathcal{F}_{n(\eps)}(z_{ \eps' },  \eps'  A)  + \mathcal{F}_{n(\eps)}(u_{ \eps' },  \eps'  B)\big) + 3  d\omega_d\eps^{d-1} ( 2  \eta + \rho_\eps)
\end{align}
for a sequence $\rho_\eps$ converging to zero, where we use that $(\mathcal{F}_n)_n$ satisfy {\rm (${\rm H_4}$)} and {\rm (${\rm H_5}$)} uniformly. Applying \eqref{eq:repr-cor1}(ii) and following the lines of \eqref{eq: rep1} we get 
\begin{align*}
\!\!\!\!\limsup_{\eps\to 0}\frac{\mathcal{F}_{n(\eps)}(z_{ \eps' },  \eps'  A) }{( \eps' )^{d-1}}&\le\limsup_{\eps\to 0} \frac{  \mathbf{m}_{\mathcal{F}_{n(\eps)}} (\bar{u}_{x_0},B_{(1-3\theta) \eps' }(x_0))}{ (\eps' )^{d-1}} + \omega_{d-1}\left[1-(1-4\theta)^{d-1}\right]\beta \notag \\
&\le  (1-3\theta)^{d-1}\!\limsup_{ \eps \EEE \to 0} \frac{\mathbf{m}_{\mathcal{F}}(\bar{u}_{x_0},B_{\eps''}(x_0))}{(\eps'')^{d-1}} +  \omega_{d-1}\left[1-(1-4\theta)^{d-1}\right]\beta,
\end{align*}
where we set $\eps''=  (1-3\theta)\eps'$,  and recall that $\eps'' = \eps''(\eps)$ depends on $\eps$. Admitting arbitrary sequence $\eps'' \to 0$, we do not decrease the right hand side. Therefore,
\begin{align}\label{eq:repr-cor4}
\!\!\!\!\limsup_{\eps\to 0}\frac{\mathcal{F}_{n(\eps)}(z_{ \eps' },  \eps'  A) }{( \eps' )^{d-1}}&\le(1-3\theta)^{d-1}\!\limsup_{\eps'' \to 0} \frac{\mathbf{m}_{\mathcal{F}}(\bar{u}_{x_0},B_{\eps''}(x_0))}{(\eps'')^{d-1}} +  \omega_{d-1}\left[1-(1-4\theta)^{d-1}\right]\beta\,.
\end{align}

We also get $\limsup_{\eps \to 0} (\eps')^{-(d-1)} \mathcal{F}_{n(\eps)} (u_{\eps' },  \eps'  B) \le \omega_{d-1}\left[1-(1-4\theta)^{d-1}\right]\beta $ by
\eqref{eq: blow up-new}(ii),(v) and the fact that  {\rm (${\rm H_4}$)} holds uniformly, cf.\ \eqref{eq: rep1XXX}. This together with \eqref{eq:repr-cor1}(i), \eqref{eq:repr-cor3}--\eqref{eq:repr-cor4},  $\eps' < \eps$, and   $w_{ \eps' } = u_{\eps'}  =  u$ in a neighborhood of $\partial B_{ \eps' }(x_0)$  now shows (cf.\ \eqref{eq: cocnlusi})
\begin{align*}
 \lim_{\eps \to 0}\frac{\mathbf{m}_{\mathcal{F}}(u,B_\eps(x_0))}{\omega_{d-1}\eps^{d-1}}& \le  \limsup\nolimits_{ \eps  \to 0}\frac{\mathbf{m}_{{\mathcal{F}_{n(\eps)}}}(u,B_{ \eps' }(x_0))}{\omega_{d-1} ( \eps' )^{d-1}} \le  \limsup\nolimits_{\eps \to 0}\frac{\mathcal{F}_{n(\eps)}(w_{ \eps' }, B_{ \eps' }(x_0))}{\omega_{d-1} ( \eps' )^{d-1}}  \\
 & \le (1+2\eta) \Big((1-3\theta)^{d-1}\! \limsup\nolimits_{\eps \to 0}  \frac{\mathbf{m}_{\mathcal{F}}(\bar{u}_{x_0},B_\eps(x_0))}{\omega_{d-1}\eps^{d-1}} \Big)\notag\\
& \ \ \ + 2(1+2\eta) \left[1-(1-4\theta)^{d-1}\right]\beta+6d\frac{\omega_d}{\omega_{d-1}}\eta.
\end{align*}
Passing to $\eta,\theta \to 0$ we obtain one inequality. To see the   reverse one, we follow the lines of Step 2 of the proof of Lemma \ref{lemma: minsame} and carry out similar adaptions. 
\end{proof}

We close  this section  by noting that Corollary \ref{corollary: PR-representation-tilde} follows from Corollary \ref{cor: minsame}, arguing exactly as in the proof of Theorem \ref{theorem: PR-representation}.  (Note that Lemma \ref{lemma: G=m} is applicable since $\mathcal{F}$ satisfies {\rm (${\rm H_1}$)}--{\rm (${\rm H_4}$)}.)

\section{$\Gamma$-convergence results for functionals on $PR(\Omega)$}\label{sec: gamma}

This section is devoted to the proof of Theorem \ref{th: gamma}.  Consider  a sequence of functionals $(\mathcal{F}_n)_n$   satisfying {\rm (${\rm H_1}$)} and {\rm (${\rm H_3}$)}--{\rm (${\rm H_5}$)} uniformly,  i.e.,  for the same $0 <\alpha < \beta$ and $\sigma: [0,+\infty) \to [0,\beta]$.  We will first identify a $\Gamma$-limit $\mathcal{F}$ with respect to the convergence in measure on $\Omega$. Then,  our goal is to obtain an integral representation of $\mathcal{F}$.  To this aim,  we apply the localization method for $\Gamma$-convergence together with the fundamental estimate in Lemma \ref{lemma: fundamental estimate2} to deduce that properties   {\rm (${\rm H_1}$)}--{\rm (${\rm H_4}$)} are satisfied. As mentioned before, we cannot prove directly that $\mathcal{F}$ satisfies   (${\rm H_5}$) and therefore the results of Subsections \ref{sec: representation1}, \ref{sec: representation2}  cannot be used.  To circumvent this problem, we  will use Corollary  \ref{corollary: PR-representation-tilde} to get the representation result. We will also eventually prove that  (${\rm H_5}$) is satisfied  by  showing that the integrand $f$ satisfies an equivalent property.

Consider a sequence of functionals $(\mathcal{F}_n)_n$ defined on $PR(\Omega)$. As a first step, we analyze fundamental properties of the $\Gamma$-liminf and $\Gamma$-limsup with respect to the topology of the convergence in measure. To this end, we define 
\begin{align}\label{eq: liminf-limsup}
\mathcal{F}'(u,A)&:=\Gamma-\liminf_{n \to \infty} \mathcal{F}_n(u,A)   = \inf \big\{ \liminf_{n \to \infty} \mathcal{F}_n(u_n,A):   \  u_n \to  u \hbox{ in measure on }  A  \big\}, \notag \\
\mathcal{F}''(u,A) &:= \Gamma-\limsup_{n \to \infty} \mathcal{F}_n(u,A)  = \inf \big\{  \limsup_{n \to \infty} \mathcal{F}_n(u_n,A):   \ u_n \to  u \hbox{ in measure on }  A \big\}
\end{align} 
for all $u \in PR(\Omega)$ and $A \in \mathcal{A}(\Omega)$.

\begin{lemma}[Properties of $\Gamma$-liminf and  $\Gamma$-limsup]\label{eq: liminflimsup-prop}
 Let $\Omega \subset \R^d$ open, bounded with Lipschitz boundary.  Let $\mathcal{F}_n: PR(\Omega) \times \mathcal{B}(\Omega) \to [0,\infty)$ be a sequence of functionals satisfying {\rm (${\rm H_1}$)}, {\rm (${\rm H_3}$)}--{\rm (${\rm H_5}$)} for the same $0 <\alpha < \beta$  and  $\sigma: [0,+\infty) \to [0,\beta]$. Define $\mathcal{F}'$ and $\mathcal{F}''$ as in \eqref{eq: liminf-limsup}. Then we have
\begin{align}\label{eq: infsup0}
{\rm (i)} & \ \ 
\mathcal{F}'(u,A) \le \mathcal{F}'(u,B), \ \ \ \ \ \  \mathcal{F}''(u,A) \le \mathcal{F}''(u,B) \ \ \ \text{ whenever } A \subset B, \notag \\
{\rm (ii)} & \ \ 
\alpha\mathcal{H}^{d-1}(J_u \cap A) \le \mathcal{F}'(u,A)    \le  \mathcal{F}''(u,A) \le \beta\mathcal{H}^{d-1}(J_u \cap A), \notag \\
{\rm (iii)} & \ \ 
\mathcal{F}'(u,A)  = \sup\nolimits_{B \subset \subset A} \mathcal{F}'(u,B),  \ \ \ \ \mathcal{F}''(u,A)  = \sup\nolimits_{B \subset \subset A} \mathcal{F}''(u,B)  \ \ \  \text{ whenever } A \in \mathcal{A}_0(\Omega),  \notag \\
{\rm (iv)} & \ \ 
\mathcal{F}'(u,A\cup B) \le  \mathcal{F}'(u,A) + \mathcal{F}''(u,B), \notag \\
& \ \   \mathcal{F}''(u,A\cup B) \le  \mathcal{F}''(u,A) + \mathcal{F}''(u,B)  \ \ \  \text{ whenever } A,B \in \mathcal{A}_0(\Omega). 
\end{align}

\end{lemma}

\begin{proof}
First, (i) is clear as all $\mathcal{F}_n(u,\cdot)$ are measures. The upper bound in (ii) follows from {\rm (${\rm H_4}$)} by taking the constant sequence $u_n = u$ in \eqref{eq: liminf-limsup}. For the lower bound in (ii), we take an (almost) optimal sequence in  \eqref{eq: liminf-limsup}, use {\rm (${\rm H_4}$)}, as well as the lower semicontinuity stated in \BBB Lemma \ref{lemma: PR compactness}(ii). \EEE

To see (iii), we fix  $D \in \mathcal{A}_0(\Omega)$  and first prove that for all sets   $E,F \in \mathcal{A}_0(\Omega)$, $E \subset \subset F \subset \subset D$,  we have
\begin{align}\label{eq: infsup2}
\mathcal{F}'(u,D) \le \mathcal{F}'(u,F) + \mathcal{F}'(u,D \setminus \overline{E}), \ \ \ \ \ \ \mathcal{F}''(u,D) \le \mathcal{F}''(u,F) + \mathcal{F}''(u,D \setminus \overline{E}). 
\end{align} 
(We use different notation for the sets to avoid confusion with the notation in Lemma \ref{lemma: fundamental estimate2}.)  Indeed, let $(u_n)_n, (v_n)_n \subset PR(\Omega)$ be sequences converging in measure to $u$ on $F$  and $D \setminus \overline{E}$, respectively, such that
\begin{align}\label{eq: infsup1}
\mathcal{F}''(u,F) = \limsup\nolimits_{n \to \infty} \mathcal{F}_n(u_n,F), \ \ \ \ \ \ \ \mathcal{F}''(u,D \setminus \overline{E}) = \limsup\nolimits_{n \to \infty} \mathcal{F}_n(v_n,D \setminus \overline{E}).
\end{align}
We apply Lemma \ref{lemma: fundamental estimate2} for  $\psi(t):=\frac{t}{1+t}$, $A = F$, $B =  D  \setminus \overline{E}$,   and some $A' \in \mathcal{A}_0(\Omega)$, $E \subset \subset A' \subset \subset F$,  to obtain $w_n \in PR(D)$ satisfying (see \eqref{eq: assertionfund}(i))
\begin{align}\label{eq: infsup3}
\mathcal{F}_n(w_n, D) \le  \big(\mathcal{F}_n(u_n,F)  + \mathcal{F}_n(v_n, D\setminus \overline{E})\big) (1 + \eta + \rho_n) + C(\eta + \rho_n),
\end{align}
where $C$ depends on $E,F,D$, and  for brevity we set  $\rho_n :=  M\sigma\left(\Lambda(u_n, v_n)\right)$.   We observe that $u_n-v_n$ tends to $0$ in measure on $F\setminus \overline{E}$, which is equivalent to
\[
\int_{F \setminus \overline{E}}\psi(|u_n - v_n|) \to 0
\]
for $\psi(t)=\frac{t}{1+t}$. Hence, $\Lambda(u_n, v_n)\to 0$ by \eqref{eq: Lambda0}, which implies $\rho_n \to 0$. \BBB Since by assumption $u_n \to u$ and $v_n \to u$ in measure on $F$ and $D\setminus \overline{E}$, respectively,  and $\Vert \min\lbrace |w_n - u_n|, |w_n-v_n| \rbrace \Vert_{L^\infty(D)}  \le \Lambda(u_n,v_n)$ by   \eqref{eq: assertionfund}(ii), \EEE the functions $w_n$ converge to $u$ in measure on $D$. Thus, passing to the limit $n \to \infty$ and using \eqref{eq: liminf-limsup}, \eqref{eq: infsup1}--\eqref{eq: infsup3}, we obtain 
$$\mathcal{F}''(u,D) \le \limsup\nolimits_{n \to \infty} \mathcal{F}_n (w_n, D) \le \big(\mathcal{F}''(u,F) + \mathcal{F}''(u,D \setminus \overline{E})\big) (1+\eta) +  C\eta. $$
Since $\eta>0$ was arbitrary, we obtain \eqref{eq: infsup2} for $\mathcal{F}''$. For $\mathcal{F}'$ we argue in a similar fashion.  

By \eqref{eq: infsup2} and \eqref{eq: infsup0}(ii) we get $\mathcal{F}''(u,D) \le \mathcal{F}''(u,F) + \beta \mathcal{H}^{d-1}(J_u \cap( D \setminus \overline{E}))$. As $\mathcal{H}^{d-1}(J_u \cap( D \setminus \overline{E}))$ can be taken arbitrarily small and  $\mathcal{F}''(u,\cdot)$   is an increasing set function, we obtain $\mathcal{F}''(u,D)  \le \sup\nolimits_{F \subset \subset D} \mathcal{F}''(u,F)$. This shows (iii) for  $\mathcal{F}''$. The proof of $\mathcal{F}'$ is similar. 

We finally show (iv). Observe that the inequalities are clear if $A \cap B = \emptyset$. Let  $A,B \in \mathcal{A}_0(\Omega)$  with nonempty intersection. Given $\eps >0$, one can choose $M \subset \subset M' \subset \subset A$ and $N \subset \subset N' \subset \subset B$ such that  $M,M',N,N' \in \mathcal{A}_0(\Omega)$, $M' \cap N' = \emptyset$, and  $\mathcal{H}^{d-1}(J_u \cap ((A\cup B) \setminus \overline{M \cup N} )) \le \eps$,  see \cite[Proof of Lemma 5.2]{AmbrosioBraides} for details. Then using \eqref{eq: infsup0}(i),(ii) and \eqref{eq: infsup2}  
\begin{align*}
\mathcal{F}''(u,A \cup B) & \le  \mathcal{F}''(u,M' \cup N') + \mathcal{F}''(u, (A\cup B) \setminus \overline{M \cup N} ) \le  \mathcal{F}''(u,M') +  \mathcal{F}''(u,N') + \beta\eps  \\ & \le \mathcal{F}''(u,A) +  \mathcal{F}''(u,B) + \beta\eps
\end{align*}
Here, we also used $\mathcal{F}''(u,M' \cup N') \le \mathcal{F}''(u,M') +  \mathcal{F}''(u,N')$ which holds due to $M' \cap N' = \emptyset$. The statement follows as $\eps$ was arbitrary. The proof for $\mathcal{F}'$ is again the same.
\end{proof}

The previous lemma allows us to identify a $\Gamma$-limit on $PR(\Omega)$.

\begin{lemma}\label{lemma: gamma-l2}
 Let $\Omega \subset \R^d$ open, bounded with Lipschitz boundary.  Let $\mathcal{F}_n: PR(\Omega) \times \mathcal{B}(\Omega) \to [0,\infty)$ be a sequence of functionals satisfying {\rm (${\rm H_1}$)}, {\rm (${\rm H_3}$)}--{\rm (${\rm H_5}$)} for the same $0 <\alpha < \beta$  and $\sigma: [0,+\infty) \to [0,\beta]$. Then there exists $\mathcal{F}:  PR(\Omega) \times \mathcal{B}(\Omega) \to [0,\infty]$  and a subsequence (not relabeled) such that
\begin{align}\label{eq_ Gamma limes} 
\mathcal{F}(\cdot,A) =\Gamma\text{-}\lim_{n \to \infty} \mathcal{F}_n(\cdot,A), 
\end{align}
with respect to the topology of the convergence in measure, for all $A \in  \mathcal{A}_0(\Omega)$. The functional $\mathcal{F}$ satisfies {\rm (${\rm H_1}$)}--{\rm (${\rm H_4}$)}. 
\end{lemma}

\begin{proof}
We apply a compactness result for $\bar{\Gamma}$-convergence, see \cite[Theorem 16.9]{DalMaso:93}, to find  an increasing sequence of integers $(n_k)_k$ such that  the objects $\mathcal{F}'$ and $\mathcal{F}''$ defined in \eqref{eq: liminf-limsup} with respect to  $(n_k)_k$ satisfy
$$(\mathcal{F}')_-(u,A) = (\mathcal{F}'')_-(u,A)  $$
for all $u \in PR(\Omega)$ and $A \in \mathcal{A}(\Omega)$, where $(\mathcal{F}')_-$ and $(\mathcal{F}'')_-$ denote the inner regular envelope defined by 
\begin{align}\label{eq: inner enve}
(\mathcal{F}')_-(u,A) = \sup_{B \subset \subset A, B \in \mathcal{A}(\Omega)} \mathcal{F}'(u,B), \ \ \ \ \ \ (\mathcal{F}'')_-(u,A) = \sup_{B \subset \subset A, B \in \mathcal{A}(\Omega)} \mathcal{F}''(u,B).
\end{align}
We write $\mathcal{F}_0 := (\mathcal{F}'')_-$ for simplicity. This along with \eqref{eq: liminf-limsup} and Lemma \ref{eq: liminflimsup-prop}(i) yields
\begin{align}\label{eq: the same on good2}
\mathcal{F}_0 = (\mathcal{F}')_- \le \mathcal{F}' \le  \mathcal{F}''.
\end{align}
We now check that 
\begin{align}\label{eq: the same on good}
\mathcal{F}''(u,A) = \mathcal{F}_0(u,A) \ \ \ \ \  \text{for all $u \in PR(\Omega)$ and all $A \in \mathcal{A}_0(\Omega)$. }  
\end{align}
In view of \eqref{eq: the same on good2}, it suffices to show $\mathcal{F}_0(u,A) \ge \mathcal{F}''(u,A)$. To this end, we fix $u \in PR(\Omega)$, $A \in \mathcal{A}_0(\Omega)$, and $\eps >0$. We choose sets $A'' \subset \subset A' \subset \subset A$  such that $A'\in \mathcal{A}_0(\Omega)$, $A\setminus \overline{A''} \in \mathcal{A}_0(\Omega)$, and $\mathcal{H}^{d-1}(J_u \cap (A \setminus \overline{A''})) \le \eps$. We then find by Lemma \ref{eq: liminflimsup-prop}(ii),(iv) and \eqref{eq: inner enve}  
\begin{align*}
\mathcal{F}''(u,A) \le \mathcal{F}''(u, A') + \mathcal{F}''(u, A \setminus \overline{A''}) \le \mathcal{F}''(u, A') + \beta\eps \le \mathcal{F}_0(u,A) + \beta\eps.
\end{align*}
As $\eps$ is arbitrary, the desired inequality follows. 

Now \eqref{eq: the same on good2}--\eqref{eq: the same on good} show that the $\Gamma$-limit exists for all $u \in PR(\Omega)$ and all $A \in \mathcal{A}_0(\Omega)$. It remains to extend $\mathcal{F}_0: PR(\Omega) \times \mathcal{A}(\Omega) \to [0,\infty]$ to a functional $\mathcal{F}$ defined on $PR(\Omega) \times \mathcal{B}(\Omega)$. To this end, we first note that $\mathcal{F}_0$ is superadditive and inner regular, see \cite[Proposition 16.12 and Remark 16.3]{DalMaso:93}. Moreover, $\mathcal{F}_0$ is subadditive. In fact, for $A,B \in \mathcal{A}(\Omega)$, we choose $A',B' \in \mathcal{A}_0(\Omega)$ with $A' \subset \subset A$, $B' \subset \subset B$, and since $\mathcal{F}_0$ is subadditive on $\mathcal{A}_0(\Omega)$ (see Lemma \ref{eq: liminflimsup-prop}(iv) and \eqref{eq: the same on good}), we get
\begin{align*}
\mathcal{F}_0(u,A' \cup B') \le \mathcal{F}_0(A') + \mathcal{F}_0(B') \le \mathcal{F}_0(A) + \mathcal{F}_0(B).
\end{align*}
  Then $\mathcal{F}_0(u,A \cup B) \le   \mathcal{F}_0(A) + \mathcal{F}_0(B)$ follows from the inner regularity of $\mathcal{F}_0$. 
  By De  Giorgi-Letta (see \cite[Theorem 14.23]{DalMaso:93}), $\mathcal{F}_0(u,\cdot)$ can thus be extended to a Borel measure. 
  
Lemma \ref{eq: liminflimsup-prop}  also yields that the extended functional $\mathcal{F}$ satisfies {\rm (${\rm H_1}$)}, {\rm (${\rm H_3}$)}--{\rm (${\rm H_4}$)}.  The lower semicontinuity  {\rm (${\rm H_2}$)} of $\mathcal{F}(\cdot, A) = \mathcal{F}_0(\cdot,A)$ for $A \in \mathcal{A}(\Omega)$ follows from \cite[Remark 16.3]{DalMaso:93}.
\end{proof}

We are now in the position to prove   Theorem \ref{th: gamma}. 
 
\begin{proof}[Proof of Theorem \ref{th: gamma}]
We observe that $\mathcal{F}$ satisfies \eqref{eq_ Gamma limes} and   {\rm (${\rm H_1}$)}--{\rm (${\rm H_4}$)} by Lemma \ref{lemma: gamma-l2}.  Since  we assume      \eqref{eq: condition-new-new}, we can apply Corollary \ref{corollary: PR-representation-tilde}, so that  $\mathcal{F}$ admits the integral representation 
\begin{align*}
\mathcal{F}(u,B) = \int_{J_u \cap B} f(x,[u](x),\nu_u(x))\, d\mathcal{H}^{d-1}(x)
\end{align*}
with the density $f$ given in \eqref{eq:gdef}.

We are only left to show that {\rm (${\rm H_5}$)} holds. We will equivalently   prove   that $f$ satisfies 
\begin{align*}
|f(x_0,\xi,\nu) -f(x_0,\xi',\nu)| \le   \alpha^{-1}\beta\,\sigma(|(\xi-\xi'|)
\end{align*}
for all $x_0 \in \Omega$, $\xi, \xi' \in \R^d$, and $\nu \in S^{d-1}$. This shows that {\rm (${\rm H_5}$)} holds for the modulus of continuity $\alpha^{-1}\beta \sigma$. To this end, it suffices to prove that for all   $x_0 \in \Omega$,  $\xi,\xi' \in \R^d$, and $\nu \in S^{d-1}$ one has 
\begin{align}\label{eq: h5/6}
 \Big|\limsup_{\eps \to 0} \frac{\mathbf{m}_{\mathcal{F}}(u_{x_0,\xi,\nu},B_\eps(x))}{\omega_{d-1}\,\eps^{d-1}} - \limsup_{\eps \to 0} \frac{\mathbf{m}_{\mathcal{F}}(u_{x_0,\xi',\nu},B_\eps(x))}{\omega_{d-1}\,\eps^{d-1}} \Big|  \le  \alpha^{-1}\beta \sigma(|\xi - \xi'|).
\end{align}
Indeed, then the statement follows from  \eqref{eq:gdef}. 

Let us show  \eqref{eq: h5/6}. We first observe that, in view of   \eqref{eq: condition-new-new}, it suffices to prove 
\begin{align}\label{eq: h5/6-n}
|\mathbf{m}_{{\mathcal{F}_n}}(u_{x_0,\xi,\nu},B_\eps(x_0)) - \mathbf{m}_{{\mathcal{F}_n}}(u_{x_0,\xi',\nu},B_\eps(x_0))| \le  \omega_{d-1}\eps^{d-1}\alpha^{-1}\beta \sigma(|\xi - \xi'|)
\end{align}
for every $n \in \N$. Indeed, once \eqref{eq: h5/6-n} is proved, we conclude as follows: without restriction we suppose that the term inside the brackets on the left hand side of \eqref{eq: h5/6} is nonnegative as otherwise we interchange the roles of $\xi$ and $\xi'$. By \eqref{eq: condition-new-new}, for each $\eps>0$, choose $n(\eps) \in \N$ and $\eps'(\eps) < \eps$  with
\begin{align*}
&\mathbf{m}_{{\mathcal{F}}}(u_{x_0,\xi,\nu},B_{\eps}(x_0)) \le \mathbf{m}_{{\mathcal{F}_{n(\eps)}}}(u_{x_0,\xi,\nu},B_{\eps'}(x_0)) + \eps^d, \\ 
& \mathbf{m}_{{\mathcal{F}_{n(\eps)}}}(u_{x_0,\xi',\nu},B_{\eps'}(x_0)) \le \mathbf{m}_{{\mathcal{F}}}(u_{x_0,\xi',\nu},B_{\eps'}(x_0)) +  (\eps')^d. 
\end{align*}
Then, since $\eps'=\eps'(\eps) < \eps$, we get by \eqref{eq: h5/6-n}
\begin{align*}
0 & \le  \limsup_{\eps \to 0}  \frac{\mathbf{m}_{\mathcal{F}}(u_{x_0,\xi,\nu},B_\eps(x)) }{\omega_{d-1}\,\eps^{d-1}} - \limsup_{\eps \to 0}  \frac{\mathbf{m}_{\mathcal{F}}(u_{x_0,\xi',\nu},B_{\eps'}(x))}{\omega_{d-1}\,(\eps')^{d-1}}  \\
&  \le \limsup_{\eps \to 0} \frac{\mathbf{m}_{{\mathcal{F}_{n(\eps)}}}(u_{x_0,\xi,\nu},B_{\eps'}(x_0))  - \mathbf{m}_{{\mathcal{F}_{n(\eps)}}}(u_{x_0,\xi',\nu},B_{\eps'}(x_0))}{\omega_{d-1}\, (\eps')^{d-1}}   \le  \alpha^{-1}\beta \sigma(|\xi - \xi'|).
\end{align*}
This gives  \eqref{eq: h5/6}. It thus remains to show \eqref{eq: h5/6-n}.  To this  end,  let $\delta >0$ and  choose $z \in PR(B_\eps(x_0))$ with $z= u_{x_0,\xi,\nu}$ in a neighborhood of $\partial B_\eps(x_0)$ and   
\begin{align}\label{eq: new for reprisi}
\mathcal{F}_n(z,B_\eps(x_0)) \le \mathbf{m}_{\mathcal{F}_n}(u_{x_0,\xi,\nu},B_\eps(x_0))+\delta. 
\end{align}
Clearly, in view of (${\rm H_4}$), $\mathbf{m}_{\mathcal{F}_n}(u_{x_0,\xi,\nu},B_\eps(x_0)) \le \omega_{d-1}\eps^{d-1}\beta $ by taking $u_{x_0,\xi,\nu}$ as competitor. Therefore, (${\rm H_4}$) implies 
\begin{align}\label{eq: simple-jumpbound}
\mathcal{H}^{d-1}(J_z) \le  (\omega_{d-1}\eps^{d-1} \beta+\delta)   \alpha^{-1}.
\end{align}
Let $P = \lbrace z = \xi \rbrace$ and note that $P$ is a set of finite perimeter.  (In fact, up to set of negligible $\mathcal{L}^d$-measure, it coincides with one component of its pairwise distinct representation, see \eqref{eq: bdy and jump}.)  We define $z' = z + (\xi'-\xi)\chi_P$ and observe that $z' \in PR(B_\eps(x_0))$ and that   $z'= u_{x_0,\xi',\nu}$ in a neighborhood of $\partial B_\eps(x_0)$. Moreover, we have $J_{z'} \subset J_z$, $[z'] =  [z]  $ $\mathcal{H}^{d-1}$-a.e.\ on $J_{z'} \setminus \partial^* P$, and $[z'] = [z] + \xi' - \xi$ $\mathcal{H}^{d-1}$-a.e.\ on $J_{z'} \cap \partial^* P$. Since the functionals $\mathcal{F}_n$ satisfy {\rm (${\rm H_5}$)} uniformly, we get
\begin{align*}
\mathbf{m}_{\mathcal{F}_n}(u_{x_0,\xi', \nu},B_\eps(x_0)) & \le \mathcal{F}_n(z',B_\eps(x_0)) \le  \mathcal{F}_n(z,B_\eps(x_0)) + \int_{J_{z'} \cap \partial^* P} \sigma(|\xi' - \xi|)\,{\rm d}\mathcal{H}^{d-1}.
\end{align*} 
Then   by \eqref{eq: new for reprisi}  and \eqref{eq: simple-jumpbound} we derive 
\begin{align*}
\mathbf{m}_{\mathcal{F}_n}(u_{x_0,\xi', \nu},B_\eps(x_0)) & \le   \mathbf{m}_{\mathcal{F}_n}(u_{x_0,\xi,\nu},B_\eps(x_0))+\delta +  ( \omega_{d-1}\eps^{d-1}\beta+\delta) \alpha^{-1}\sigma(|\xi' - \xi|).
\end{align*} 
As $\delta>0$ was arbitrary, we obtain one inequality in \eqref{eq: h5/6-n}. The other inequality can be obtained in a similar fashion by interchanging the roles of $\xi$ and $\xi'$.  \EEE 
\end{proof}

The above proof makes use of the assumption \eqref{eq: condition-new-new}, which is not a-priori guaranteed for our functionals, due to lack of coerciveness. As a matter of fact, we below prove that the first inequality in  \eqref{eq: condition-new-new} holds always true in our setting. In the next section, we will then show how, under an additional assumption on $\mathcal{F}_n$ and for specific choices of $L$, also the second one can be confirmed. This yields  a finer $\Gamma$- convergence result for those cases.

\begin{lemma}[Convergence of minima, upper bound]\label{lemma: gamma-min-upper}
 Let $\Omega \subset \R^d$ open, bounded with Lipschitz boundary.  Let $\mathcal{F}_n: PR(\Omega) \times \mathcal{B}(\Omega) \to [0,\infty)$ be a sequence of functionals satisfying {\rm (${\rm H_1}$)}, {\rm (${\rm H_3}$)}--{\rm (${\rm H_5}$)} for the same $0 <\alpha < \beta$  and $\sigma: [0,+\infty) \to [0,\beta]$. Let ${\mathcal{F}}: PR(\Omega) \times \mathcal{B}(\Omega) \to [0,\infty]$ be the $\Gamma$-limit identified in Lemma \ref{lemma: gamma-l2}.  Then for all $A \in \mathcal{A}_0(\Omega)$ and all $u \in PR(\Omega)$ we have
$$\limsup_{n \to \infty} \mathbf{m}_{\mathcal{F}_n}(u, A) \le \mathbf{m}_{{\mathcal{F}}}(u, A). $$
\end{lemma}

\begin{proof}
 Let $D \in \mathcal{A}_0(\Omega)$ and let $\delta>0$. (It will be convenient from a notational point of view to use $D$ instead of $A$.) Let $v \in PR(D)$ with ${\mathcal{F}}(v,D) \le \mathbf{m}_{{\mathcal{F}}}(u, D) + \delta$ and $v=u$ on $N$, where $N \subset D$ is a neighborhood of $\partial D$ such that $N \in \mathcal{A}_0(\Omega)$  and  
\begin{align}\label{eq: neighborhood-smalljump}
\mathcal{H}^{d-1}(J_v \cap N)  = \mathcal{H}^{d-1}(J_u \cap N) \le \delta.
\end{align}
 Let $(v_n)_n \subset PR(D)$ be a recovery sequence for $v$, i.e., 
\begin{align}\label{eq: L2-converg}
\int_{D} \psi(|v_n - v|) \, {\rm d}x \to 0 \ \ \ \text{for} \ \ n\to \infty,
\end{align}
where $\psi(t):= \frac{t}{1+t}$, and 
\begin{align}\label{eq: en-recovery}
\lim_{n \to \infty} \mathcal{F}_n(v_n,D) =  {\mathcal{F}}(v,D) \le  \mathbf{m}_{{\mathcal{F}}}(u, D) + \delta.
\end{align}
We need to adjust the boundary data of $v_n$ to obtain competitors for the minimization problems $\mathbf{m}_{\mathcal{F}_n}(u, D)$. To this end, choose further neighborhoods $N', N''\subset D$ of $\partial D$ satisfying  $N'' \subset \subset  N' \subset \subset N$ and $D \setminus \overline{N'}, D \setminus \overline{N''} \in \mathcal{A}_0(\Omega)$.  We apply Lemma \ref{lemma: fundamental estimate2-new} with $A' = D \setminus \overline{N'}$, $A= D \setminus \overline{N''}$,   $B= N$, and some $\eta >0$    for the functions $u = v_n|_A \in PR(A)$ and $v=v|_B \in PR(B)$. We note that \eqref{eq: extra condition} is satisfied for $n$ sufficiently large since the right hand side is independent of $n$ and the left hand side converges to zero by \eqref{eq: Lambda0} and  \eqref{eq: L2-converg}. Consequently, we obtain a function $w_n \in PR(D)$, which satisfies $w_n = v=u$ on $N''$ by \eqref{eq: assertionfund-newnew}(iii). Moreover, \eqref{eq: assertionfund-newnew}(i) yields
\begin{align*}
\mathcal{F}_n(w_n,D) \le \mathcal{F}_n(v_n,D) + \mathcal{F}_n(v,N) + \big(C_\delta + \mathcal{F}_n(v_n,D) + \mathcal{F}_n(v,N)\big)(2 \eta + \rho_n),
\end{align*}
where $C_\delta$ depends on $D,N,N'$ (and thus on $\delta$), and $\rho_n$ is a sequence converging to zero by \eqref{eq: Theta convergence} and \eqref{eq: L2-converg}. In view of \eqref{eq: neighborhood-smalljump}, \eqref{eq: en-recovery}, and the fact that {\rm (${\rm H_4}$)}  holds for each $\mathcal{F}_n$, we then derive  
\begin{align*}
\limsup_{n\to \infty}   \mathbf{m}_{\mathcal{F}_n}(u, D) \le \limsup_{n \to \infty} \mathcal{F}_n(w_n,D) \le \big(\mathbf{m}_{{\mathcal{F}}}(u, D) + \delta + \beta\delta\big)(1+2  \eta) +  2   C_\delta\eta.
\end{align*}
Letting first $\eta \to 0$ and afterwards $\delta \to 0$, we obtain the desired inequality.  
\end{proof}

 \section{Examples}\label{sec: examples}
In this final section, we focus on the case  $L = \R^{d\times d}_{\rm skew}$ and  $L = SO(d)$, with $d=2,3$, which is relevant from the point of view of the applications. We consider an additional assumption {\rm (${\rm H_6}$)} (in the spirit of \cite{Caterina, Manuel}) for the functionals $\mathcal{F}_n$, and use it to truncate piecewise rigid functions at a low energy expense. This will allow us to overcome the lack of coercivity of our functionals, and to  deduce the lower bound in the inequality \eqref{eq: condition-new-new} (see Lemma \ref{lemma: gamma-min-lower}). With this, a full integral representation result for the $\Gamma$-limit holds true, which we state in Theorem \ref{thm: final}.

\subsection{Truncation}\label{sec: truncation}

We point out that in general, for a sequence $(u_n)_n \subset PR_L(\Omega)$, the bound  $\sup_n \int_{\Omega} \psi(|u_n|) <+\infty$   needed to apply \BBB Lemma \ref{lemma: PR compactness}(i) \EEE  is not guaranteed by the growth condition {\rm (${\rm H_4}$)}. As a remedy, we will therefore  truncate piecewise rigid functions in a suitable way. In this context, we will need to assume 
\begin{itemize}
\item [{\rm (${\rm H_6}$)}]  there exists $c_0\ge 1$   such that for any $u,v \in PR_L(\Omega)$ and $S \in \mathcal{B}(\Omega)$ with the property $S \subset \lbrace x \in J_u \cap J_v:   c_0 \le |[v]| \le c_0^{-1}|[u]|\rbrace  $ we have
$$\mathcal{F}(v,S)  \le  \mathcal{F}(u,S).$$ 
\end{itemize} 
This condition can be interpreted as a kind of `monotonicity condition at infinity' for the jump height. A similar assumption was used in \cite{Caterina, Manuel}, we refer to \cite[Remark 3.2, 3.3]{Caterina} for more details.   Recall the constants $\beta,c_0$ in {\rm (${\rm H_4}$)} and {\rm (${\rm H_6}$)}, respectively.

\begin{lemma}[Truncation]\label{lemma: truncation}
Let $d=2,3$, let $\Omega \subset \R^d$ open, bounded with Lipschitz boundary,   and let $L = \R^{d \times d}_{\rm skew}$ or $L = SO(d)$.  Let $\theta>0$.  Then there exists  $C_{\theta}=C_{\theta}(\theta,c_0,\Omega)>0$   such that for every $u \in PR_L(\Omega)$  and every $\lambda \ge 1$ the following holds:  there exist a rest set  $R \subset \R^d$  with
\begin{align}\label{eq: truncation-rest}
\mathcal{L}^d(R) \le \theta \big( \mathcal{H}^{d-1}(J_u) + \mathcal{H}^{d-1}(\partial \Omega)\big)^{d/(d-1)}, \ \ \ \ \  \mathcal{H}^{d-1}(\partial^*R) \le \theta(\mathcal{H}^{d-1}(J_u) + \mathcal{H}^{d-1}(\partial \Omega)),
\end{align}
and a function $v \in PR_L(\Omega) \cap L^\infty(\Omega;\R^d)$ such that
\begin{align}\label{eq: truncation-main}
{\rm (i)} & \ \  \lbrace u \neq v \rbrace  \subset R \cup \lbrace |u| > \lambda \rbrace \ \ \ \text{up to a set of negligible $\mathcal{L}^d$-measure},\notag\\
{\rm (ii)}& \ \ \Vert v \Vert_{L^\infty(\Omega)} \le C_\theta \lambda  ,\notag\\
{\rm (iii)}& \ \ \mathcal{F}(v,\Omega) \le \mathcal{F}(u,\Omega) + \beta\mathcal{H}^{d-1}(\partial^*R)
\end{align}
for all $\mathcal{F}$  satisfying  {\rm (${\rm H_1}$)}, {\rm (${\rm H_3}$)}, {\rm (${\rm H_4}$)}, and {\rm (${\rm H_6}$)}.
\end{lemma}

We remark that the function $v$ also lies in $SBV(\Omega;\R^d)$. For $L = SO(d)$ this is clear. For $L = \R^{d \times d}_{\rm skew}$, this follows from the (much more general) embedding $SBD^2(\Omega) \cap L^\infty(\Omega;\R^d) \hookrightarrow SBV(\Omega;\R^d)$ (see \cite[Theorem 2.7]{Friedrich:15-4})  or from  
 \cite[Theorem 2.2]{Conti-Iurlano:15}. 

\begin{remark}\label{rem: truncation}
In the statement of Lemma \ref{lemma: truncation},  we can additionally get that $R\subset \Omega$ if $L=SO(d)$, as we are going to show in the proof. In the case $L = \R^{d \times d}_{\rm skew}$, our construction of $R$  might in principle not comply with the above inclusion. It can however be easily recovered a posteriori for many geometries of $\Omega$.  Indeed,  e.g., for convex $\Omega$, we can simply replace $R$ with $R \cap \Omega$ \BBB at the expense of a larger, but still universal, constant in \eqref{eq: truncation-rest} and \eqref{eq: truncation-main}(iii). \EEE This follows from the fact that 
\begin{align}\label{eq: the very last equation2}
\mathcal{H}^{d-1}(\partial \Omega \cap R) \le C\mathcal{H}^{d-1}(\partial^* R).
\end{align}
To see this, \BBB we fix a finite subset $\tilde S \subset  S^{d-1}$ and we  first remark that, up to changing $R$ on a null set, we can assume that,  for $\nu \in \tilde S$ and $y \in \Pi_\nu:= \lbrace x\in \R^d: \ \langle  x, \nu \rangle = 0 \rbrace$, either the line $y + \R \nu$ intersects $R$ on a set of positive Lebesgue measure, or has empty intersection therewith. If now $\partial \Omega \cap R\cap (y + \R \nu)\neq \emptyset$, on the one hand the line intersects  $\partial \Omega$ at most twice (due to convexity). On the other hand, for $\mathcal{H}^{d-1}$-a.e. $y \in \Pi_\nu$ with $\partial \Omega \cap R\cap (y + \R \nu)\neq \emptyset$, applying slicing properties \cite[Theorem 3.108]{Ambrosio-Fusco-Pallara:2000} for the $BV$ function $\chi_R$  we have 
\[
\mathcal{H}^0\big((y + \R \nu) \cap \partial^* R\big) =  \mathcal{H}^0\big(\partial^*((y + \R \nu) \cap  R)\big)\ge 2
\]
since $(y + \R \nu) \cap  R$ is bounded with positive measure. \EEE Thus, for each $\nu \in \tilde S$ there holds by the Slicing Theorem (see, e.g., \cite[Theorem 3.2.22]{Federer})
\begin{align*}
\int_{\partial \Omega \cap R} |\nu \cdot \nu_{\Omega}| {\rm d}\mathcal{H}^{d-1} & = \int_{\Pi_\nu}  \mathcal{H}^0((y + \R \nu) \cap \partial \Omega \cap R) {\rm d}\mathcal{H}^{d-1}(y)   \\
& \le    \int_{\Pi_\nu}  \mathcal{H}^0((y + \R \nu) \cap \partial^* R) {\rm d}\mathcal{H}^{d-1}(y) \le   \mathcal{H}^{d-1}(\partial^*R).
\end{align*}
This applied for a finite collection of $(\nu_i)_i \in S^{d-1}$ such that $\sup_{i} |\langle \nu,\nu_i\rangle| \ge \frac{1}{2}$ for all $\nu \in S^{d-1}$ yields  \eqref{eq: the very last equation2}. 
\end{remark}
 
 We  point out that standard Lipschitz-truncation techniques in  $SBV$, see   \cite[Lemma 3.5]{Braides-Defranceschi} or \cite[Lemma 4.1]{Caterina},   are not applicable here as they do not preserve the property that the function is  piecewise rigid. The main idea in the construction consists in replacing the function $u= \sum_j q_j\chi_{P_j}$ by a  constant function on components where $q_j$ is `too large'.  Since the energy in general depends on the jump height, the energy is affected by such modifications. Thus, this constant has to be chosen in a careful way, and one needs to use (${\rm H_6}$) to ensure \eqref{eq: truncation-main}(iii).  In this context, it is essential to control the maximal and minimal values of $q_j$ on each component $P_j$ outside of a rest set $R$ with small perimeter. To this aim, an additional tool is required when dealing with the case $L= \R^{d \times d}_{\rm skew}$, namely a careful decomposition of sets (Lemma \ref{lemma: decomp-sets}) for which an additional rest set $R_{\mathrm{aux}}$ has to be introduced.  Our  construction is inspired by similar techniques used in \cite[Theorem 3.2]{Manuel} and \cite[Theorem 4.1]{FriedrichSolombrino}.

 While Lemma  \ref{lemma: truncation} can be  proved directly in the case $L=SO(d)$, so that a reader only interested in this case can now already skip to its proof, we need  two auxiliary lemmas to deal with the case $L = \R^{d \times d}_{\rm skew}$. \EEE In the sequel, given $Q \in \R^{d \times d}_{\rm skew}$ and $b \in \R^d$, we denote by $\pi_{\ker Q}(b) \in \R^d$ the orthogonal projection of $b$ onto the kernel of $Q$. Likewise, $\pi_{{\ker Q}^\perp}(b) \in  \R^d$ denotes the projection on the orthogonal complement of $\ker Q$.  The first lemma concerns a uniform control for an affine function $q$ in terms of its  minimal  modulus on sets whose minimal and maximal distance  from the affine space $\lbrace q = \pi_{\ker Q}(b) \rbrace$ are comparable.

\begin{lemma}[Minimal and maximal values of rigid motions]\label{lemma: max-min}
Let $d =2,3$, let $E \subset \R^d$   be a  set of finite perimeter, and let $q = q_{Q,b}$ with $Q \in \R^{d \times d}_{\rm skew}$ such that 
\begin{align}\label{eq: el-lem1-2}
{\rm ess\, sup}_{x \in E}  \, \dist(x, \lbrace q = \pi_{\ker Q}(b) \rbrace ) \le C_0 \, {\rm ess\, inf}_{x \in E} \, \dist(x, \lbrace q = \pi_{\ker Q}(b) \rbrace )
\end{align}
 for some $C_0 \ge 1$. Then there holds 
 \begin{align*}
\Vert q \Vert_{L^\infty(E)} \le C_0\, {\rm ess \, inf}_{x \in E} |q(x)|.
\end{align*}
\end{lemma}

\begin{proof}
We start with $d=2$.  Without restriction we can suppose that $Q \neq 0$. Then $Q$ is invertible,  hence $\ker Q=\{0\}$,  and  $\lbrace q = 0 \rbrace = \lbrace z \rbrace$ for $z : = - Q^{-1} b$. If  $|Q|$ denotes the Frobenius norm, we have $|Qy| = \frac{\sqrt{2}}{2}|Q||y|$ for all $y \in \R^2$. Then the fact that $q(z) = 0 $ implies
$$|q(x)| = |q(x) - q(z)| = |Q(x-z)| = \frac{\sqrt{2}}{2}|Q||x-z| =  \frac{\sqrt{2}}{2}|Q| \, \dist(x, \lbrace q = 0 \rbrace). $$
By \eqref{eq: el-lem1-2} this implies 
\begin{align}\label{eq: 2D-calc}
{\rm ess \, inf}_{x \in E} |q(x)| & =   \frac{\sqrt{2}}{2}|Q| \, {\rm ess \, inf}_{x \in E}  \, \dist(x, \lbrace q = 0 \rbrace) \ge   \frac{\sqrt{2}}{2C_0}|Q| \, {\rm ess \, sup}_{x \in E}  \, \dist(x, \lbrace q = 0 \rbrace) \notag \\
&  = \frac{1}{C_0}\Vert q \Vert_{L^\infty(E)}.
\end{align}
This yields the statement for $d = 2$. The case $d=3$ may simply be reduced to the two-dimensional problem by performing calculation \eqref{eq: 2D-calc} restricted to planes which are orthogonal to the line $\lbrace q = \pi_{\ker Q}(b) \rbrace$. (Note that $\lbrace q= \pi_{\ker Q}(b) \rbrace$ is one-dimensional unless $Q  = 0$.) 
\end{proof}

  Note that for $Q \neq 0$ we have ${\rm dim} \, \lbrace q = \pi_{\ker Q}(b) \rbrace = 0$ if $d=2$ and  ${\rm dim} \, \lbrace q = \pi_{\ker Q}(b) \rbrace = 1$ if $d=3$. Property \eqref{eq: el-lem1-2} can always be achieved by introducing a suitable partition of sets of finite perimeter, as the following lemma shows.  Its proof is deferred to Appendix \ref{proof2}. \EEE
 
\begin{lemma}[Decomposition of sets]\label{lemma: decomp-sets}
There exists a universal constant $c>0$ such that the following holds for each $0 < \theta < 1$ :

(a) For each $x_0 \in \R^2$ and each indecomposable, bounded set of finite perimeter $E \subset \R^2$ there exists $R \subset \R^2$ with $\mathcal{H}^1(\partial^* R) \le \theta \mathcal{H}^1(\partial^* E)$ such that
\begin{align}\label{eq: essinf0}
{\rm ess\, sup}_{x \in E \setminus R}  \, |x-x_0| \le c\theta^{-1} \,  {\rm ess\, inf}_{x \in E \setminus R} \,  |x-x_0|.
\end{align}  

(b) For each line $K = x_0 + \R \nu \subset \R^3$, $x_0, \nu \in \R^3$, and each indecomposable, bounded set of finite perimeter $E \subset \R^3$ there exist pairwise disjoint sets of finite perimeter $R$ and $(D_j)_{j=1}^J$ satisfying $\bigcup_{j=1}^J D_j \subset  E \subset R \cup \bigcup_{j=1}^J D_j$ and
\begin{align}\label{eq: lengthi bound}
\mathcal{H}^2(\partial^* R) \le \theta \mathcal{H}^2(\partial^* E), \  \ \ \ \ \  \sum\nolimits_{j=1}^J \mathcal{H}^2(\partial^* D_j \setminus \partial^* E) \le \theta \mathcal{H}^2(\partial^* E)
\end{align}
 such that
\begin{align}\label{eq: essinf}
{\rm ess\, sup}_{x \in D_j}  \, \dist(x,K) \le c\theta^{-3} \, {\rm ess\, inf}_{x \in D_j} \, \dist(x,K) \ \ \  \text{ for all $j=1\ldots,J$.}
\end{align} 
 
\end{lemma}

 We now proceed with the proof  of Lemma \ref{lemma: truncation}.
 
\begin{proof}[Proof of Lemma \ref{lemma: truncation}]
We  first provide the proof for $L = \R^{3 \times 3}_{\rm skew}$. Then, we briefly indicate the necessary changes for the two-dimensional case  $L = \R^{2 \times 2}_{\rm skew}$. In both cases, an additional step using the previous lemmata is needed to construct an auxiliary rest set $R_{\mathrm{aux}}$ and to derive \eqref{eq: lengthi bound--appl}--\eqref{eq: new boundary length}. We then sketch the proof for  the nonlinear case $L = SO(d)$, $d=2,3$, which follows by a similar argument but does not need  Lemmata \ref{lemma: max-min} and \ref{lemma: decomp-sets}. We note that it suffices to prove the lemma for $\theta \le \theta_0$ for some small $\theta_0 \le \frac{1}{2} $ depending on $c_0$ and $\Omega$.

\emph{Proof for $L = \R^{3 \times 3}_{\rm skew}$:} Let $u \in PR_L(\Omega)$ and let $u= \sum_{i\in \N} q'_i\chi_{P'_i}$ be an  indecomposable representation   (see Section \ref{sec: rig/piec-rig}).  On each $P'_i$ with  $Q'_i \neq 0$, we have $\dim \lbrace  q'_i = \pi_{\ker Q_i'}(b_i') \rbrace = 1$. Hence, we may apply Lemma \ref{lemma: decomp-sets}(b) for $K= \lbrace  q'_i = \pi_{\ker Q_i'}(b_i') \rbrace$  to obtain a  covering $P'_i \subset   R_i \cup \bigcup_{j=1}^{J_i} D^i_j$  with $D^i_j \subset P'_i$, $ j  = 1,\ldots,J_i$, satisfying
\eqref{eq: lengthi bound}--\eqref{eq: essinf}. Otherwise, if $Q'_i=0$, it trivially holds $ \lbrace  q'_i = \pi_{\ker Q_i'}(b_i') \rbrace  =\mathbb{R}^3$. On such components $P'_i$,  we simply set $R_i = \emptyset$ and $D^i_1 = P'_i$.

We define $R_{\rm aux} = \bigcup_{i \in \N} R_i$ and denote by $(P_j)_{j \in \N}$ the partition of $\Omega \setminus R_{\rm aux}$ consisting of the sets $(D^i_j \setminus R_{\rm aux})_{i,j}$. For each $j \in \N$, we let $q_j=q_{Q_j,b_j} = q'_{i_j}$, where the index $i_j \in \N$ is chosen such that $P_j \subset P'_{i_j}$. From  \eqref{eq: lengthi bound}--\eqref{eq: essinf} and Theorem \ref{th: local structure} we then obtain
\begin{align}\label{eq: lengthi bound--appl}
&\mathcal{H}^2(\partial^* R_{\rm aux}) \le \theta \sum\nolimits_{i\in \N} \mathcal{H}^2(\partial^* P'_i),\\
&\sum\nolimits_{j \in \N} \mathcal{H}^2\Big(\partial^* P_j \setminus \bigcup\nolimits_{i \in \N} \partial^* P'_i\Big) \le
\BBB  \sum\nolimits_{i \in \N}  \sum\nolimits_{j=1}^{J_i} \mathcal{H}^2(\partial^* D^i_j \setminus \partial^* P_i') +   \mathcal{H}^2(\partial^* R_{\rm aux})   \EEE \notag\\ & \quad\quad \ \, \quad\quad\quad\quad\quad\quad\quad\quad\quad\quad\quad \le  \theta \sum\nolimits_{i \in \N} \mathcal{H}^2(\partial^* P'_i) + \BBB \mathcal{H}^2(\partial^* R_{\rm aux})  \EEE \le  \BBB 2 \EEE \theta \sum\nolimits_{i \in \N} \mathcal{H}^2(\partial^* P'_i)\notag\,.
\end{align}
Moreover, we have
\begin{align}\label{eq: essinf-appl}
{\rm ess\, sup}_{x \in P_j}  \, \dist(x, \lbrace q_j =  \pi_{\ker Q_j}(b_j)  \rbrace ) \le c\theta^{-3} \, {\rm ess\, inf}_{x \in P_j} \, \dist(x, \lbrace q_j =  \pi_{\ker Q_j}(b_j)  \rbrace)   
\end{align} 
for all $j \in \N$. Indeed, if $\dim \lbrace q_j =  \pi_{\ker Q_j}(b_j)  \rbrace =1$, \eqref{eq: essinf-appl} follows from \eqref{eq: essinf} and the fact that $P_j \subset D^i_k$ for some $D^i_k$.  If   $\lbrace  q_j = \pi_{\ker Q_j}(b_j) \rbrace  =\mathbb{R}^3 $, it  is trivially satisfied. We also note that \eqref{eq: bdy and jump-new},  \eqref{eq: lengthi bound--appl}, and Theorem \ref{th: local structure} imply 
\begin{align}\label{eq: new boundary length}
\sum\nolimits_{j \in \N} \mathcal{H}^2(\partial^* P_j) \le c\sum\nolimits_{i\in \N} \mathcal{H}^2(\partial^* P'_i) \le c(\mathcal{H}^2(J_u) + \mathcal{H}^2(\partial \Omega)),
\end{align}
where $c>0$ is universal.

We define $I_\lambda =  \lbrace j\in \N: \ \Vert q_j \Vert_{L^\infty(P_j)}>\lambda\theta^{-6} \rbrace$ and introduce a decomposition of $I_\lambda$ according to the $L^\infty$-norms of the rigid motions: for $k \in \N$ we introduce the set of indices
\begin{align}\label{eq: trunc1}
{I}^k_\lambda = \lbrace j \in I_\lambda:  \lambda \theta^{-6k}  < \Vert q_j \Vert_{L^\infty(P_j)}  \le \lambda \theta^{-6(k+1)}   \rbrace
\end{align}
and  define $s_k = \sum_{j \in {I}^k_\lambda} \mathcal{H}^2(\partial^* P_j)$ for $k \in \N$. By \eqref{eq: new boundary length}  we find some $K_\theta \in \N$, $K_\theta \le \theta^{-1}$, such that 
\begin{align}\label{eq: trunc2} 
s_{K_\theta} \le c\theta (\mathcal{H}^2(J_u) + \mathcal{H}^2(\partial \Omega)).
\end{align}
We define the index set 
\begin{align}\label{eq: indexset}
I = \bigcup\nolimits_{k>K_\theta} {I}^k_\lambda
\end{align}
and introduce the rest set
\begin{align}\label{eq: indexset2}
R =  \bigcup\nolimits_{j \in {I}^{K_\theta}_\lambda} P_j \  \cup \  R_{\rm aux}.
\end{align}
By Theorem \ref{th: local structure}, \eqref{eq: lengthi bound--appl}, \eqref{eq: new boundary length},  and  \eqref{eq: trunc2}  we find 
\begin{align}\label{eq: trunc3} 
\mathcal{H}^2(\partial^* R) \le \sum_{j \in I_\lambda^{K_\theta}} \mathcal{H}^2(\partial^* P_j) +   \mathcal{H}^2(\partial^* R_{\rm aux}) \le s_{K_\theta}+ \theta \sum_{i\in \N} \mathcal{H}^2(\partial^* P'_i) \le c\theta(\mathcal{H}^2(J_u) + \mathcal{H}^2(\partial \Omega)).
\end{align}
In view of Lemma \ref{lemma: max-min} and \eqref{eq: essinf-appl}, we obtain for each $j \in I$  
\begin{align}\label{eq: maximini}
\Vert q_j \Vert_{L^\infty(P_j)} \le c \theta^{-3} \,  {\rm ess \, inf}_{x \in P_j} |q_j(x)| \le (3\theta^{4})^{-1} \,  {\rm ess \, inf}_{x \in P_j} |q_j(x)|,
\end{align}
where the last step holds for $\theta_0$ sufficiently small.  We define $U = R \cup \bigcup\nolimits_{j \in I} P_j$  and get by \eqref{eq: trunc1}, \eqref{eq: indexset}--\eqref{eq: indexset2}, and \eqref{eq: maximini} that
\begin{align}\label{eq: trunc4} 
{\rm (i)}& \ \ \Vert u \Vert_{L^\infty(\Omega \setminus U)} \le \lambda\theta^{-6K_\theta}, \notag \\ 
{\rm (ii)}& \ \  {\rm ess \, inf} \lbrace  |u(x)|: \ x \in U \setminus R \rbrace  \ge  3\theta^4 \lambda\theta^{-6 (K_\theta+1)} = 3\lambda\theta^{-6K_\theta-2}.
\end{align}
We define $v \in PR_L(\Omega) \cap L^\infty(\Omega;\R^3)$  by
\begin{align}\label{eq: trunc5} 
v := u\chi_{\Omega \setminus U} + b e_1 \,  \chi_{U}, \ \ \ \  \text{where} \ b:= \lambda\theta^{-6K_\theta-1}.
\end{align} 
We now show \eqref{eq: truncation-rest}--\eqref{eq: truncation-main} and start with \eqref{eq: truncation-main}. First, \eqref{eq: truncation-main}(i) follows from \eqref{eq: trunc4}(ii). Setting $C_\theta = \theta^{-6/\theta-1}$, \eqref{eq: truncation-main}(ii) follows from \eqref{eq: trunc4}(i), \eqref{eq: trunc5}, and the fact that $K_\theta \le 1/\theta$.  

We now address   \eqref{eq: truncation-main}(iii). As a preparation, we compare the jump sets of $u$ and $v$. First, \eqref{eq: trunc4}(i) and \eqref{eq: trunc5}  show that $J_v \supset \partial^* U \cap \Omega$ up to an $\mathcal{H}^{2}$-negligible set. Choose the orientation of $\nu_v(x)$ for $x \in \partial^* U \cap \Omega$ such that $v^+(x)$ coincides with the trace of $v\chi_{U}$ at $x$ and  $v^-(x)$ coincides with the trace of $v\chi_{\Omega \setminus U}$ at $x$. (The traces have to be understood in the sense of \cite[Theorem 3.77]{Ambrosio-Fusco-Pallara:2000}.) Moreover, we suppose that $\nu_v= \nu_u$ on $J_u \cap \partial^* U \cap \Omega$. Suppose that $\theta \le \theta_0 \le \frac{1}{2}$. For $\mathcal{H}^{2}$-a.e.\ $x \in ( \partial^* U \cap \Omega) \setminus \partial^* R$ we derive  by   \eqref{eq: trunc4}--\eqref{eq: trunc5} and $[v](x) = v^+(x) - v^-(x) =  b e_1 - u^-(x)$ that
\begin{align*}
\lambda\theta^{-6K_\theta} \le b- \Vert u \Vert_{L^\infty(\Omega \setminus U)} \le 
|[v](x)|    \le  b +  \Vert u \Vert_{L^\infty(\Omega \setminus U)} \le  2\lambda\theta^{-6K_\theta-1}.
\end{align*} 
In a similar fashion, we obtain
$$| [u](x)| \ge  |u^+(x)| - |u^-(x)| \ge  3\lambda\theta^{-6K_\theta-2} - \lambda\theta^{-6K_\theta} \ge 2\lambda\theta^{-6K_\theta-2}. $$
Therefore, since $\lambda \ge 1$ and   $K_\theta \ge 1$,  we find 
\begin{align}\label{eq: trunc6} 
\theta^{-1} \le |[v](x)| \le \theta |[u](x)|
\end{align}
for $\mathcal{H}^2$-a.e.\ $x \in (\partial^* U \cap \Omega) \setminus \partial^* R$. We are now in a position to show \eqref{eq: truncation-main}(iii). By {\rm (${\rm H_1}$)},  {\rm (${\rm H_3}$)},  $u=v$ on $\Omega \setminus U$,  and the fact that $v$ is constant on $U$, we get
\begin{align*}
\mathcal{F}(v,\Omega) &= \mathcal{F}(v, (U)^1) + \mathcal{F}(v, (\Omega \setminus U)^1) + \mathcal{F}(v, \partial U^* \cap \Omega)   = \mathcal{F}(u, (\Omega \setminus U)^1) + \mathcal{F}(v, \partial U^* \cap \Omega)\\
& \le \mathcal{F}(u, (\Omega \setminus U)^1) + \mathcal{F}(v, (\partial^* U \setminus \partial^* R)  \cap \Omega) +  \mathcal{F}(v, \partial^* R  \cap \Omega).
\end{align*}
By {\rm (${\rm H_4}$)}, {\rm (${\rm H_6}$)}, and \eqref{eq: trunc6}  (for $\theta$ sufficiently  small   such that $\theta^{-1} \ge c_0$) we get
\begin{align*}
\mathcal{F}(v,\Omega) 
& \le \mathcal{F}(u, (\Omega \setminus U)^1) + \mathcal{F}(u, (\partial^* U \setminus \partial^* R)  \cap \Omega) + \beta \mathcal{H}^1(\partial^* R )\\
& \le \mathcal{F}(u, \Omega) +  \beta\mathcal{H}^1(\partial^* R ),
\end{align*}
where in the last step we again used {\rm (${\rm H_1}$)} and the fact that $\mathcal{F}(u,(U)^1)\ge 0$. This concludes the proof of \eqref{eq: truncation-main}(iii).

It remains to show \eqref{eq: truncation-rest}. By \eqref{eq: trunc3}  and the isoperimetric inequality  we obtain the desired estimate with $c\theta$ in place of $\theta$. Clearly, the constant $c$ can be absorbed in $\theta$ by repeating the above arguments for $\theta/c$ in place of $\theta$. This concludes the proof for $L = \R^{3 \times 3}_{\rm skew}$. 

\emph{Adaptions for $L = \R^{2 \times 2}_{\rm skew}$:} For the two-dimensional case $L = \R^{2 \times 2}_{\rm skew}$, the following small adaption is necessary: before \eqref{eq: lengthi bound--appl}, for components $P'_i$ with $\dim \lbrace  q'_i = 0 \rbrace = 0$ (i.e., $Q'_i \neq 0$), we apply Lemma \ref{lemma: decomp-sets}(a) in place of Lemma \ref{lemma: decomp-sets}(b). (This case is even easier since the collection $(D^i_j)_j$  consists of one set only.)

\emph{Proof for $L = SO(d)$, $d=2,3$:} Here, we do not need to introduce a decomposition using Lemma \ref{lemma: decomp-sets}, and we can work  directly  with the indecomposable representation $u=\sum_{j\in \N} q_j\chi_{P_j}$. We define the index  sets  $I^k_\lambda$, the integer $K_\theta$, and the index set $I$ exactly as in \eqref{eq: trunc1}--\eqref{eq: indexset}. We set
\[
R=\bigcup\nolimits_{j \in {I}^{K_\theta}_\lambda} P_j\,.
\] 
Notice that, by construction, we have $R\subset \Omega$ as stated in Remark \ref{rem: truncation}. \EEE By Theorem \ref{th: local structure}  and  \eqref{eq: trunc2}  we find
\[
\mathcal{H}^{d-1}(\partial^* R) \le \sum\nolimits_{j \in I_\lambda^{K_\theta}} \mathcal{H}^{d-1}(\partial^* P_j) \le s_{K_\theta}\le c\theta(\mathcal{H}^{d-1}(J_u) + \mathcal{H}^{d-1}(\partial \Omega))\,.
\]
We further observe by \eqref{eq: trunc1}  and \eqref{eq: indexset} that we have for each $j \in I$ that $\Vert q_j \Vert_{L^\infty(P_j)}   \ge \lambda \theta^{-6} \ge  2{\rm diam}(\Omega)$, where the second step holds for $\theta_0$ sufficiently small. As   $q_j$ is an isometry, there holds 
\begin{align*}
\Vert q_j \Vert_{L^\infty(P_j)} \le   {\rm ess \, inf}_{x \in P_j} |q_j(x)| + {\rm diam}(\Omega) \le {\rm ess \, inf}_{x \in P_j} |q_j(x)| + \frac{1}{2}\Vert q_j \Vert_{L^\infty(P_j)},
\end{align*}
which in turn implies $\Vert q_j \Vert_{L^\infty(P_j)} \le 2 \,   {\rm ess \, inf}_{x \in P_j} |q_j(x)|$ for all $j \in I$. This inequality clearly  yields \eqref{eq: maximini}. The result then follows by {\it verbatim} repeating the argument after \eqref{eq: maximini}. 
\end{proof}

\subsection{A finer $\Gamma$-convergence result}
We first show that, under assumption {\rm (${\rm H_6}$)} and  for  $L = \R^{d\times d}_{\rm skew}$ or $L = SO(d)$, $d=2,3$, the second inequality in \eqref{eq: condition-new-new} holds  as a consequence of Lemma \ref{lemma: truncation}.

\begin{lemma}[Convergence of minima, lower bound]\label{lemma: gamma-min-lower}
 Let $d=2,3$, and let $L = \R^{d\times d}_{\rm skew}$ or $L = SO(d)$. Let $\Omega \subset \R^d$ open, bounded with Lipschitz boundary.  Let $\mathcal{F}_n: PR_L(\Omega) \times \mathcal{B}(\Omega) \to [0,\infty)$ be a sequence of functionals satisfying {\rm (${\rm H_1}$)}, {\rm (${\rm H_3}$)}--{\rm (${\rm H_6}$)} for the same $0 <\alpha < \beta$, $c_0 \ge 1$,  and $\sigma: [0,+\infty) \to [0,\beta]$. Let ${\mathcal{F}}: PR_L(\Omega) \times \mathcal{B}(\Omega) \to [0,\infty]$ be the $\Gamma$-limit identified in Lemma \ref{lemma: gamma-l2}.  Then for each ball $B_\eps(x_0) \subset \Omega$  and all $u \in PR_L(\Omega)$ we have
$$\sup\nolimits_{0<\eps' < \eps} \,  \liminf_{n \to \infty} \mathbf{m}_{\mathcal{F}_n}(u, B_{\eps'}(x_0)) \ge \mathbf{m}_{{\mathcal{F}}}(u, B_\eps(x_0)). $$
\end{lemma}

\begin{proof}
For convenience, we again drop the subscript $L$ in the proof and write $A = B_\eps(x_0)$. Let $\theta>0$.  Fix $u \in PR(\Omega)$ and choose a ball $A' := B_{\eps'}(x_0)$, $\eps' < \eps$, such that
\begin{align}\label{eq: almost optimal}
 \mathcal{H}^{d-1}(J_u\cap (A\setminus A')) \le \theta.
\end{align}
As $u$ is measurable, we may fix  a nonnegative, monotone increasing, and  coercive function $\psi$ with
\begin{align}\label{eq: ourpsi}
\int_A \psi(|u|)\,\mathrm{d}x<+\infty\,.
\end{align}
Now, let $u = \sum_{j\in \N} q_j\chi_{P_j}$ be the piecewise distinct representation. In view of Theorem \ref{th: local structure}, we can choose $J\in\N$ sufficiently large such that the set $S_\theta := \bigcup_{j >J} P_j$ satisfies
\begin{align}\label{eq: gamma-lower-1}
\mathcal{H}^{d-1}\big( J_u\cap \big( \partial^* S_\theta \cup  (S_\theta)^1 \big)\big) \le \theta,
\end{align}
where $(S_\theta)^1$ denotes the set of points with density $1$.  Since $J$ is finite, we may fix  $\lambda_\theta \ge 1$  such that 
\begin{align}\label{eq: boundustheta}
\Vert u \Vert_{L^\infty(A\setminus S_\theta)} <\lambda_\theta.
\end{align}
We now consider a sequence $(v_n)_n \subset PR(A')$ with $v_n = u$ in a neighborhood $N_n \subset A'$ of $\partial A'$ and   
\begin{align}\label{eq: almost optimal-new}
\mathcal{F}_n(v_n,A') \le \mathbf{m}_{\mathcal{F}_n}(u,A') + 1/n\,.
\end{align}
Without restriction we can suppose that $\sup_n \mathcal{F}_n(v_n,A') <+ \infty$, i.e., $\sup_n \mathcal{H}^{d-1}(J_{v_n}) <+ \infty$ by  {\rm (${\rm H_4}$)}. We apply Lemma \ref{lemma: truncation} and Remark \ref{rem: truncation}  with $A'$ in place of $\Omega$  and for $\lambda= \lambda_\theta$ on each $v_n$ and find $v_n' \in PR(A') \cap L^\infty(A';\R^d)$  and sets of finite perimeter $R_n^{\theta} \subset A'$  such that  by \eqref{eq: truncation-rest} and \eqref{eq: truncation-main}(iii)  
\begin{align}\label{eq: gamma-lower-2}
\mathcal{F}_n(v_n',A') \le \mathcal{F}(v_n,A') + C\beta\theta, \ \ \  \ \  \  \  \  \mathcal{H}^{d-1}(\partial^* R_n^{\theta}) \le C\theta,
 \end{align}
 where $C$ depends on $A$ and  $\sup_n \mathcal{H}^{d-1}(J_{v_n}) <+ \infty$.   Observe  by  \eqref{eq: truncation-main}(i) that we have $\lbrace v_n \neq v_n' \rbrace\subset R_n^\theta \cup \lbrace  |v_n|  > \lambda_\theta \rbrace$, so that using \eqref{eq: boundustheta} we deduce that $v'_n = u$ on $N_n \setminus (R_n^\theta \cup S_\theta)$. 

We introduce the functions $v_n^\theta \in PR(A)$ by
\begin{align}\label{eq: vntheta}
v_n^{\theta} = \begin{cases} u & \text{ on } (A \setminus \overline{A'}) \cup S_\theta,\\
v_n' & \text{ else. }
\end{cases}
\end{align}
By {\rm (${\rm H_1}$)}, {\rm (${\rm H_3}$)}, and {\rm (${\rm H_4}$)} this implies 
\begin{align}\label{eq: the very last equation}
\mathcal{F}_n(v_n^\theta,A) &\le \mathcal{F}(u,(A \setminus \overline{A'}) \cup (S_\theta)^1)+ \mathcal{F}_n(v_n', A' \cap (S_\theta)^0) \notag\\
& \ \ \ + \beta\mathcal{H}^{d-1}(\partial^* S_\theta) +\beta\mathcal{H}^{d-1}(J_{v_n^\theta} \cap \partial A' \cap (S_\theta)^0),
 \end{align}
where  $(S_\theta)^0$ denotes the set of points with density $0$.  Since $v'_n = u$ on $N_n \setminus (R_n^\theta \cup S_\theta)$, we have $J_{v_n^\theta} \cap \partial A' \cap (S_\theta)^0 \subset  \partial^* R_n^\theta$.   With this, using \eqref{eq: the very last equation}, {\rm (${\rm H_1}$)}, and {\rm (${\rm H_4}$)}, we get 
\begin{align*}
\mathcal{F}_n(v_n^\theta,A) & \le \mathcal{F}_n(v_n', A') + \beta\mathcal{H}^{d-1}\big(J_u \cap \big( (A \setminus \overline{A'}) \cup (S_\theta)^1 \big) \big) + \beta\mathcal{H}^{d-1}(\partial^* S_\theta) + \beta\mathcal{H}^{d-1}(\partial^* R^\theta_n).
 \end{align*}  
Therefore, by  \eqref{eq: almost optimal}, \eqref{eq: gamma-lower-1}, and \eqref{eq: gamma-lower-2} we get
\begin{align}\label{eq: gamma-lower-2.2}
\mathcal{F}_n(v_n^\theta,A)  \le \mathcal{F}_n  (v_n,A') + C\beta\theta.
\end{align} 
Since $\sup_n \mathcal{F}_n(v_n, A'  ) <+ \infty$, we get  $\sup_n \mathcal{H}^{d-1}(J_{v^\theta_n}) <+ \infty$ by {\rm (${\rm H_4}$)} and \eqref{eq: gamma-lower-2.2}. By  \eqref{eq: truncation-main}(ii) and the  construction in \eqref{eq: vntheta}, it holds $|v_n^\theta(x)| \le \max \lbrace C_\theta \lambda_\theta, |u(x)| \rbrace$ for a.e.\ $x \in A$,   where $C_\theta$ is the constant in \eqref{eq: truncation-main}(ii).   With \eqref{eq: ourpsi} we then have $\sup_n \int_A \psi(|v_n^\theta|)\,\mathrm{d}x<+\infty$. Hence,  we can apply \BBB Lemma \ref{lemma: PR compactness}(i) \EEE to find $v^\theta \in PR(A)$ such that, up to a subsequence (not relabeled), $v_n^{\theta} \to v^\theta$ in measure on $A$.  Clearly, by \eqref{eq: vntheta} we have $v^\theta = u$ on $A \setminus \overline{A'}$.    \EEE     By  \eqref{eq_ Gamma limes}, \eqref{eq: almost optimal-new},  and \eqref{eq: gamma-lower-2.2} we get
\begin{align*}
{\mathcal{F}}(v^\theta, A)  \le \liminf_{n \to \infty}\mathcal{F}_n(v_n^{\theta},A) \le \liminf_{n \to \infty} \mathcal{F}_n(v_n,A') +  C\beta\theta \le \liminf_{n \to \infty} \mathbf{m}_{\mathcal{F}_n}(u,A') +  C\beta\theta.
\end{align*} 
As $v^\theta = u$ in a neighborhood of $\partial A$,  we get 
$$\mathbf{m}_{\mathcal{F}}(u,A) \le {\mathcal{F}}(v^\theta, A) \le \liminf_{n \to \infty} \mathbf{m}_{\mathcal{F}_n}(u,A') +  C\beta\theta\le \sup\nolimits_{0<\eps' < \eps} \,  \liminf_{n \to \infty} \mathbf{m}_{\mathcal{F}_n}(u, B_{\eps'}(x_0)) +C\beta\theta, $$   
where in the last step we used that $A' = B_{\eps'}(x_0)$.  By passing to $\theta \to 0$ we conclude the proof. 
\end{proof}  

By combining the above lemma with Theorem \ref{th: gamma} and Lemma \ref{lemma: gamma-min-upper} we finally get a full integral representation result for the $\Gamma$-limit in the setting considered in this section.

\begin{theorem}\label{thm: final}
 Let $d=2,3$, and let $L = \R^{d\times d}_{\rm skew}$ or $L = SO(d)$. Let $\Omega \subset \R^d$ open, bounded with Lipschitz boundary.  Let $\mathcal{F}_n: PR_L(\Omega) \times \mathcal{B}(\Omega) \to [0,\infty)$ be a sequence of functionals satisfying {\rm (${\rm H_1}$)}, {\rm (${\rm H_3}$)}--{\rm (${\rm H_6}$)} for the same $0 <\alpha < \beta$, $c_0 \ge 1$,  and $\sigma: [0,+\infty) \to [0,\beta]$.   Then there exists $\mathcal{F}:  PR_L(\Omega) \times \mathcal{B}(\Omega) \to [0,\infty)$ satisfying  {\rm (${\rm H_1}$)}--{\rm (${\rm H_5}$)} and a subsequence (not relabeled) such that
$$\mathcal{F}(\cdot,A) =\Gamma\text{-}\lim_{n \to \infty} \mathcal{F}_n(\cdot,A) \ \ \ \ \text{with respect to convergence in measure on $A$} $$
for all $A \in  \mathcal{A}_0(\Omega) $. Moreover,  $\mathcal{F}$ admits the representation \eqref{eq:gdef-new}--\eqref{eq:gdef}.

\end{theorem}

\section*{Acknowledgements} 
\BBB This work was supported by the DFG project FR 4083/1-1 and \EEE by the Deutsche Forschungsgemeinschaft (DFG, German Research Foundation) under Germany's Excellence Strategy EXC 2044 -390685587, Mathematics M\"unster: Dynamics--Geometry--Structure. The support by the Alexander von Humboldt Foundation is gratefully acknowledged. The authors are gratefully indebted to {\sc Marco Cicalese} for many stimulating discussions on the content of this work.  

\appendix

\section{Proof of Proposition \ref{prop: rotations}}\label{proof1}
\begin{proof}
We give the proof only for $d=3$   since it is  similar and simpler for $d=2$.   Consider the exponential map $S\mapsto \exp(S)$,   which is  surjective  from a compact subset of $\R^{3\times 3}_{\rm skew}$ to   $SO(3)$. Clearly,   once we have proved property \eqref{eq: repr0.0}  for $\exp$,    the desired map $\Psi_L$ can be defined as the composition of $\exp$  with the canonical isomorphism between $\R^{3\times 3}_{\rm skew}$  and $\R^3$. Throughout the proof, we denote by $|\cdot|$ the Frobenius norm of a matrix, by $|\cdot|_2$ its spectral norm, and with $c_2 \in (0,1)$ an equivalence constant between the two norms. Since in this case $r_L < + \infty$ (we can take for instance $r_L = 2\pi$), up to
rescaling the constant,  it is equivalent to prove (2.2) for a ball of radius $c_0$  in place of $c_L r_L$.

We start  by    fixing $c_0<\frac14$ and $\bar R\in SO(3)$. If $|\bar R-I|\le\frac12$, then $|R-I|<\frac34$ for all $R \in B_{c_0}(\bar R)$: therefore, a smooth inverse given by a  matrix logarithm is well-defined in  $B_{c_0}(\bar R)$ through the usual Taylor expansion around the identity.   By  its smoothness it clearly satisfies \eqref{eq: repr0.0}.

We therefore focus on the case where $|\bar R-I|>\frac12$. Since $c_0 < \frac{1}{4}$, there holds
\begin{align}\label{eq: barR distance}
|R-I| \ge \tfrac{1}{4} \ \ \ \text{for all} \ \ R \in B_{c_0}(\bar R).
\end{align} 
In view of Rodrigues' rotation formula and the power series expansion of $\exp$, for each $R \in B_{c_0}(\bar R) \cap SO(3)$ there exists a unit vector $n_R$ and an angle $\theta_R$ such that
\begin{align}\label{eq: Rodrigues}
R = I + \sin(\theta_R) N_R + (1-\cos(\theta_R)) N_R^2  = \exp(\theta_R N_R)\,,
\end{align}
where $N_R \in \R^{3\times 3}_{\rm skew}$ denotes the unique matrix with $N_R u = n_R \times u$ for all $u \in \R^3$. In particular, for each $u \in \R^3$ there holds  with the help of the Gra\ss{}mann identity $n_R \times (n_R \times u) = n_R \langle n_R,u \rangle- u$ 
\begin{align}\label{eq: Rodrigues2}
Ru =  \cos(\theta_R) \, u + \sin(\theta_R) \,  (n_R \times u)  + \langle n_R,u \rangle (1-\cos(\theta_R)) \,  n_R .
\end{align}    
Thus, our goal is to specify the choice of $N_R$ and $\theta_R$. The desired matrix $S_R \in \R_{\rm skew}^{3 \times 3}$ is then defined by $S_R = \theta_R N_R$ since  $\exp(S_R) = R$, see \eqref{eq: Rodrigues}. We start with some preliminary facts (Step 1). Then, we define the map $R \mapsto S_R$ on $B_{c_0}(\bar R) \cap SO(3)$  and show that it is Lipschitz (Step 2).  

\emph{Step 1: Preliminary facts.} Let $R\neq I$ be a rotation. Then there exists a unit eigenvector $n$ of $R$ with eigenvalue $1$ which corresponds to the rotation axis. Let $n^\perp = \lbrace w\in \R^3: \langle n,w \rangle = 0 \rbrace $. Since $Rn = n$, for all $w\in n^{\perp}$ with $|w|=1$ there holds \BBB that \EEE
\begin{equation}\label{eq: property1}
|(R-I)w|=|R-I|_2
\end{equation}
\BBB is constant with respect to $w$. \EEE Indeed, the fact that $|(R-I)w|$ is constant follows from the Rodrigues' rotation formula \eqref{eq: Rodrigues2}, and the second implication immediately follows by the definition of the spectral norm. Notice also that for $R$, $n$ as before, and for all $w\in n^{\perp}$ with $|w|=1$ it holds $|(R-I)w|^2=-2\langle(R-I)w, w\rangle$ since $R$ is a rotation.  Combining with \eqref{eq: property1} we get
\begin{equation}\label{eq: property2}
\langle(R-I)w, w\rangle=-\frac12|R-I|^2_2\,.
\end{equation}
As a further preparation, we show that for $R_1, R_2 \in  B_{c_0}(\bar R) \cap SO(3)$ there holds
\begin{align}\label{eq: axis diff}
\sqrt{1 - |\langle n_1, n_2 \rangle |^2} \le 4c_2^{-1}|R_1 - R_2| \le 8c_2^{-1} c_0\,,
\end{align} 
  where  $n_i$ are unit eigenvectors of $R_i$ with eigenvalue $1$ for $i=1,2$. The second inequality is clear. To see  the first, by $R_2n_2=n_2$ we  get on the one hand 
\[
|(R_1-I)n_2|=|( R_1-R_2 )n_2|\le |R_1-R_2|.
\]
On the other hand, writing $n_2=\mu n_1 +\sqrt{1-\mu^2}y$, where $\mu=\langle n_1, n_2  \rangle$ and $y\in n_1^{\perp}$ with $|y|=1$, we have by  \eqref{eq: barR distance},  \eqref{eq: property1}, and $R_1n_1 = n_1$
\[
|(R_1-I)n_2|=\sqrt{1-\mu^2}|(R_1-I)y|=\sqrt{1-\mu^2}|R_1-I|_2\ge c_2 \sqrt{1-\mu^2}|R_1-I|\ge \tfrac{1}{4} c_2  \sqrt{1-\mu^2}\,,
\]
where we also used $|\cdot|_2 \ge c_2|\cdot|$. By combining the two estimates we get \eqref{eq: axis diff}.

\emph{Step 2: Construction of the inverse mapping.} Given $\bar R$, we define a positive orthonormal basis $\{\bar n, \bar w, \bar z\}$ of $\R^3$ with $\bar R \bar n= \bar n$.   Consider $R \in  B_{c_0}(\bar R) \cap SO(3)$ and let $n$ be a unit vector with $Rn=n$.  Provided  that we let $c_0 \le c_2/16$, $n$ can be chosen such that $\langle n,\bar{n} \rangle \ge \frac{3}{4}$, see \eqref{eq: axis diff}. We define a positive orthonormal basis $\{n, w, z\}$ by $w = (\bar{z} \times n)/|\bar{z} \times n|$ and $z = n \times w$. Since $R \in SO(3)$ and $R n= n$, the vectors $\{n, R w, R z\}$ form a positive orthornormal basis of $\R^3$ as well. Hence,  we get the equalities
\begin{align*}\label{eq: decomposition}
Rw=\langle Rw, w\rangle w + \langle Rw, z\rangle z, \ \ \ \ \ \ \
Rz=-\langle Rw, z\rangle w + \langle Rw, w\rangle z\,.
\end{align*}
Now, the point $(\langle R w, w\rangle, \langle R w, z\rangle)$ lies in $S^1$. By an elementary computation along with \eqref{eq: barR distance} and \eqref{eq: property2} we get
\[
|(\langle R w, w\rangle, \langle R w, z\rangle)-(1,0)|=\sqrt{2\langle(I- R)w, w\rangle} = |R-I|_2 \ge c_2 |R-I| \ge c_2/4.
\]
Hence, we can consider a smooth inverse $\Theta$ of the mapping $\theta \mapsto (\cos(\theta), \sin(\theta))$ defined on $S^1\setminus B_{ c_2/4}(1,0)$ and with values in a compact interval of the form $[\eta ,2\pi-\eta]$. We define 
\begin{align}\label{eq_ thetaR}
\theta_R = \Theta(\langle Rw, w\rangle, \langle Rw, z\rangle).
\end{align} 
The function $\Theta$ can be taken globally Lipschitz on its domain since the latter is at positive distance to the singularity at $(1,0)$. 

Summarizing, given $R \in B_{c_0}(\bar R) \cap SO(3)$, we let $n_R = n$ with $Rn=n$, $|n|=1$, $N_R \in \R^{3\times 3}_{\rm skew}$ with $N_R u = n_R \times u$ for all $u \in \R^3$,  $\theta_R$ as in \eqref{eq_ thetaR}, 
and  $S_R =  \theta_R  N_R$. Recall that $R = \exp(S_R)$, see  \eqref{eq: Rodrigues}.  Finally, to check that $R \mapsto S_R$ is Lipschitz, we  first note that $R \mapsto n_R$ is Lipschitz.  Indeed, let  $n_1$ and $n_2$  be the rotation axes corresponding to  $R_1$ and $R_2$ with $\langle n_i, \bar n \rangle \ge \frac{3}{4}$ for $i=1,2$. Then it is elementary to check that $\langle n_1, n_2 \rangle \ge \langle n_2,\bar n \rangle - |n_1 - \bar n| \ge \frac{1}{4} \ge 0$. By \eqref{eq: axis diff} we then get 
\begin{align}\label{eq: n1n2}
|n_1 - n_2| = \sqrt{2 - 2\langle n_1,n_2 \rangle} \le \sqrt{2}\sqrt{ (1 + \langle n_1,n_2 \rangle) (1 - \langle n_1,n_2 \rangle)  }  \le 4\sqrt{2}c_2^{-1}|R_1 - R_2|\,.
\end{align}
In a similar fashion, $R \mapsto \theta_R$ is Lipschitz with $\theta_R$ from  \eqref{eq_ thetaR}. In fact, $\Theta$ is globally Lipschitz on $S^1\setminus B_{ c_2/4}(1,0)$, and by construction together with  \eqref{eq: n1n2}  there holds
$$|w_1-w_2| \le C|R_1-R_2|\,,\quad \quad|z_1-z_2| \le C|R_1-R_2|$$
for some universal  $C>0$, where $w_i=(\bar z \times n_i)/|\bar z \times n_i|$ and $z_i= n_i \times w_i$ for $i=1,2$.
\end{proof}

\section{Proof of Lemma \ref{lemma: decomp-sets}}\label{proof2}
\begin{proof}
 For the proof we use the notation  $\diam(F) = {\rm ess }\sup\lbrace \vert x-y \vert: x,y \in F \rbrace$ and 
 $${\rm diam}_1(F) = {\rm ess\, sup} \lbrace | \langle x-y,  e_1 \rangle |: \, x,y \in F \rbrace$$
for bounded, measurable sets $F \subset\R^d$, $d=2,3$.

Part (a) relies on the property that for each indecomposable, bounded set  of finite perimeter $E$ one has
\begin{align}\label{eq: maggi}
{\rm diam}(E) \le \mathcal{H}^1(\partial^* E).
\end{align}
For a proof we refer to \cite[Propostion 12.19, Remark 12.28]{maggi}. Then the statement  follows simply by letting $R =  B_r(x_0)$  be the circle with center $x_0$   and radius $r = \frac{1}{2\pi}\theta{\rm diam}(E)$. Then \eqref{eq: maggi} implies $\mathcal{H}^1(\partial R) \le \theta \mathcal{H}^1(\partial^* E)$ and \eqref{eq: essinf0} holds   since
 $${\rm ess\, sup}_{x \in E \setminus R} \,  |x-x_0| \le   {\rm ess\, inf}_{x \in E \setminus R} \,  |x-x_0| +  {\rm diam}(E) \le (1+2\pi \theta^{-1}) \, {\rm ess\, inf}_{x \in E \setminus R} \,  |x-x_0|. $$
 
For (b), we may    suppose that $K = \R \times \lbrace (0,0) \rbrace$ after applying an isometry. The proof  is considerably more difficult than the one in (a) since an estimate of the form \eqref{eq: maggi} is wrong in general and the object 
\begin{align}\label{eq: r def}
r :=  ({\rm diam}_1(E))^{-1}\, \mathcal{H}^{2}(\partial^* E)
\end{align}
might be much smaller than $1$. To this end, we will first need to construct a decomposition of $E$ into pieces with smaller diameter in $e_1$ direction (Step 1) which allows us to control the relation of perimeter and ${\rm diam}_1$ (see Step 2). Afterwards, a further tubular decomposition of each of these pieces is needed (Step 3). In Step 4 we will finally show that the constructed partition satisfies \eqref{eq: lengthi bound}--\eqref{eq: essinf}.  Throughout Steps 1 - 2 we will assume that ${\rm diam}_1(E) > 2\mathcal{H}^{2}(\partial^* E)^{\frac12}$, so that in particular ${\rm diam}_1(E) > 0$ and ${\rm diam}_1(E) > 4r$. If instead ${\rm diam}_1(E) \le 2\mathcal{H}^{2}(\partial^* E)^{\frac12}$ holds, one can directly skip to Step 3, consider a single $T_1 = E$ in Step 3 - 4, and observe that in this case \eqref{eq: diam-surf} is clearly satisfied for $\theta \le c_{\pi, 2}$, where this latter is the isoperimetric constant in the plane.

\noindent \emph{Step 1 (Cutting in $e_1$ direction):}  The goal of this step is to construct a decomposition of $E$  into pairwise disjoint sets $(T_i)_{i=1}^I$ of the form $T_i = E \cap ((t_{i-1}, t_i) \times \R^2 )$, $i=1,\ldots, I$, for suitable $ t_0 < t_ 1 < \ldots < t_{I} $, which satisfy
\begin{align}\label{eq: added length5}
\sum\nolimits_{i=1}^I \mathcal{H}^2(\partial^* T_i \setminus \partial^* E) \le c\theta\mathcal{H}^2(\partial^* E)  
\end{align}
for a universal constant $c>0$ and
\begin{align}\label{eq: main cut3}
\mathcal{H}^2\big( E \cap (\lbrace t \rbrace \times  \R^2 )   \big) > \theta \mathcal{H}^2\big( \partial^* E \cap ((t_{i-1},t_i) \times  \R^2 )   \big) \ \ \ \text{for $\mathcal{H}^1$-a.e.\ $t_{i-1} + r  \le t \le t_i-r$}.
\end{align}
We point out that ${\rm diam}_1(T_i) = t_i - t_{i-1} < 2r$  is possible. In this case, condition \eqref{eq: main cut3} is trivial.

 To achieve this, we perform an iterative decomposition of the set $E$. Choose the largest $t' \in \R$ and the smallest $t'' >t'$  such that $E \subset (t',t'') \times \R^2$ up to a set of negligible $\mathcal{L}^3$-measure.  We start to construct a first auxiliary decomposition $(S_j)_{j=1}^J$. We describe  the first step of the construction of $(S_j)_{j=1}^J$ in detail: choose $s_1 \in (t',t'']$ such that 
\begin{align}\label{eq: main cut}
{\rm (i)} & \ \  \mathcal{H}^2\big( E \cap (\lbrace s_1 \rbrace \times  \R^2 )   \big) \le 2\theta \mathcal{H}^2\big( \partial^* E \cap ((t',s_1) \times  \R^2 )   \big) \ \ \ \text{or} \ \ \ s_1 = t'',\notag \\
{\rm (ii)} & \ \  \mathcal{H}^2\big( E \cap (\lbrace t \rbrace \times  \R^2 )   \big) > 2\theta \mathcal{H}^2\big( \partial^* E \cap ((t',t) \times  \R^2 )   \big) \ \ \ \text{for $\mathcal{H}^1$-a.e.\ $t'+r \le t \le s_1-r$}.
\end{align}
In fact,  this is possible: let 
$$M = \big\{ t \in (t'+r,t''): \ \mathcal{H}^2\big( E \cap (\lbrace t \rbrace \times  \R^2 )   \big) \le 2\theta \mathcal{H}^2\big( \partial^* E \cap ((t',t) \times  \R^2 )   \big) \big\}. $$
If $M   \neq \emptyset$, select $s_1 \in M$ such that $(t' +r , s_1 - r) \cap M = \emptyset$.  (This is indeed possible by choosing $s_1 \in M \cap [\inf M ,\inf M + r)$. As pointed out below \eqref{eq: main cut3}, $s_1 - t' <2r$ is admissible. In this case, the interval is empty and condition \eqref{eq: main cut}(ii) is trivial.) 
If $M = \emptyset$, let $s_1 = t''$.  Define $S_1 := ((t',s_1) \times \R^2) \cap E$. Observe that this also implies ${\rm diam}_1(S_1) \ge r$. 
 
 We now proceed iteratively: suppose that $(S_j)_{j=1}^k$ have been defined and let $E_k = E \setminus \bigcup_{j=1}^k S_j$. As long as ${\rm diam}_1(E_k)>r$, we then repeat the above procedure for $E_k$ in place of $E$. Hereby, after a finite number of iterations, we obtain a decomposition $E = \bigcup_{j=1}^J S_j$, where ${\rm diam}_1(S_j) \ge r$ for all $j=1, \ldots, J-1$. (Note that the  control from below by $r$ on ${\rm diam}_1$ ensures that the iteration procedure stops after a finite number of steps.) Setting $s_0 = t'$ and  $s_J = t''$  for convenience, we find by \eqref{eq: main cut}(ii)
\begin{align}\label{eq: main cut4}
\mathcal{H}^2 \big( E \cap (\lbrace t \rbrace \times  \R^2 )   \big) > 2\theta \mathcal{H}^2 \big( \partial^* E \cap ((s_{j-1},t) \times  \R^2 )   \big) \ \ \ \text{for $\mathcal{H}^1$-a.e.\ $s_{j-1}+r \le t \le s_j-r$} 
\end{align}  
for all $j = 1,\ldots, J$.   Moreover, by using \eqref{eq: main cut}(i)  (with $s_{j-1}$ and  $s_j$ in place of $t'$ and   $s_1$)   we get
 \begin{align}\label{eq: added length}
\sum\nolimits_{j=1}^J \mathcal{H}^2(\partial^* S_j \setminus \partial^* E) &\le 2\sum\nolimits_{j=1}^{J-1} \mathcal{H}^2( E \cap (\lbrace s_j \rbrace \times  \R^2 )   ) \notag \\ 
& \le 4\theta \sum\nolimits_{j=1}^{J-1}  \mathcal{H}^2 \big( \partial^* E \cap ((s_{j-1},s_j) \times  \R^2 )   \big) \le 4\theta\mathcal{H}^2(\partial^* E). 
 \end{align}

We now repeat the above procedure for each $S_j$ in place of $E$ starting from the right instead of from the left: the first set in the decomposition of each $S_j$ is obtained by choosing $s^1_j  \in [s_{j-1},s_j)$ such that 
\begin{align}\label{eq: main cut2}
{\rm (i)} & \ \  \mathcal{H}^2 \big( S_j \cap (\lbrace  s^1_j  \rbrace \times  \R^2 )   \big) \le 2\theta \mathcal{H}^2 \big( \partial^* S_j \cap (( s^1_j, s_j) \times  \R^2 )   \big)  \ \ \ \text{or} \ \ \  s^1_j = s_{j-1}, \\
{\rm (ii)} & \ \  \mathcal{H}^2 \big( S_j \cap (\lbrace t \rbrace \times  \R^2 )   \big) > 2\theta \mathcal{H}^2 \big( \partial^* S_j \cap ((t,s_j) \times  \R^2 )   \big) \  \text{for $\mathcal{H}^1$-a.e.\ $s^1_j  +r \le t \le s_j-r$}.\notag
\end{align}
We set $S_j^1 := (s_j^1,s_j) \cap S_j = (s_j^1,s_j) \cap E$ and proceed iteratively as before to define sets $(S_j^k)_{k \ge 1}$ of the form $S_j^k =  ((s_j^{k+1},s_j^k) \times \R^2)   \cap E$. 

For convenience, we denote the decomposition $(S_j^k)_{j,k}$ of $E$ by $(T_i)_{i=1}^I$ and observe that there exist $t' =  t_0 < t_ 1 < \ldots < t_{I}  = t'' $ such that $T_i = E \cap ((t_{i-1}, t_i) \times \R^2 )$ for all $i=1,\ldots, I$. 

We show \eqref{eq: main cut3}: first, by \eqref{eq: main cut4}  and the fact that $(s_j^{k+1},t) \subset (s_{j-1},t)$ for all $s_j^{k+1} + r  \le t \le s_j^{k}-r$ we get 
$$
\mathcal{H}^2\big( E \cap (\lbrace t \rbrace \times  \R^2 )   \big) > 2\theta \mathcal{H}^2\big( \partial^* E \cap ((t_{i-1},t) \times  \R^2 )   \big) \ \ \ \  \text{for $\mathcal{H}^1$-a.e.\ $t_{i-1} + r  \le t \le t_i-r$}.
$$
The fact that in  \eqref{eq: main cut2}(ii)  we may replace $S_j$ by $E$ without changing the estimate yields 
$$
\mathcal{H}^2\big( E \cap (\lbrace t \rbrace \times  \R^2 )   \big) > 2\theta \mathcal{H}^2\big( \partial^* E \cap ((t,t_{i}) \times  \R^2 )   \big) \  \ \ \  \text{for $\mathcal{H}^1$-a.e.\ $t_{i-1}+r \le t \le t_i-r$}.
$$
Combining the previous two estimates and using that $\mathcal{H}^2\big( \partial^* E \cap (\lbrace t \rbrace \times  \R^2 )   \big) = 0$ for $\mathcal{H}^1$-a.e.\ $t$, we get \eqref{eq: main cut3}. Moreover, repeating the argument \eqref{eq: added length} we derive
\begin{align}\label{eq: added length2}
\sum\nolimits_{k \ge 1} \mathcal{H}^2(\partial^* S^k_j \setminus \partial^* S_j) \le 4\theta\mathcal{H}^2(\partial^* S_j)  
\end{align}
for all $j=1,\ldots,J$.    Then from \eqref{eq: added length} and \eqref{eq: added length2} we obtain
\begin{align*}
\sum\nolimits_{i=1}^I \mathcal{H}^2(\partial^* T_i \setminus \partial^* E) & \le \sum\nolimits_{j=1}^J \mathcal{H}^2(\partial^* S_j \setminus \partial^* E) + \sum\nolimits_{j=1}^J\sum\nolimits_{k \ge 1} \mathcal{H}^2(\partial^* S^k_j \setminus \partial^* S_j) \\
& \le 4\theta\Big( \mathcal{H}^2(\partial^* E)  + \sum\nolimits_{j=1}^J \mathcal{H}^2(\partial^* S_j)\Big)  \le 4\theta\Big( \mathcal{H}^2(\partial^* E)  +  2\mathcal{H}^2\Big( \bigcup\nolimits_{j=1}^J\partial^* S_j\Big)\Big)\\
& \le 12\theta \mathcal{H}^2(\partial^* E)  + 8\theta \sum\nolimits_{j=1}^J \mathcal{H}^2(\partial^* S_j \setminus \partial^* E)  \le (12\theta + 32\theta^2) \mathcal{H}^2(\partial^* E),
\end{align*} 
where we also used that each $x \in \R^3$ lies in at most two different $\partial^* S_j$. This yields \eqref{eq: added length5}.

\noindent \emph{Step 2 (Relation of ${\rm diam}_1$ and perimeter):} We now prove a fundamental ${\rm diam}_1$-perimeter-relation of the sets $(T_i)_{i=1}^I$ which have been constructed in Step 1: each  $T_i$ with ${\rm diam}_1(T_i) \ge 4r$ satisfies 
\begin{align}\label{eq: diam-surf}
 {\rm diam}_1(T_i) \le 2\sqrt{c_{\pi,2} \sigma_i /\theta},
\end{align}
where $c_{\pi,2}$ denotes the isoperimetric constant in dimension two and for brevity we use the notation
\begin{align}\label{eq: sigma-shorthand}
\sigma_i := \mathcal{H}^2 \big( \partial^* E \cap ((t_{i-1},t_{i}) \times  \R^2 )\big).
\end{align}

We prove   \eqref{eq: diam-surf}. Since we are assuming ${\rm diam}_1(T_i) \ge 4r$,  \eqref{eq: main cut3} is nontrivial and yields 
$$ \theta \sigma_i  < \mathcal{H}^2\big( E \cap (\lbrace t \rbrace \times  \R^2 )   \big)  $$
for $\mathcal{H}^1$-a.e.\ $t_{i-1} + r \le t \le t_i-r$. Thus, the isoperimetric inequality in dimension two applied on the sets $E \cap (\lbrace t \rbrace \times  \R^2)$ implies for $\mathcal{H}^1$-a.e.\ $t_{i-1} + r \le t \le t_i-r$ that  
$$\theta \sigma_i \le c_{\pi,2}  \Big(\mathcal{H}^1\big(  \partial^*( E \cap (\lbrace t \rbrace \times  \R^2) )   \big) \Big)^{2}.$$
We recall that the coarea formula on rectifiable sets (see, e.g., \cite[Theorem 18.8 and Formula (18.25)]{maggi}) gives, for all $a, b \in \mathbb{R}$, that
\[
 \mathcal{H}^2\big(  \partial^*E \cap ((a,b) \times  \R^2)    \big)\ge \int_{a}^{b} \mathcal{H}^1\big(  \partial^* E \cap (\lbrace t \rbrace \times  \R^2)  \big)  \, \mathrm{d}t= \int_{a}^{b} \mathcal{H}^1\big(  \partial^*( E \cap (\lbrace t \rbrace \times  \R^2) )   \big)  \, \mathrm{d}t\,,
\]
where the last equality is proved, for instance, in \cite[Theorem 18.11]{maggi}.
With this, \EEE by using  $t_i - t_{i-1} - 2r = {\rm diam}_1(T_i) - 2r \ge  {\rm diam}_1(T_i)/2$  and by integrating from $t_{i-1} +r$ to $t_i-r$ we find
\begin{align*}
\frac{1}{2}{\rm diam}_1(T_i)\sqrt{\theta \sigma_i} \le    (t_i - t_{i-1} - 2r)  \sqrt{\theta \sigma_i} \le \sqrt{c_{\pi,2} }\int_{t_{i-1} +r}^{t_i-r} \mathcal{H}^1\big(  \partial^*( E \cap (\lbrace t \rbrace \times  \R^2) )   \big)  \, dt \le \sqrt{c_{\pi,2} }\sigma_i,
\end{align*}
where the last step follows from the shorthand \eqref{eq: sigma-shorthand}. This  yields \eqref{eq: diam-surf} and concludes the proof of this step.

\noindent \emph{Step 3 (Tubular covering of each $T_i$):} Consider $T_i = ((t_{i-1},t_i) \times \R^2) \cap E$. For notational convenience, we will often not add indices $i$, even if the following objects depend on $i$.    We set   $w_j = j\theta^{-1} (\mathcal{H}^2(\partial^* T_i))^{1/2} $ for all $j \in \N$. We introduce the function  $f: \R^3 \to \R$ defined by  $f(x) = \dist(x, \R e_1 ) = \sqrt{|x_2|^2 + |x_3|^2}$ for $x \in \R^3$. We now decompose $T_i$ into sublevel sets of the function $f$: define $z_0 = \theta^2(\mathcal{H}^2(\partial^* T_i))^{1/2}$ and choose $z_j \in (w_{j},w_{j+1})$ for $j \in \N$ such that 
\begin{align} \label{eq: decompo}
\mathcal{H}^2(\lbrace f= z_j \rbrace \cap T_i) & \le \frac{1}{w_{j+1}-w_{j}} \int_{w_{j}}^{w_{j+1}}  \mathcal{H}^2(\lbrace f= z \rbrace \cap T_i) \, \mathrm{d}z   = \frac{1}{w_{1}} \int_{w_{j}}^{w_{j+1}}   \mathcal{H}^2(\lbrace f= z \rbrace \cap T_i) \, \mathrm{d}z,
\end{align}
where the last step follows from the definition of $w_j$. 

We define a covering of $T_i$ by setting $U^i_0 := ((t_{i-1},t_i) \times \R^2) \cap \lbrace f \le z_0 \rbrace$ and $U^i_j := T_i \cap \lbrace z_{j-1}  <  f \le z_j \rbrace$ for $j \ge 1$.   We  observe that this decomposition   is finite since $E$ (and thus $T_i$) is a bounded set in $\R^3$. For later purposes, we observe that 
\begin{align}\label{eq: dist-max-min}
\inf_{x \in U_1^i} f(x) \ge z_{0} = \theta^3  w_{1} = \frac{1}{2} \theta^3 w_{2}   \ge   \frac{1}{2} \theta^3 z_{1}  \ge  \frac{1}{2} \theta^3 \sup_{x \in U_1^i}\EEE f(x),
\end{align}
and in a similar fashion, for all $j \ge 2$,
\begin{align}\label{eq: dist-max-min2}
\inf_{x \in U_j^i} f(x) \ge z_{j-1} \ge  w_{j-1} = \frac{j-1}{j+1} w_{j+1}   \ge  \frac{1}{3} z_j \ge  \frac{1}{3}   \sup_{x \in U_j^i} f(x).
\end{align}
(Clearly, the above property is  false for $U_0^i$.) 
We now estimate the perimeter of the sets $(U^i_j)_{j \ge 0 }$. First, observe  that  by construction we clearly have $\sigma_i\le \mathcal{H}^2(\partial^* T_i)$, where $\sigma_i$ was defined in \eqref{eq: sigma-shorthand}. Moreover, by \eqref{eq: diam-surf} we get  $ {\rm diam}_1(T_i) \le 2 \sqrt{c_{\pi,2}  \sigma_i /\theta}$ if  ${\rm diam}_1(T_i) \ge 4r$ and   $ {\rm diam}_1(T_i) \le 2\sqrt{r}({\rm diam}_1(T_i))^{1/2}$  otherwise. To summarize both cases,  by  recalling the previous observation, we can write
\[
{\rm diam}_1(T_i) \le 2 \sqrt{c_{\pi,2}/\theta} \, \mathcal{H}^2(\partial^* T_i)^\frac12+ 2\sqrt{r}({\rm diam}_1(T_i))^{1/2}.
\]
Thus, recalling $z_0 = \theta^2(\mathcal{H}^2(\partial^* T_i))^{1/2} $ we can estimate  the perimeter of the cylinder $U^i_0$  by 
\begin{align}\label{eq: U0 length}
\mathcal{H}^2(\partial^* U^i_0)&   = 2 \cdot \pi z_0^2 + 2\pi z_0 {\rm diam}_1(T_i)   \le c\theta^4 \mathcal{H}^2(\partial^* T_i)+ c \theta^2 (\mathcal{H}^2(\partial^* T_i))^{1/2}  {\rm diam}_1(T_i) \notag \\
& \le  c\theta \mathcal{H}^2(\partial^* T_i)+ c \theta (\mathcal{H}^2(\partial^* T_i))^{1/2}  \sqrt{r}({\rm diam}_1(T_i))^{1/2},
\end{align} 
where in the last step we suitably enlarged the absolute  constant $c$ and \EEE also used $\theta^m \le \theta$ for $m \ge 1$. By using \eqref{eq: decompo} and the coarea formula we get
\begin{align*}
\sum_{j \ge 1} \mathcal{H}^2(\partial^* U^i_j \setminus \partial^* T_i) &\le 2\sum_{j \ge 1} \mathcal{H}^2(\lbrace f= z_j \rbrace \cap T_i) \le \frac{2}{w_1} \int_{0}^{\infty}  \mathcal{H}^2(\lbrace f= z \rbrace \cap T_i) \, \mathrm{d}z =  \frac{2}{w_1}\mathcal{L}^3(T_i)\,,
\end{align*}
 since $|\nabla f|=1$ a.e. in $\mathbb{R}^3$.  Then the isoperimetric inequality in dimension three applied on the set $T_i$ yields
\begin{align}\label{e: othersetsbound}
\sum\nolimits_{j \ge 1} \mathcal{H}^2(\partial^* U^i_j \setminus \partial^* T_i) \le   \frac{2c_{\pi,3} }{w_1 } \big(\mathcal{H}^2(\partial^* T_i) \big)^{3/2} \le 2c_{\pi,3} \theta \mathcal{H}^2(\partial^* T_i), 
\end{align} 
where $c_{\pi,3}$ denotes the isoperimetric constant in dimension three. Here, in the last step we used the definition $w_1 = \theta^{-1} (\mathcal{H}^2(\partial^* T_i))^{1/2} $.

\noindent \emph{Step 4 (Conclusion):} We are now in a position to define the covering of $E$ and to confirm \eqref{eq: lengthi bound}--\eqref{eq: essinf}. Define $R = \bigcup_{i=1}^I U^i_0$  and let $(D_j)_{j=1}^{J}$ be the partition of $E \setminus R$ consisting of the sets $\lbrace U^i_j: \ i=1,\ldots,I, \, j \ge 1 \rbrace$ constructed in Step 3. Then \eqref{eq: essinf} follows directly from \eqref{eq: dist-max-min}--\eqref{eq: dist-max-min2}. To see  \eqref{eq: lengthi bound}, we first recall $\sum\nolimits_{i=1}^I\mathcal{H}^2(\partial^* T_i) \le c\mathcal{H}^2(\partial^* E)$ by \eqref{eq: added length5} and that $\sum\nolimits_{i=1}^I {\rm diam}_1(T_i) = {\rm diam}_1(E)$. We compute by \eqref{eq: U0 length} and H\"older's inequality
\begin{align*}
\mathcal{H}^2(\partial^* R) & \le \sum\nolimits_{i=1}^I \mathcal{H}^2(\partial^* U_0^i) \le c\theta  \sum\nolimits_{i=1}^I \mathcal{H}^2(\partial^* T_i) + c\theta \sqrt{r} \sum\nolimits_{i=1}^I (\mathcal{H}^2(\partial^* T_i))^{1/2}({\rm diam}_1(T_i))^{1/2} \notag \\
& \le c\theta \mathcal{H}^2(\partial^* E) + c\theta \sqrt{r}  ({\rm diam}_1(E))^{1/2} \big(  \mathcal{H}^2(\partial^* E) \big)^{1/2} \le c\theta  \mathcal{H}^2(\partial^* E),
\end{align*} 
 where the last step follows from the definition of $r$ in \eqref{eq: r def}.
In a similar fashion, by using  \eqref{eq: added length5} and  \eqref{e: othersetsbound} we get 
\begin{align*}
\sum\nolimits_{j=1}^J\mathcal{H}^2(\partial^* D_j \setminus \partial^* E) &\le  \sum\nolimits_{i=1}^I \sum\nolimits_{j \ge 1}\mathcal{H}^2\big(\partial^* U^i_j \setminus \partial^* T_i \big) + \sum\nolimits_{i=1}^I   \mathcal{H}^2\big( \partial^* T_i \setminus \partial^* E \big) \\
&  \le c\theta \sum\nolimits_{i=1}^I \mathcal{H}^2(\partial^* T_i)  +   c\theta \mathcal{H}^2(\partial^* E) \le c\theta \mathcal{H}^2(\partial^* E). 
\end{align*}     
The previous two estimates show \eqref{eq: lengthi bound} and conclude the proof.  (Clearly, the constant $c$ can be absorbed in $\theta$ by repeating the above arguments for $\theta/c$ in place of $\theta$.) 
\end{proof}


 \typeout{References}

\end{document}